\documentclass[11pt,leqno]{amsart}

\usepackage{pdfpages}
\usepackage{amsthm}
\usepackage{mathtools}
\usepackage{amsmath,amssymb}
\usepackage{microtype}
\usepackage{booktabs}
\usepackage{xspace}
\usepackage{dsfont}
\usepackage{appendix}
\usepackage{pgfplots}
\usepackage{cancel}
\usepackage{mathrsfs}
\usetikzlibrary{shapes,snakes}
\usetikzlibrary{arrows}
\usepackage{caption}
\usepackage{subcaption}
\usetikzlibrary{decorations.pathmorphing}
\pgfmathsetseed{42}
\usetikzlibrary{decorations.markings}
\usetikzlibrary{arrows.meta,bending}
\usetikzlibrary{patterns}
\usepackage{xcolor}			

\usepackage{pst-node}
\usepackage{tikz-cd}
\usepackage{esint}
\usepackage[]{stackengine}
\usepackage{listings,lstautogobble}
\usepackage{minibox}
\usepackage{pstricks}

\usepackage{float}
\usepackage{mathabx}

\newlength\correct
\makeatletter
\newcommand*\bigcdot{\mathpalette\bigcdot@{.5}}
\newcommand*\bigcdot@[2]{\mathbin{\vcenter{\hbox{\scalebox{#2}{$\m@th#1\bullet$}}}}}

\definecolor{darkgreen}{rgb}{-0.5,0.65,0.05} 
\definecolor{lightblue}{rgb}{0,0.35,1}
\definecolor{royalblue}{rgb}{0.284,0.464,1}
\usepackage[left=2cm,right=2cm,top=4cm,bottom=4cm]{geometry}
\usepackage{epsfig}
\usepackage{mathrsfs}
\usepackage{enumerate}
\usepackage{cancel}
\usepackage{graphicx}
\usepackage{hyperref}
\hypersetup{colorlinks,%
	citecolor=cyan,%
	filecolor=red,%
	linkcolor=red,%
	urlcolor=red}

\newtheorem{theorem}{Theorem}
\newtheorem{lemma}[theorem]{Lemma}
\newtheorem{proposition}[theorem]{Proposition}

\newtheorem{remark}[theorem]{Remark}
 
\newtheorem{assumptions}[theorem]{Assumptions}

\newcounter{hypo}

\def\C{{\mathbb C}}

\def\N{{\mathbb N}} 
\def\R{{\mathbb R}} 

\def\Z{{\mathbb Z}}

\def\bkappa{\boldsymbol{\kappa}}

\def\brho{\boldsymbol{\rho}_{\pm}}
\def\bP{\boldsymbol{P}}
\def\mtp{\mathfrak{p}}
\def\mq{\mathfrak{q}}

\def\bvarphi{\bar{\varphi}}

\def\bk{\boldsymbol{k}}

\def\bK{\boldsymbol{K}}
\def\bP{\boldsymbol{P}}
\def\bQ{\boldsymbol{Q}}

\def\CC{\mathcal {C}}

\def\CB{\mathcal {B}}

\def\CF{\mathcal {F}}

\def\CH{\mathcal {H}}

\def\CM{\mathcal {M}}

\def\CO{\mathcal {O}}

\def\CX{\mathcal {X}}

\def\p{\partial}

\def\ker{\mathop{\rm Ker}\nolimits}

\newcommand\CBox{\cancel{\!\!\!\!\!\!\qed}_{g}}

\author{Nicolas Besset}\thanks{Université Paris Saclay, Département de Mathématiques d’Orsay, F-91405 Orsay Cedex, France. E--mail: nicolas.besset@universite-paris-saclay.fr}
\title[]{Parametrix construction and numerical approximation of resonances of azimuthal harmonics of the charged Klein-Gordon operator on general cosmological slowly accelerating and rotating charged black hole type spacetimes}

\makeatother

\subjclass[2000]{}
\begin{document}
	\maketitle
	\begin{abstract}
		\centering We show the index 0 Fredholm property of the spectral family of the azimuthal harmonics of the charged Klein-Gordon operator on general cosmological slowly accelerating and rotating charged black hole type spacetimes, including the De Sitter-Kerr-Newman family, using a parametrix construction for abstract totally characteristic operators. We then present a numerical scheme to compute the meromorphic poles of the inverse of the spectral family and provide an explicit estimate of the numerical error.
	\end{abstract}




	\section{Introduction}
	\label{Introduction}
	
	
	\subsection{Resonances and decay of the local energy}
	\label{Resonances and decay of the local energy}
	Long time behaviour of waves in curved spacetimes is an important first step in the analysis of stability of the underlying metric (\emph{cf.} \cite{HV} for the non-linear stability of De Sitter-Kerr metric and \cite{HHV} for the linear stability of Kerr spacetime). It is also a prerequisite in proving asymptotic completeness for wave type equations (\emph{cf.} \cite{HN} for massless Dirac fields on Kerr metric, \cite{GGH} for scalar waves on De Sitter-Kerr type metrics and \cite{B2} for charged Klein-Gordon (KG) fields on the De Sitter-Reissner-Nordstr\"{o}m metric when the charged product is sufficiently small).
	
	Similarly to solutions of the wave equation on a compact spatial domain expend as a sum over the (discrete set of) eigenvalues of the spatial Laplacian, we know since the work of Bony-H\"{a}fner \cite{BoHa} that scalar waves on cosmological uncharged spherical black hole metric with compactly supported inital data also have a local expansion of the form
	\begin{align}
	\label{Expansion res}
	\chi u(t,\bullet)&=\sum_{\substack{{\sigma_j\in\mathrm{Res}}\\{\mathrm{Im}(\sigma_j)>-C}}}\sum_{k=0}^{m(\sigma_j)}t^k\mathrm{e}^{-\mathrm{i}\sigma_{j}t}u_{j,k}+\widetilde{u},\qquad\|\widetilde{u}(t,\bullet)\|\lesssim\mathrm{e}^{-Ct}.
	\end{align}
	Above $t>0$ is the time variable, $C>0$ and $\chi$ is a smooth spatial cut-off; the equation holds in some energy space which the norm $\|\bullet\|$ is associated to. The special frequencies appearing in the expansion are called \textbf{quasinormal modes}: they are complex frequencies, proper to the black hole and the ambient spacetime, that describe late time oscillation and damping of waves. They are not the eigenvalues of the spatial part of the wave operator as the associated eigenvectors fail in general to be in the appropriate $L^2$ space (whence the appellation \emph{quasi}normal modes).
	
	Theory of resonances provides a spectral definition of these quasinormal modes (and then gives a precise description of the long time behaviour of wave type fields) when the resolvent of the operator describing the underlying dynamics has sufficiently nice decay at high frequencies. \textbf{Resonances} are defined as those frequencies $\sigma\in\C$ such that a related operator $P_\sigma$ (the spectral family, see Sect. \ref{Spectral family} below) is invertible between appropriate functional spaces. Thus resonances and quasinormal modes coincide provided $\|P_\sigma^{-1}\|$ (when it exists) decays sufficiently fast as $\mathrm{Re}(\sigma)\to\pm\infty$, as we can then shift contour integration and use the residue theorem to obtain the expansion \eqref{Expansion res} with resonances (for a very nice and complete presentation of the mathematical theory of resonances, we refer the reader to \cite{DZ}).	Non trivial elements in the kernel of $P_\sigma$ are called \textbf{resonant states} and those of the dual operator $P_\sigma^*$ are called \textbf{dual states}; the term $u_{jk}$ in \eqref{Expansion res} is the projection on the $(k+1)$ dimensional vector space spanned by the resonant states associated to the resonance $\sigma_j$ with $k\leq m(\sigma_j)$, $m(\sigma_j)$ being the multiplicity of the resonance $\sigma_j$ (see \cite[eq. (1.10) \& eq. (1.11)]{BoHa} for details for the wave equation on the De Sitter-Schwarzschild metric). Resonance expansions have first been obtained for (massive) waves on De Sitter-Kerr spacetimes by Dyatlov \cite{D}; we refer to \cite[Sect. 1.1]{H3} for more references in this context; see also \cite{B1} for charged KG fields on De Sitter-Reissner-Nordstr\"{o}m spacetimes. Notice from \eqref{Expansion res} that the presence of resonances in the upper half complex plane entails exponential growth in time of the solution (this is the case for small charge $q$ and mass $m$ of the field with $m\lesssim|q|$, and small rotation of the black hole, see \cite{BeHa20}).
	
	Existence of resonances for (massive) scalar waves in perturbations of De Sitter-Kerr spacetimes has been proved by the seminal work of Vasy \cite{V} who cast the theory of resonances in the robust framework of analytic Fredholm theory for non elliptic operators. Notice however that the operator $P_\sigma$ in \cite{V} is defined on an enlarged spatial domain and a complex potential supported beyond the horizons is added, so that resonances are created beyond the horizons; however, in a cut-off sense, that is in the physical outer domain of the spacetime, they do not depend on the choice of the extended domain nor on the absorbing potential (as this potential is microlocally supported above sets lying beyond the horizons of the spacetime). By the index 0 analytic Fredholm theory, they are thus characterized as the poles of the meromorphic extension from $\{\sigma\in\C\ \vert\ \mathrm{Im}(\sigma)\gg0\}$ to some larger domain in $\C$ of the cut-off resolvent $\chi P_\sigma^{-1}\chi$ (here $\chi$ is a radial cut-off). This definition coincide with that provided by Mazzeo-Melrose work \cite{MM} (used by S\'a Barreto-Zworski \cite{SZ} then Bony-H\"{a}fner \cite{BoHa} to define resonances).
	
	In this paper, we propose to define resonances of the azimuthal harmonics of charged KG fields on general cosmological spacetimes, including the De Sitter-Kerr-Newman family, using a parametrix construction. We will not deal with semiclassical asymptotics in very high energy regimes. As in \cite{V}, resonances are the poles of the meromorphic extension of the inverse of the spectral family associated to the charged KG operator, the meromorphic extension being constructed by analytic Fredholm theory. We do not only extend \cite[Thm. 1.1]{V} (without the trapping statement) to more general settings, but we also make numerical computations of resonances possible; on the other hand, we only consider harmonics of the charged KG operators and work on Sobolev spaces with integer orders (in particular, we do not discuss how and why the order $1/2$ is critical for the Fredholm theory, we refer the reader to \cite{V}). The parametrix construction and the underlying analytic Fredholm theory are presented in an abstract setting in Sect. \ref{General boundary parametrix construction} below. As for numerical approximation of resonances, there exists several methods in the litterature (we can mention \emph{e.g.} \cite{CCDHJ} on De Sitter-Reissner-Nordstr\"{o}m spacetimes and, more recently, \cite{HX1} on De Sitter and \cite{HX2} on De Sitter-Schwarzschild spacetimes); to our knowledge, the approach in the present paper provides for the very first time a scheme with an explicit estimate of the numerical error for a very large class of cosmological spacetimes. We will only present this scheme and derive the numerical error in a purely theoretical point of view; we plan to explicitly apply it in a future work.

	
	\subsection{Main results}
	\label{Main results}
	We consider a cosmological, accelerating and rotating charged black hole spacetime solution to the Einstein-Maxwell equation, \emph{cf.} Sect. \ref{The metric in Boyer-Lindquist coordinates}. Let $(\bP_\sigma,\CX^k)$, $k\in\N\setminus\{0\}$, be the harmonic\footnote{That is, we consider restriction of $\bP_\sigma$ to $\ker(D_\varphi+\ell)$ where $\varphi$ is the azimuth and $\ell\in\Z$.} spectral family associated to the charged KG operator realized on is natural domain $\CX^k:=\{u\in H^k\ \vert\ \bP_\sigma u\in H^{k-1}\}$ ($H^k$ are standard Sobolev spaces), \emph{cf.} Sect. \ref{Spectral family}. Under the Assumptions \ref{Assumptions metric g}, the following result holds (\emph{cf.} Sect. \ref{Proof of Thm. Thm Fredholm index 0 cKG op DSKN} for the proof)
	\begin{theorem}
		\label{Thm Fredholm index 0 cKG op DSKN}
		For all $k\in\N\setminus\{0\}$, let
		\begin{align*}
		\C_{k}&:=\big\{\sigma\in\C\ \big\vert\ \mathrm{Im}(\sigma)\notin\{-(k-1/2)\bkappa_{-},-(k-1/2)|\bkappa_{+}|\}\big\}
		\end{align*}
		where $\bkappa_{\pm}:=\frac{(\p_r\mu)(r_{\pm})}{2(1+\lambda)(r_\pm^2+a^2)}$. Then $\C_k\ni\sigma\mapsto(\bP_\sigma,\CX^k)$ is a analytic family of index 0 Fredholm operators. For non rotating black holes, this results holds true for the full operator (\emph{i.e.} we do not need any harmonical restriction).
	\end{theorem}
	By analytic Fredholm theory (\emph{cf.} \emph{e.g.} \cite[Thm. C.8]{DZ}), $\C_k\ni\sigma\mapsto(\bP_\sigma^{-1},\CB(H^{k-1},\CX^k))$ is a meromorphic family provided that $\bP_{\sigma_0}^{-1}$ exists for some $\sigma_0\in\C_k$ (in the De Sitter-Kerr setting, Vasy proved this last point from a semiclassical estimate \emph{cf.} \cite[Sect. 6.4]{V}). The poles of $\bP_\sigma^{-1}$ are by definition the resonances (of $\bP_\sigma$).
	\begin{remark}
		\label{Rmk Indicial Family}
		The critical strip $\R-\mathrm{i}(k-1/2)|\bkappa_\pm|$ is associated to the indicial family defined in \eqref{Eq Indicial family 1} near the corresponding horizon ($\rho$ there has to be understood as $|r-r_\pm|$, with '$-$' near the event horizon and '$+$' near the cosmological horizon), and coincides of course to that found in \cite[Thm. 1.1]{V}; whenever $\sigma$ lies on that strip, the indicial family has a real root and the Fredholm property is not ensured anymore. We refer to Sect. \ref{Decay in the basis} for details.
	\end{remark}
	\begin{remark}
		\label{Remark Lambda=0}
		When the cosmological constant $\Lambda$ is $0$, the manifold has a hyperbolic and an Euclidean ends. There is then an accumulation of resonances at the 0 energy (\emph{cf.} \cite[Thm. 1.1]{H3}) so that 0 is not itself a resonance and the resolvent is not meromorphic in any neighborhood of 0. In this situation, only polynomial decay of waves is expected (\emph{cf.} \emph{e.g.} \cite{H2} for a sharp rate of decay on Kerr metric).
		
		While there is seemingly no continuity in the limit $\Lambda\to0$ in the expansion \eqref{Expansion res}, we can take this limit on the spectral side in Thm \ref{Thm Fredholm index 0 cKG op DSKN}: in the case of a non accelerating black hole with sufficiently small charge and angular momentum, we can check that $r_{-}\to M+\sqrt{M^2-(Q^2+a^2)}$ and $r_{+}=\CO(\Lambda^{-1/2})$ as $\Lambda\to0$, so that $\bkappa_{-}\to\frac{\sqrt{M^2-(Q^2+a^2)}}{2M(M+\sqrt{M^2-(Q^2+a^2)})-Q^2}>0$ while $\bkappa_{+}=\CO(\Lambda^{1/2})\to0$. As a result, the highest critical strip in the Fredholm theory is simply $\R$; it is then no longer possible to use a simple contour deformation to get a long time behaviour for solutions to the charged KG equation as in \cite{BoHa} or \cite{V} (but we can use a limiting absorption principle as in \cite{H2}).
	\end{remark}
	\begin{remark}
	\label{Rmk critical strip Regge Wheeler coordinate}
		In the first version of this paper \cite{B3}, we considered other realizations of $\bP_\sigma$ on De Sitter-Reissner-Nordstr\"{o}m metrics. It turned out that the index 0 Fredholm property holds for $\mathrm{Im}(\sigma)>-\min\{\bkappa_{-},|\bkappa_{+}|\}$ independently of the value of $k$. This rigid restriction also applies when using the Regge-Wheeler, or tortoise, coordinate (\emph{cf.} \cite[Thm. 3.8]{B1}) -- using this coordinate pushes the horizons (the boundaries) at infinity and we do not really see them. Thus, it is crucial to take into account the boundaries of the spacetime in the analysis as we do in the present paper or as Vasy did in \cite{V} in order to lower the critical strips as we increase Sobolev regularity.
	\end{remark}
	As explained in Sect. \ref{Resonances and decay of the local energy} above, the proof of Thm. \ref{Thm Fredholm index 0 cKG op DSKN} relies on a parametrix construction. The main point using such an approach is that it allows numerical computations of resonances. This is the point of our second main result (see Sect. \ref{Numerical approximation of resonances} for the proof):
	\begin{theorem}
		\label{Thm error estimate for numerical scheme}
		We assume that $\bP_{\sigma_0}^{-1}$ exists for some $\sigma_0\in\C_k$ (it is the case if all the parameters of the problem (charges, mass of the field, angular momentum and acceleration of the black hole) are sufficiently small in order to enter the framework of small perturbations of De Sitter-Kerr spacetimes treated in \cite{V}). For all $k,N'\in\N\setminus\{0\}$, let
		\begin{align*}
		\C_{k,N'}&:=\big\{\sigma\in\C\ \big\vert\ \mathrm{Im}(\sigma)\notin\{-(k-1/2+j-1)\bkappa_{-},-(k-1/2+j-1)|\bkappa_{+}|\},\ \forall j\in\{1,\ldots,N'\}\big\}\subset\C_k.
		\end{align*}
		There exists an analytic function $\C_{k,N'}\ni\sigma\mapsto D(\sigma)\in\C$ whose zeros are exactly the poles (counted with multiplicity) of $(\bP_\sigma^{-1},\CB(H^{k-1},\CX^k))$.
		
		Furthermore, for all familly of projectors $(\Pi_R)_{R\in\N\setminus\{0\}}$ acting on $H^{k-1}$, there exists a family of analytic functions $\C_{k,N'}\ni\sigma\mapsto(D_R(\sigma))_{R\in\N\setminus\{0\}}$ (which are the determinants of matrices acting on the range of the $\Pi_R$) with the following property: for all positively oriented contour $\Gamma\in\C_{k,N'}$, let $N(\Gamma)$ and $S(\Gamma,n)$ ($n\in\N\setminus\{0\}$) be respectively the number of zeros of $D$ lying inside $\Gamma$ and the sum of the $n$-th power of these zeros (counted with their multiplicity), then define similarly $N_R(\Gamma)$ and $S_R(\Gamma,n)$; then
		\begin{align*}
		|N(\Gamma)-N_R(\Gamma)|&\leq C_R(\Gamma),\qquad |S(\Gamma,n)-S_R(\Gamma,n)|\leq\widetilde{C}_R(\Gamma,n)
		\end{align*}
		where $C_R(\Gamma),\widetilde{C}_R(\Gamma,n)$ are explicitly given in \eqref{Eq Constants Numerics} and decay as $\CO(\|\mathds{1}_{H^{k-1}}-\Pi_R\|_{\CB(H^{k-1})})$ as $R\to+\infty$.
	\end{theorem}
	\begin{remark}
		\label{Rmk 1 Thm error numerics}
		Since $N(\Gamma)\in\N$, the \emph{exact} number of resonances (counted with their multiplicity) inside $\Gamma$ is given by $N_R(\Gamma)$ when $R\gg0$ so that $C_R(\Gamma)<1$. Besides, if we do know that only one resonance lies inside $\Gamma$, then $S_R(\Gamma,1)$ gives an approximative value of that resonance.
	\end{remark}
	\begin{remark}
		\label{Rmk 2 Thm error numerics}
		The farther below the real line we want to seek resonances, the larger we have to take $k\in\N\setminus\{0\}$: this increases the numerical error as it uses the Sobolev norm $\|\bullet\|_{H^{k-1}}$ (for simple spherically symmetric problems, we expect that resonances associated to large values of the angular momentum of the KG field are localized far below the real line, see \emph{e.g.} \cite[Thm. 2.1]{BoHa}, \cite[Thm. 5.1]{B1} and \cite[Fig. 1.7]{DZ}). Increasing $N'\in\N\setminus\{0\}$ offers better rate of convergence of the scheme at the cost of adding much more terms in the parametrix.
		
		In practice, we first look for resonances starting with very simple black holes (\emph{e.g.} De Sitter-Reissner-Nordstr\"{o}m black holes) for which we do know from \cite{BoHa} that all the resonances for the wave equation lie below $\R$ except one which is at 0, and from \cite{B1} and \cite{BeHa20} that that resonance 0 for waves is expelled from 0 for charged KG fields and moves above or below the real axis depending on the respective size of the charge and the mass of the field (when the mass controls the charge then the resonance moves below and \emph{vice versa}). Then, we increase the perturbative parameters (such as the angular momentum of the black hole or the product of the black hole charge with the KG field charge) to guess where the resonances move (by analyticity of the spectral family in these parameters, the path followed by the resonances under these perturbations is smooth).
		
		As we expect from the expansion \eqref{Expansion res} that resonances are multiple of the (tiny) characteristic length $\sqrt{\Lambda}$ (which is of order $10^{-26}$ according to Planck's latest measure \emph{cf.} \cite{PC}), we can use the following rescaling (see Sect. \ref{The metric in Boyer-Lindquist coordinates} for the notations):
		\begin{align*}
		\tilde{t}:=\sqrt{\Lambda}t,\qquad &\tilde{r}:=\sqrt{\Lambda}r,\qquad \tilde{M}:=\sqrt{\Lambda}M,\qquad \tilde{Q}:=\sqrt{\Lambda}Q,\qquad \tilde{\Lambda}:=1,\\
		\tilde{a}:=\sqrt{\Lambda}a,\qquad &\tilde{\alpha}:=\frac{\alpha}{\sqrt{\Lambda}},\qquad \tilde{q}:=\frac{q}{\sqrt{\Lambda}},\qquad \tilde{m}:=\frac{m}{\sqrt{\Lambda}},\qquad\tilde{\sigma}:=\frac{\sigma}{\sqrt{\Lambda}}.
		\end{align*}
		In practice, this amounts to "zooming" on resonances and changing the characteristic distance from $\sqrt{\Lambda}$ to $1$. Observe that the charge product $\tilde{q}\tilde{Q}=qQ$ is \emph{invariant under the rescaling}.
	\end{remark}
	\begin{remark}
		\label{Rmk 3 Thm error numerics}
		It is possible to define resonances as the zeros of some Wronskian $W$ for numerical approximation (see \cite[Sect. 5.4]{Bthesis}; note that this approach uses the Regge-Wheeler coordinate $x\in\R$ so that only resonances above the upper critical strip $\R-\mathrm{i}\min\{\bkappa_{-},|\bkappa_{+}|\}$ can be approximated\footnote{Moreover, the error estimate of the numerical approximation worsens below $\R$ in the following sense: if $h>0$ is the numerical step (that is the scheme converges as $h\to0$), then the error of approximation of a resonance $\sigma$ lying below $\R$ is proportional to $h^{1-|\mathrm{Im}(\sigma)|/\kappa}$ where $\kappa>0$ is such that the upper critical strip is $\{z\in\C\ \vert\ \mathrm{Im}(z)=-\mathrm{i}\kappa\}$.}, accordingly to Rmk. \ref{Rmk critical strip Regge Wheeler coordinate}). This Wronskian involves Jost solutions $e_\pm$ (\emph{cf.} \cite[Sect. 3.3]{B1}) which are outgoing solutions to $\bP_\sigma u=0$; recall that outgoing means that $e_\pm(x)\sim\mathrm{e}^{\pm\mathrm{i}\sigma x}$ as $x\to\pm\infty$, that is, they behave as free linear waves near the appropriate horizon. Thus, above the upper critical strip, $W$ and $D$ have the same zeros (counted with their multiplicity); for one dimensional Schr\"{o}dinger type operator with suitable potential (\emph{i.e.} when $\bP_\sigma=-\p_x^2+V(x)$), \cite[Prop. 5.7]{S} shows that, in fact, $W=D$.
	\end{remark}
	%
	
	
	\subsection{Outline of the method}
	\label{Outline of the method}
	While some parts in this paper use tedious computations, the overall method is simple.
	
	In the quite general setting of Sect. \ref{The metric in Boyer-Lindquist coordinates}, the restriction to azimuthal harmonics of the charged KG operator is \textbf{totally characteristic}. An operator is said to be totally characteristic in some domain if its Fourier principal symbol vanishes at the boundary of that domain -- in other words, the operator is elliptic inside the domain but its characteristic set is not empty on the boundary. The lack of ellipticity prevents us from using the standard construction of parametrices; however, in the Mellin quantization sense, the operator becomes uniformly elliptic and an explicit (thus workable for numerics) construction is now possible. This construction amounts to first produce a decay in the fiber direction, then in the basis direction of the cotangent bundle in order to obtain compact remainder.
	
	We then deduce the Fredholm property of the spectral family $\bP_\sigma$, $\sigma\in\C$, associated to the (harmonic) charged KG operator outside critical strips of frequencies in $\C$; this restriction comes from the existence of real roots of indicial polynomials associated to $\bP_\sigma$. This amounts to working separately near each radial boundary of the spacetime in a Mellin setting, construct local parametrices there, then glue them together. As a rough summary:
	\begin{align*}
	\text{Fourier ellipticity}\qquad\implies\qquad\text{Fredholm property},\\
	\text{Mellin ellipticity }+\text{ restriction outside critical strips}\qquad\implies\qquad\text{Fredholm property}.
	\end{align*}
	In the last case, we will see that we need to lose one derivative with respect to the standard Fourier ellipticity; another noticeable difference is the role of the \textbf{subprincipal part} of the spectral family in the frequency restriction (\emph{cf.} Rmk. \ref{Rmk Roots Indicial Family} below where $b(0)$ belongs to the subprincipal part of $\bP_\sigma$).
	
	Having established the equation $\bP_\sigma\bQ_\sigma=\mathds{1}+\bK_\sigma$ for some suitable $\bQ_\sigma$ and $\bK_\sigma$, we can push the spectral analysis of the latter further and show its trace class property. As a result, the \textbf{Fredholm determinant} $\det(\mathds{1}+\bK_\sigma)$ is well-defined and can be approximated by $\det(\Pi_R(\mathds{1}+\bK_\sigma))$ where $\Pi_R$ projects onto a $R$ dimensional linear space; furthermore, the determinant cancels if and only if $\bK_\sigma$ has a non trivial kernel\footnote{The power of the Fredholm theory is to allow one to apply the standard linear algebra results in finite dimension to problems in infinite dimension.}. If we manage to show that $\bP_\sigma$ has index 0 (that is both the dimensions of its kernel and cokernel coincide), then
	\begin{align*}
	\sigma\text{ is a resonance}\ \qquad\Longleftrightarrow\qquad\det(\mathds{1}+\bK_\sigma)=0.
	\end{align*}
	As $\sigma\mapsto\det(\mathds{1}+\bK_\sigma)$ is analytic for $\sigma$ outside the critical strips, we can finally use complex integrals of $\det(\Pi_R(\mathds{1}+\bK_\sigma))$ (which is the determinant of a matrix) to approximate or count the number of resonances inside a fixed domain in $\C$.

	
	\subsection{Quantizations}
	\label{Quantizations}
	We introduce in this section the two types of quantization that we will use in this paper.
	
	The Fourier transform $\CF\in\CB(L^2(\R,\mathrm{d}x))$ and its inverse are defined by:
	\begin{align*}
	\CF[u](\xi)&:=\int_{\R}\mathrm{e}^{-\mathrm{i}x\xi}u(x)\mathrm{d}x,\qquad\qquad\CF^{-1}[v](x)=\frac{1}{2\pi}\int_{\R}\mathrm{e}^{\mathrm{i}x\xi}v(\xi)\mathrm{d}\xi.
	\end{align*}
	The Mellin transform $\CM\in\CB\left(L^2\left((0,+\infty),\frac{\mathrm{d}\rho}{\rho}\right),L^2(\R,\mathrm{d}\xi)\right)$ and its inverse are defined by:
	\begin{align*}
	\CM[u](\xi)&:=\int_{\R_+}\rho^{-\mathrm{i}\xi}u(\rho)\frac{\mathrm{d}\rho}{\rho},\qquad\qquad\CM^{-1}[v](\rho)=\frac{1}{2\pi}\int_\R\rho^{\mathrm{i}\xi}v(\xi)\mathrm{d}\xi.
	\end{align*}
	The change of variable $\rho:=\mathrm{e}^x$, $x\in\R$, allows us to transpose properties of Fourier type quantization to Mellin type quantization\footnote{In particular, that $\CM\in\CB\left(L^2\left((0,+\infty),\frac{\mathrm{d}\rho}{\rho}\right),L^2(\R,\mathrm{d}\xi)\right)$ is simply a consequence of Plancherel formula. We should by the way talk respectively about the Fourier-Plancherel and Mellin-Plancherel extensions of the standard Fourier and Mellin transforms.}. Setting $v=u\circ\exp$, we have $\CM[u]=\CF[v]$ and $(\rho\p_\rho u)(\rho)=(\p_xv)(x)$\footnote{From the spectral perspective, Mellin type quantization associates $\rho D_\rho$ to $\xi$, \emph{i.e.} totally characteristic operators near a boundary are in fact Mellin elliptic.}; in particular,
	\begin{align*}
	u\in L^2\left((0,+\infty),\frac{\mathrm{d}\rho}{\rho}\right)\qquad&\Longleftrightarrow\qquad v\in L^2(\R,\mathrm{d}x).
	\end{align*}
	Note however that this equivalence fails to hold true between Sobolev spaces, as the above change of variable modifies the differential structures (in short, $\p_{\rho}$ is not equivalent to $\p_x$). Thus, when dealing with Mellin type pseudo-differential operators $H^{k'}\left((0,+\infty),\frac{\mathrm{d}\rho}{\rho}\right)\to H^k\left((0,+\infty),\frac{\mathrm{d}\rho}{\rho}\right)$ with $k,k'\in\N$ such that $k'\geq k$, we will first take $(k'-k)$ derivatives (with respect to the variable $\rho\in(0,+\infty)$) then we will change the variable to work with a Fourier type pseudo-differential operator acting on $L^2(\R,\mathrm{d}x)$ (this will be used in the proof of Lem. \ref{Lem Compactness K} below).
	
	Let $a\in\CC^\infty((0,+\infty)_\rho\times\R^n_\omega,\C)$. The mixed Mellin-Fourier quantization of $a$ is by definition the following operator:
	\begin{align*}
	(\mathrm{Op}_{\CM,\CF}[a]u)(\rho,\omega)&:=\frac{1}{(2\pi)^{n+1}}\int_{\R_\xi}\int_{\R^n_\eta}\rho^{\mathrm{i}\xi}\mathrm{e}^{\mathrm{i}\omega\cdot\eta}a(\rho,\omega,\xi,\eta)\left(\int_{(0,+\infty)_y}\int_{\R^n_\zeta}y^{-\mathrm{i}\xi}\mathrm{e}^{-\mathrm{i}\zeta\cdot\eta}u(y,\zeta)\mathrm{d}\zeta\frac{\mathrm{d}y}{y}\right)\mathrm{d}\eta\mathrm{d}\xi.
	\end{align*}
	As explained above, we can see this operator as a full Fourier type quantization on $\R_x\times\R^n_\omega$. We refer the reader to \emph{e.g.} \cite[Chap. XVIII]{Ho} for the basic properties (depending on $a$) of these pseudo-differential operators acting on Sobolev spaces.
	
	
	\subsection{Plan of the paper and notations}
	\label{Plan of the paper and notations}
	The paper is organized as follows: Sect. \ref{General boundary parametrix construction} describes the construction in an abstract setting of a parametrix for radially totally characteristic operators and establishes the associated local Fredholm theory, Sect. \ref{The charged KG operator near cosmological accelerating and rotating charged black holes} presents and puts the charged KG operator defined in some cosmological black hole type spacetimes into the general setting of the previous section and Sect. \ref{Numerical scheme} presents a numerical scheme to compute resonances. We give in App. \ref{App computation P_sigma} a detailed derivation of the spectral family of Sect. \ref{Spectral family}. in coordinates defined near the poles of $\mathbb{S}^2$.
	
	Throughout the paper, we will use the following notations: $[\bullet\,,\bullet]_{+}$ is the positive commutator ($[a,b]_{+}:=ab+ba$), $\CC^\omega(X,\C)$ denotes the set of all the analytic functions defined on a vector space $X$, $a\lesssim b$ means that there exists a universal constant $C>0$ such that $a\leq Cb$, and $D_\bullet:=-\mathrm{i}\p_\bullet$.

	%
	%
	%
	%

	%
	%
	\section{General boundary parametrix construction}
	\label{General boundary parametrix construction}
	We construct in this section local (\emph{i.e.} near a boundary) parametrices for abstract totally characteristic operators in the case where the domain has a radial symmetry, using Mellin type quantization in the radial direction.

	\subsection{Definition of the functional framework and main results of the section}
	\label{Definition of the functional framework and main results of the section}
	Let $(k,l,\sigma)\in\mathbb{N}\times\R\times\C$ and let $\CM$ be a $n$ dimensional compact smooth manifold. Let $(\phi_\bullet)_{\bullet}$ be a smooth partition of unity of $\CM$ then choose $\widetilde{\phi}_\bullet\in\CC^\infty(\CM,[0,1])$ such that $\widetilde{\phi}_\bullet\phi_\bullet=\phi_\bullet$ on $\mathrm{Supp}(\phi_\bullet)$; in the sequel, we will identify $\CM\cap\mathrm{Supp}(\widetilde{\phi}_\bullet)$ with an open subset of $\R^n$. Let also $\chi\in\CC^\infty([0,+\infty),[0,1])$ such that
	\begin{align}
	\label{Eq chi def}
	\chi(\rho)&=\begin{cases}
	1&\text{if $\rho\leq\rho_0$},\\
	0&\text{if $\rho\geq\rho_0'$}
	\end{cases}
	\end{align}
	for some $0<\rho_0<\rho_0'$. In applications, $\sigma$ is a spectral parameter and $\CM$ is the unit sphere $\mathbb{S}^n$.
	
	Let $X:=(0,+\infty)\times\CM$. On $\mathrm{Supp}(\chi)\times\mathrm{Supp}(\phi_\bullet)$, we consider
	\begin{align*}
	\bP_\sigma&:=\rho^{-1}P_\sigma+\cancel{P}_{\!\sigma},\\
	P_\sigma&:=a(\rho,\omega,\sigma)(\rho D_\rho)^2+b(\rho,\omega,\sigma)\rho D_\rho+c(\rho,\omega,\sigma),\qquad\cancel{P}_{\!\sigma}:=\sum_{j=1}^{n}\left(\cancel{a}_j(\rho,\omega,\sigma)D_{\omega_j}^2+\cancel{b}_j(\rho,\omega,\sigma)D_{\omega_j}\right)
	\end{align*}
	where $a,b,c,\cancel{a}_j,\cancel{b}_j\in\CC^\infty(\overline{\mathrm{Supp}(\chi)}\times\CM\times\C)$ -- these functions as well as the coordinate $\omega$ of course depend on the cut-off $\phi_\bullet$ and we should write $a_\bullet(\rho,\omega_\bullet,\sigma)$, $\cancel{a}_{j,\bullet}(\rho,\omega_\bullet,\sigma)$ and so on; we will however not emphasize this dependence in notations below in order to not overload them.
	
	We define the following weighted boundary Sobolev spaces:
	\begin{align*}
	\CH&:=\CH^{0,0}:=L^2\Big(X,\frac{\mathrm{d}\rho}{\rho}\mathrm{d}\omega\Big),\\
	\CH^{k,l}&:=\big\{u:X\to\C\ \big\vert\ \rho^{-l}u\in\CH,\ \rho^{-l}(\p_{\rho}^{k_1}\p_{\omega}^{k_2}u)\in\CH\;\text{for all }k_1+|k_2|=k\big\}.
	\end{align*}
	We equip $\CH$ with its standard scalar product and define:
	\begin{align*}
	\|u\|_{\CH^{k,l}}^2&:=\|\rho^{-l}u\|_{\CH}^2+\|\rho^{-l}(\p_{\rho}^ku)\|_{\CH}^2+\sum_{j=1}^n\|\rho^{-l}(\p_{\omega_j}^ku)\|_{\CH}^2.
	\end{align*}
	We will use the realization $(\bP_\sigma,\CX^{k+1,l})$ with
	\begin{align*}
	\CX^{k+1,l}&:=\{u\in\CH^{k+1,l}\ \vert\ \bP_\sigma u\in\CH^{k,l}\},\qquad\|u\|_{\CX^{k+1,l}}^2:=\|u\|_{\CH^{k+1,l}}^2+\|\bP_\sigma u\|_{\CH^{k,l}}^2.
	\end{align*}

	We now make some assumptions. Given $(\rho,\omega,\xi,\eta,\sigma)\in[0,+\infty)\times\CM\times\R\times\R^n\times\C$, we let
	\begin{align*}
	\xi_{k,l}&:=\xi-\mathrm{i}(k+l),\qquad|\eta|^2_{\cancel{a}}:=\sum_{j=1}^n\cancel{a}_j(\rho,\omega,\sigma)\eta_j^2.
	\end{align*}
	\begin{assumptions}
	\label{Assumptions a,b,c...}
	We assume that $k+l>0$. We also assume that for all $(\xi,\eta,\sigma)\in\mathrm{Supp}(\chi)\times\CM\times\R\times\R^n\times\C$,
	\begin{align*}
	&a(\rho,\omega,\sigma)\neq0\ \mathrm{on}\ \overline{\mathrm{Supp}(\chi)},\qquad\cancel{a}_j(\rho,\omega,\sigma)\neq0\ \mathrm{on}\ \overline{\mathrm{Supp}(\phi_\bullet)}
	\end{align*}
	and if $(\rho,\omega)\in(0,+\infty)\times\CM$, then
	\begin{align*}
	&|\cancel{a}_j(\rho,\omega)\eta_j|\lesssim\big||\eta|_{\cancel{a}}\big|,\qquad a(\rho,\omega,\sigma)\xi_{k,l}^2+\rho|\eta|_{\cancel{a}}^2\neq0,\\
	&\lim_{\|(\xi,\eta)\|\to+\infty}|a(\rho,\omega,\sigma)\xi_{k,l}^2+\rho|\eta|_{\cancel{a}}^2|=+\infty.
	\end{align*}
	%
	%
	%
	\end{assumptions}
	\begin{remark}
	\label{Rmk on Assumptions 1}
	Assumptions \ref{Assumptions a,b,c...} are verified for example if $a$ and $\cancel{a}$ are positive valued, except the condition $a(\rho,\omega,\sigma)\xi_{k,l}^2+\rho|\eta|_{\cancel{a}}^2\neq0$ which may fail to be true if $\xi=0$ and $\rho|\eta|_{\cancel{a}}^2=a(\rho,\omega,\sigma)(k+l)^2$. In that case, we can replace it by $a(\rho,\omega,\sigma)(\xi_{k,l}^2+(k+l+\epsilon)^2)+\rho|\eta|_{\cancel{a}}^2$, $\epsilon>0$, and modify accordingly the denominator of $\mq_{\sigma,k,l}^{\mathrm{f},1}$ in Sect. \ref{Decay in the fiber} below (and replace $c$ with $c-a(k+l+\epsilon)^2$) to proceed, as we will do in Sect. \ref{Verification of Assumptions Assumptions a,b,c...}.
	\end{remark}
	Under Assumptions \ref{Assumptions a,b,c...}, we will establish the following result (\emph{cf.} Sect. \ref{Global boundary parametrices} for the proof):
	\begin{proposition}
		\label{Prop Local Parametrix}
		We defined the \textbf{indicial polynomials} associated to $\bP_\sigma$:
		\begin{align*}
		\mtp_{\sigma,k,l}(\xi)&:=a(0,\omega,\sigma)\xi_{k,l}^2+(b(0,\omega,\sigma)-2\mathrm{i}a(0,\omega,\sigma))\xi_{k,l}-\mathrm{i}b(0,\omega,\sigma)-a(0,\omega,\sigma)+c(0,\omega,\sigma).
		\end{align*}
		For all $N,N'\in\N\setminus\{0\}$, let
		\begin{align}
		\label{Def C_k,l,N'}
		\C_{k,l,N'}&:=\bigg\{\sigma\in\C\ \bigg\vert\ 0\notin\bigcup_{j=1}^{N'}\mtp_{\sigma,k+j-1,l}(\R)\bigg\}.
		\end{align}
		If $N>2(N'-1)$, then there exist a right local parametrix $\bQ^{N,N'}_{\sigma,k,l}\in\CB(\CH^{,k,l},\CX^{k+1,l})$ as well as a left local parametrix $\widetilde{\bQ}^{N,N'}_{\sigma,k,l}\in\CB(\CH^{,k,l},\CX^{k+1,l})$ such that:
		\begin{align*}
		\bP_\sigma\bQ^{N,N'}_{\sigma,k,l}&=\chi\mathds{1}_{\CH^{,k,l}}+\bK^{N,N'}_{\sigma,k,l},\\
		\widetilde{\bQ}^{N,N'}_{\sigma,k,l}\widetilde{\bP}_\sigma&=\chi\mathds{1}_{\CX^{k+1,l}}+\widetilde{\bK}^{N,N'}_{\sigma,k,l}.
		\end{align*}
		The remainders $\bK^{N,N'}_{\sigma,k,l}\in\CB(\CH^{k,l})$ and $\widetilde{\bK}^{N,N'}_{\sigma,k,l}\in\CB(\CX^{k+1,l})$ are compact. If $\left(\sigma\mapsto\bP_\sigma\right)\in\CC^{p}(\C,\CB(\CX^{k+1,l},\CH^{k,l}))$ for some $p\in\N\cup\{+\infty,\omega\}$, then $\Big(\sigma\mapsto\bK^{N,N'}_{\sigma,k,l}\Big)\in\CC^{p}(\C_{k,l,N'},\CB(\CH^{k,l}))$ and $\Big(\sigma\mapsto\widetilde{\bK}^{N,N'}_{\sigma,k,l}\Big)\in\CC^{p}(\C_{k,l,N'},\CB(\CX^{k+1,l}))$. Furthermore, if $N\geq3N'-2$, then $\bK^{N,N'}_{\sigma,k,l}\in\CB(\CH^{k,l},\CH^{k+N-2(N'-1),l})$ and if $0\notin\bigcup_{0\leq\alpha\leq N'}\bigcup_{j=1}^{N'}\mtp_{\sigma,k+j-1,l}(\R-\mathrm{i}\alpha)$, then $\widetilde{\bK}^{N,N'}_{\sigma,k,l}\in\CB(\CX^{k+1+N-2(N'-1),l})$.
	\end{proposition}
	\begin{remark}
		\label{Rmk Roots Indicial Family}
		If $c(0,\omega,\sigma)=0$ then
		\begin{align*}
		\mtp_{\sigma,k+j-1,l}(\xi)&=a(0,\omega,\sigma)\xi_{k+l+j}\left(\xi_{k+l+j}+\frac{b(0,\omega,\sigma)}{a(0,\omega,\sigma)}\right).
		\end{align*}
		and $\mtp_{\sigma,k,l,j}$ has a real root if and only if $\mathrm{Im}(b(0,\omega,\sigma))\mathrm{Re}(a(0,\omega,\sigma))-\mathrm{Re}(b(0,\omega,\sigma))\mathrm{Im}(a(0,\omega,\sigma))=k+l+j$. Otherwise, notice that all the complex roots of $\mtp_{\sigma,k,l,j}$ lie in the upper half plane whenever $\mathrm{Im}(b(0,\omega,\sigma))\mathrm{Re}(a(0,\omega,\sigma))-\mathrm{Re}(b(0,\omega,\sigma))\mathrm{Im}(a(0,\omega,\sigma))\leq k+l+j$.
	\end{remark}
	Strenghtening Assumptions \ref{Assumptions a,b,c...}, we can obtain the following Fredholm property for $\bP_\sigma$ (\emph{cf.} Sect. \ref{Local index 0 Fredholm property} for the proof): 
	\begin{proposition}
		\label{Prop Invertibility Q}
		Let $(k,l,N')\in\N\times\R\times\N\setminus\{0\}$ and $X'\Subset X$ such that $\chi_{\vert X'}=\mathds{1}_{X'}$. Denote by $r_{X'}$ the restriction to $X'$ and $\CH'^{k,l},\CX'^{k+1,l}$ the functional spaces defined over $X'\times\CM$. Assume that $\pm a(\rho,\omega,\sigma)>0$ and $\pm \cancel{a}_j(\rho,\omega,\sigma)>0$ on $\overline{X'}\times\CM$ (so that Assumptions \ref{Assumptions a,b,c...} are satisfied). Then the restriction of $\C_{k,l,N'}\ni\sigma\mapsto r_{X'}\circ\bP_\sigma\in\CB(\CX'^{k+1,l},\CH'^{k,l})$ is an analytic family of index 0 Fredholm operators.
	\end{proposition}

	In proving Prop. \ref{Prop Local Parametrix}, we will use the following function:
	\begin{align*}
	F_{k,l}(\rho,\xi)&:=\frac{\rho^{k+l}}{(\mathrm{i}\xi+k+l)_{(k)}}.
	\end{align*}
	Above $(z)_{(w)}:=\frac{\Gamma(z+1)}{\Gamma(z-w+1)}$ denotes the falling factorial ($\Gamma$ is the Gamma function).
	\begin{lemma}
	\label{Lem F_{k,l}}
	Let $u\in H^{k,l}(X,\frac{\mathrm{d}\rho}{\rho}\mathrm{d}\omega)$ and put $v(\rho,\omega):=\rho^{-l}(\p_{\rho}^ku)(\rho,\omega)$. Then $\mathrm{Op}_{\CM,\CF}[F_{k,l}]v=u$.
	\end{lemma}
	\begin{proof}
	As $F_{k,l}$ is purely radial, we can directly integrate with respect to the angular variable $\omega$ and its dual variable $\eta$. Next, shifting the contour integral from $\R$ to $\R+\mathrm{i}k$ then performing $k$ integrations by parts, we compute:
	\begin{align*}
	(\mathrm{Op}_{\CM,\CF}[F_{k,l}]v)(\rho,\omega)&=\frac{1}{2\pi}\int_\R\rho^{\mathrm{i}\xi+k+l}\frac{1}{(\mathrm{i}\xi+k+l)_{(k)}}\left(\int_{(0,+\infty)}y^{-\mathrm{i}\xi}(y^{-l}\p_y^{k}u)(y,\omega)\frac{\mathrm{d}y}{y}\right)\mathrm{d}\xi\\
	&=\frac{1}{2\pi}\int_\R\rho^{\mathrm{i}\xi+l}\frac{1}{(\mathrm{i}\xi+l)_{(k)}}\left(\int_{(0,+\infty)}y^{-\mathrm{i}\xi+k}(y^{-l}\p_y^{k}u)(y,\omega)\frac{\mathrm{d}y}{y}\right)\mathrm{d}\xi\\
	&=\frac{1}{2\pi}\int_\R\rho^{\mathrm{i}\xi+l}\frac{1}{(\mathrm{i}\xi+l)_{(k)}}\left(\int_{(0,+\infty)}(-1)^{k}(\p_y^ky^{-\mathrm{i}\xi+k-l-1})u(y,\omega)\mathrm{d}y\right)\mathrm{d}\xi\\
	&=\frac{1}{2\pi}\int_\R\rho^{\mathrm{i}\xi+l}\frac{1}{(\mathrm{i}\xi+l)_{(k)}}\left(\int_{(0,+\infty)}(\mathrm{i}\xi+l)_{(k)}y^{-\mathrm{i}\xi}y^{-l}u(y,\omega)\frac{\mathrm{d}y}{y}\right)\mathrm{d}\xi\\
	&=\rho^{l}\mathrm{Op}_\CM[1]((\rho,\omega)\mapsto\rho^{-l}u(\rho,\omega))\\
	&=u(\rho,\omega).
	\end{align*}
	The contour shift uses Cauchy-Riemann equations: to check them, we first need to consider $u\in\CC^\infty_0((0,+\infty),\C)$ (the space of smooth functions vanishing at $+\infty$ together with all their derivatives) and integrate by parts once to change $y^{-\mathrm{i}\xi-l-1}\p_y^{k}u$ to $y^{-\mathrm{i}\xi-l}\p_y^{k+1}u$, then we use the density of $\CC^\infty_0((0,+\infty),\C)$ into $H^k((0,+\infty),\mathrm{d}\rho)$ (\emph{cf.} \cite[Thm. 3.22]{AF}). This completes the proof.
	\end{proof}
	\subsection{Decay in the fiber}
	\label{Decay in the fiber}
	Let $u\in H^{k,l}(X,\frac{\mathrm{d}\rho}{\rho}\mathrm{d}\omega)$ and set
	\begin{align*}
	v_{k,l,N/S}(\rho,\omega)&:=\widetilde{\phi}_{N/S}(\omega)\rho^{-l}(\p_{\rho}^k u)(\rho,\omega).
	\end{align*}
	We define
	\begin{align*}
	\bQ_{\sigma,k,l}^{\mathrm{f},1}u&:=\sum_{\bullet}\mathrm{Op}_{\CM,\CF}[F_{k,l}\mq_{\sigma,k,l,\bullet}^{\mathrm{f},1}]v_{k,l,\bullet}
	\end{align*}
	with
	\begin{align*}
	\mq_{\sigma,k,l,\bullet}^{\mathrm{f},1}&:=\chi\phi_\bullet\mq_{\sigma,k,l}^{\mathrm{f},1},\qquad\mq_{\sigma,k,l}^{\mathrm{f},1}:=\frac{\rho}{a\xi_{k,l}^2+\rho|\eta|_{\cancel{a}}^2}.
	\end{align*}
	In order to lighten notations below, we will just write $\mq_\bullet^{\mathrm{f},1}$ instead of $\mq_{\sigma,k,l,\bullet}^{\mathrm{f},1}$ (and similarly for $\mq_{\sigma,k,l}^{\mathrm{f},1}$). We will also use the notation $\widetilde{w}:=\rho^{-1}w$. Let us compute:
	\begin{align*}
	\bP_\sigma\rho^{\mathrm{i}\xi+k+l}\mathrm{e}^{\mathrm{i}\omega\cdot\eta}\mq_\bullet^{\mathrm{f},1}&=\chi\phi_\bullet\bP_\sigma\rho^{\mathrm{i}\xi+k+l}\mathrm{e}^{\mathrm{i}\omega\cdot\eta}\mq^{\mathrm{f},1}+[\bP_{\sigma},\chi\phi_\bullet]\rho^{\mathrm{i}\xi+k+l}\mathrm{e}^{\mathrm{i}\omega\cdot\eta}\mq^{\mathrm{f},1},
	\end{align*}
	\begin{align*}
	&\quad\bP_\sigma\rho^{\mathrm{i}\xi+k+l}\mathrm{e}^{\mathrm{i}\omega\cdot\eta}\mq^{\mathrm{f},1}\\
	&=\mathrm{e}^{\mathrm{i}\omega\cdot\eta}P_\sigma\rho^{\mathrm{i}\xi+k+l}\mq^{\mathrm{f},1}+\rho^{\mathrm{i}\xi+k+l}\cancel{P}_{\!\sigma}\mathrm{e}^{\mathrm{i}\omega\cdot\eta}\mq^{\mathrm{f},1}\\
	&=\rho^{\mathrm{i}\xi+k+l}\mathrm{e}^{\mathrm{i}\omega\cdot\eta}\Bigg\{\!\left(\frac{a\xi_{k,l}^2+b\xi_{k,l}+c}{\rho}\right)\mq^{\mathrm{f},1}-(2\mathrm{i}a\xi_{k,l}+\mathrm{i}b+a)(\p_{\rho}\mq^{\mathrm{f},1})-a\rho(\p_{\rho}^2\mq^{\mathrm{f},1})\\
	&\qquad\qquad\qquad\quad+\sum_{j=1}^n\left((\cancel{a}_j\eta_j^2+\cancel{b}_j\eta_j)\mq^{\mathrm{f},1}-\mathrm{i}(2\cancel{a}_j\eta_j+\cancel{b}_j)(\p_{\omega_j}\mq^{\mathrm{f},1})-\cancel{a}_j(\p_{\omega_j}^2\mq^{\mathrm{f},1})\right)\!\Bigg\}\\
	&=\rho^{\mathrm{i}\xi+k+l}\mathrm{e}^{\mathrm{i}\omega\cdot\eta}\bigg\{\!\left(a\xi_{k,l}^2+(b-2\mathrm{i}a)\xi_{k,l}-\mathrm{i}b-a+c\right)\widetilde{\mq}^{\mathrm{f},1}-\rho(2\mathrm{i}a\xi_{k,l}+\mathrm{i}b+3a)(\p_{\rho}\widetilde{\mq}^{\mathrm{f},1})-a\rho^2(\p_{\rho}^2\widetilde{\mq}^{\mathrm{f},1})\\
	&\qquad\qquad\qquad\quad+\rho\sum_{j=1}^n\left((\cancel{a}_j\eta_j^2+\cancel{b}_j\eta_j)\widetilde{\mq}^{\mathrm{f},1}-\mathrm{i}(2\cancel{a}_j\eta_j+\cancel{b}_j)(\p_{\omega_j}\widetilde{\mq}^{\mathrm{f},1})-\cancel{a}_j(\p_{\omega_j}^2\widetilde{\mq}^{\mathrm{f},1})\right)\!\bigg\},
	\end{align*}
	\begin{align*}
	&\quad[\bP_{\sigma},\chi\phi_\bullet]\rho^{\mathrm{i}\xi+k+l}\mathrm{e}^{\mathrm{i}\omega\cdot\eta}\mq^{\mathrm{f},1}\\
	&=\left(\phi_\bullet\mathrm{e}^{\mathrm{i}\omega\cdot\eta}[P_{\sigma},\chi]\rho^{\mathrm{i}\xi+k+l}+\chi\rho^{\mathrm{i}\xi+k+l}[\cancel{P}_{\sigma},\phi_\bullet]\mathrm{e}^{\mathrm{i}\omega\cdot\eta}\right)\mq^{\mathrm{f},1}\\
	&=-\left(\phi_\bullet\mathrm{e}^{\mathrm{i}\omega\cdot\eta}\left(\chi'(2a\rho\p_{\rho}+\mathrm{i}b+a)+a\rho\chi''\right)\rho^{\mathrm{i}\xi+k+l}+\chi\rho^{\mathrm{i}\xi+k+l}\sum_{j=1}^{n}\left((\p_{\omega_j}\phi_\bullet)(2\cancel{a}_j\p_{\omega_j}+\mathrm{i}\cancel{b}_j)+\cancel{a}_j(\p_{\omega_j}^2\phi_\bullet)\right)\mathrm{e}^{\mathrm{i}\omega\cdot\eta}\right)\mq^{\mathrm{f},1}\\
	&=-\rho^{\mathrm{i}\xi+k+l}\mathrm{e}^{\mathrm{i}\omega\cdot\eta}\bigg\{\phi_\bullet\left((2\mathrm{i}a\xi_{k,l}+\mathrm{i}b+a)\chi'\mq^{\mathrm{f},1}+a\rho\chi''\mq^{\mathrm{f},1}+2a\rho\chi'(\p_{\rho}\mq^{\mathrm{f},1})\right)\\
	&\qquad\qquad\qquad\qquad+\chi\sum_{j=1}^{n}\left((2\mathrm{i}\cancel{a}_j\eta_j+\mathrm{i}\cancel{b}_j)(\p_{\omega_j}\phi_\bullet)\mq^{\mathrm{f},1}+\cancel{a}_j(\p_{\omega_j}^2\phi_\bullet)\mq^{\mathrm{f},1}+2\cancel{a}_j(\p_{\omega_j}\phi_\bullet)(\p_{\omega_j}\mq^{\mathrm{f},1})\right)\bigg\}\\
	&=-\rho^{\mathrm{i}\xi+k+l}\mathrm{e}^{\mathrm{i}\omega\cdot\eta}\bigg\{\phi_\bullet\rho\left(\big((2(1+\mathrm{i}\xi_{k,l})a+\mathrm{i}b+a)\chi'+a\rho\chi''\big)\widetilde{\mq}^{\mathrm{f},1}+2a\rho\chi'(\p_{\rho}\widetilde{\mq}^{\mathrm{f},1})\right)\\
	&\qquad\qquad\qquad\qquad+\chi\rho\sum_{j=1}^{n}\left(\big(\mathrm{i}(2\cancel{a}_j\eta_j+\cancel{b}_j)(\p_{\omega_j}\phi_\bullet)+\cancel{a}_j(\p_{\omega_j}^2\phi_\bullet)\big)\widetilde{\mq}^{\mathrm{f},1}+2\cancel{a}_j(\p_{\omega_j}\phi_\bullet)(\p_{\omega_j}\widetilde{\mq}^{\mathrm{f},1})\right)\bigg\}.
	\end{align*}
	Above $'$ denotes the derivative with respect to $\rho$. It follows:
	\begin{align}
	\label{Eq PQ_sigma expression f}
	\bP_\sigma\mathrm{Op}&_{\CM,\CF}[F_{k,l}\mq_\bullet^{\mathrm{f},1}]=\chi\phi_\bullet\mathrm{Op}_{\CM,\CF}[F_{k,l}]\nonumber\\
	&+\chi\phi_\bullet\mathrm{Op}_{\CM,\CF}\bigg[F_{k,l}\bigg\{\!\left((b-2\mathrm{i}a)\xi_{k,l}-\mathrm{i}b-a+c\right)\widetilde{\mq}^{\mathrm{f},1}-\rho(2\mathrm{i}a\xi_{k,l}+\mathrm{i}b+3a)(\p_{\rho}\widetilde{\mq}^{\mathrm{f},1})-a\rho^2(\p_{\rho}^2\widetilde{\mq}^{\mathrm{f},1})\nonumber\\
	&\qquad\qquad\qquad\qquad\quad+\rho\sum_{j=1}^n\left(\cancel{b}_j\eta_j\widetilde{\mq}^{\mathrm{f},1}-\mathrm{i}(2\cancel{a}_j\eta_j+\cancel{b}_j)(\p_{\omega_j}\widetilde{\mq}^{\mathrm{f},1})-\cancel{a}_j(\p_{\omega_j}^2\widetilde{\mq}^{\mathrm{f},1})\right)\!\bigg\}\bigg]\nonumber\\
	&+\mathrm{Op}_{\CM,\CF}\bigg[F_{k,l}\bigg\{\phi_\bullet\rho\left(\big((2(1+\mathrm{i}\xi_{k,l})a+\mathrm{i}b+a)\chi'+a\rho\chi''\big)\widetilde{\mq}^{\mathrm{f},1}+2a\rho\chi'(\p_{\rho}\widetilde{\mq}^{\mathrm{f},1})\right)\nonumber\\
	&\qquad\qquad\qquad\quad\ +\chi\rho\sum_{j=1}^{n}\left(\big(\mathrm{i}(2\cancel{a}_j\eta_j+\cancel{b}_j)(\p_{\omega_j}\phi_\bullet)+\cancel{a}_j(\p_{\omega_j}^2\phi_\bullet)\big)\widetilde{\mq}^{\mathrm{f},1}+2\cancel{a}_j(\p_{\omega_j}\phi_\bullet)(\p_{\omega_j}\widetilde{\mq}^{\mathrm{f},1})\right)\!\bigg\}\bigg].
	\end{align}
	By Lem. \ref{Lem F_{k,l}}, we have
	\begin{align*}
	\chi\phi_\bullet\mathrm{Op}_{\CM,\CF}[F_{k,l}]v_{k,l,\bullet}&=\chi\phi_\bullet(\widetilde{\phi}_\bullet u)=\chi\phi_\bullet u;
	\end{align*}
	Denoting by $\mathfrak{N}_{d}\subset\N^{d}$ the subset of multi-indices with all coefficients zero but one which lies in $\{0,1,2\}$ (so $(0,2,0,\ldots,0)\in\mathfrak{N}_{d}$ but $(1,1,0,\ldots,0)\notin\mathfrak{N}_{d}$), we get 
	\begin{align}
	\label{Eq P_sigma f ...}
	%
	\bP_\sigma\mathrm{Op}_{\CM,\CF}[F_{k,l}\mq_\bullet^{\mathrm{f},1}]v_{k,l,\bullet}&=\chi\phi_\bullet u+\sum_{\alpha\in\mathfrak{N}_{n+1}}\p^\alpha(\chi\phi_\bullet)\sum_{\substack{\beta\in\N^{n+1}\\|\beta|\leq1}}\sum_{\gamma\in\mathfrak{N}_{2n+3}}\mathrm{Op}_{\CM,\CF}[F_{k,l}A_{\alpha,\beta,\gamma}(\p^\gamma\widetilde{\mq}^{\mathrm{f},1})]v_{k,l,\bullet}
	\end{align}
	where the multi-index $\beta$ indicates the powers of $\xi_{k,l}$ or $\rho\eta_j$ \emph{i.e.} $A_{\alpha,\beta,\gamma}=\CO\big(\xi_{k,l}^{\beta_1}+(\rho\eta)^{\beta_2}\big)$ with $|\beta|=\beta_1+|\beta_2|$ (so $A_{(0,\ldots,0),(0,\ldots,0),(0,\ldots,0)}=-\mathrm{i}b-a+c$, $A_{(0,\ldots,0),(1,\ldots,0),(0,\ldots,0)}=(b-2\mathrm{i}a)\xi_{k,l}$, etc.). \textbf{Notice that $\eta_j$ in $A_{\alpha,\beta,\gamma}$ is always multiplied by $\rho$}. Notice also that $A_{\alpha,\beta,\gamma}=0$ whether $|\alpha|\neq0$ and $|\gamma|=2$, or $|\beta|\neq0$ and $|\gamma|=2$.
	
	To lower the total order in $\xi$ and $\eta_j$ of the symbol of the remainder, we set
	\begin{align*}
	\mq^{\mathrm{f},2}_{\bullet}&:=\sum_{\alpha\in\mathfrak{N}_{n+1}}\p^\alpha(\chi\phi_\bullet)\mq^{\mathrm{f},2,\alpha},\qquad\mq^{\mathrm{f},2,\alpha}:=-\sum_{\substack{\beta\in\N^{n+1}\\|\beta|=1}}\sum_{\substack{\gamma\in\N^{2n+3}\\|\gamma|\leq1}}A_{\alpha,\beta,\gamma}(\p^\gamma\widetilde{\mq}^{\mathrm{f},1})\mq^{\mathrm{f},1};
	\end{align*}
	then using \eqref{Eq P_sigma f ...}, we compute
	\begin{align*}
	\bP_\sigma\mathrm{Op}_{\CM,\CF}[F_{k,l}\mq_{\bullet}^{\mathrm{f},2}]&=-\sum_{\alpha\in\mathfrak{N}_{n+1}}\p^\alpha(\chi\phi_\bullet)\sum_{\substack{\beta\in\N^{n+1}\\|\beta|=1}}\sum_{\gamma\in\mathfrak{N}_{2n+3}}\mathrm{Op}_{\CM,\CF}[F_{k,l}A_{\alpha,\beta,\gamma}(\p^\gamma\widetilde{\mq}^{\mathrm{f},1})]\\
	&\quad+\sum_{\alpha_1\in\mathfrak{N}_{n+1}}\sum_{\alpha_2\in\mathfrak{N}_{n+1}}\p^{\alpha_1+\alpha_2}(\chi\phi_\bullet)\sum_{\substack{\beta\in\N^{n+1}\\|\beta|\leq1}}\sum_{\gamma\in\mathfrak{N}_{2n+3}}\mathrm{Op}_{\CM,\CF}[F_{k,l}A_{\alpha_2,\beta,\gamma}(\p^\gamma\widetilde{\mq}^{\mathrm{f},2,\alpha_1})]
	\end{align*}
	\emph{i.e.}
	\begin{align*}
	\bP_\sigma\mathrm{Op}_{\CM,\CF}[F_{k,l}(\mq_{\bullet}^{\mathrm{f},1}+\mq_{\bullet}^{\mathrm{f},2})&]v_{k,l,\bullet}=\chi\phi_\bullet u+\sum_{\alpha\in\mathfrak{N}_{n+1}}\p^\alpha(\chi\phi_\bullet)\sum_{\substack{\beta\in\N^{n+1}\\|\beta|=0}}\sum_{\gamma\in\mathfrak{N}_{2n+3}}\mathrm{Op}_{\CM,\CF}[F_{k,l}A_{\alpha,\beta,\gamma}(\p^\gamma\widetilde{\mq}^{\mathrm{f},1})]v_{k,l,\bullet}\\
	&+\sum_{\alpha_1\in\mathfrak{N}_{n+1}}\sum_{\alpha_2\in\mathfrak{N}_{n+1}}\p^{\alpha_1+\alpha_2}(\chi\phi_\bullet)\sum_{\substack{\beta\in\N^{n+1}\\|\beta|\leq1}}\sum_{\gamma\in\mathfrak{N}_{2n+3}}\mathrm{Op}_{\CM,\CF}[F_{k,l}A_{\alpha_2,\beta,\gamma}(\p^\gamma\widetilde{\mq}^{\mathrm{f},2,\alpha_1})]v_{k,l,\bullet}.
	\end{align*}

	By induction over $N\in\N\setminus\{0,1,2\}$, we show that
	\begin{align*}
	\bP_\sigma\bQ_{\sigma,k,l}^{\mathrm{f},N}u&=\chi u+\bK_{\sigma,k,l}^{\mathrm{f},N}u
	\end{align*}
	where
	\begin{align*}
	\bQ_{\sigma,k,l}^{\mathrm{f},N}u&=\sum_{j=1}^N\left(\sum_{\bullet}\mathrm{Op}_{\CM,\CF}[F_{k,l}\mq_{\sigma,k,l,\bullet}^{\mathrm{f},j}]\right)v_{k,l,\bullet},\qquad\bK_{\sigma,k,l}^{\mathrm{f},N}u=\sum_{\bullet}\mathrm{Op}_{\CM,\CF}[F_{k,l}\bk_{\sigma,k,l,\bullet}^{\mathrm{f},N}]v_{k,l,\bullet}
	\end{align*}
	with\footnote{By convention, $\widetilde{\mq}_{\sigma,k,l}^{\mathrm{f},j-2,\alpha_1,\ldots,\alpha_{\lfloor j/2\rfloor-1}}:=\widetilde{\mq}_{\sigma,k,l}^{\mathrm{f},1}$ for $j=3$.}
	\begin{align*}
	\mq_{\sigma,k,l,\bullet}^{\mathrm{f},j}&=\sum_{i=\lfloor j/2\rfloor}^{j-1}\sum_{\alpha_1,\ldots,\alpha_i\in\mathfrak{N}_{n+1}}\p^{\alpha_1+\ldots+\alpha_i}(\chi\phi_\bullet)\mq_{\sigma,k,l}^{\mathrm{f},j,\alpha_1,\ldots,\alpha_i},\\
	\mq_{\sigma,k,l}^{\mathrm{f},j,\alpha_1,\ldots,\alpha_{\lfloor j/2\rfloor}}&=-\sum_{\gamma\in\mathfrak{N}_{2n+3}}\left(\sum_{\substack{\beta\in\N^{n+1}\\|\beta|=0}}A_{\alpha_{\lfloor j/2\rfloor},\beta,\gamma}(\p^\gamma\widetilde{\mq}_{\sigma,k,l}^{\mathrm{f},j-2,\alpha_1,\ldots,\alpha_{\lfloor j/2\rfloor-1}})\right.\\[-2mm]
	&\qquad\qquad\qquad\quad\left.+\delta_{j\in2\N}\sum_{\substack{\beta\in\N^{n+1}\\|\beta|=1}}A_{\alpha_{\lfloor j/2\rfloor},\beta,\gamma}(\p^\gamma\widetilde{\mq}_{\sigma,k,l}^{\mathrm{f},j-1,\alpha_1,\ldots,\alpha_{\lfloor j/2\rfloor-1}})\right)\mq_{\sigma,k,l}^{\mathrm{f},1},\\
	\mq_{\sigma,k,l}^{\mathrm{f},j,\alpha_1,\ldots,\alpha_i}&=-\sum_{\gamma\in\mathfrak{N}_{2n+3}}\left(\sum_{\substack{\beta\in\N^{n+1}\\|\beta|=0}}A_{\alpha_i,\beta,\gamma}(\p^\gamma\widetilde{\mq}_{\sigma,k,l}^{\mathrm{f},j-2,\alpha_1,\ldots,\alpha_{i-1}})+\sum_{\substack{\beta\in\N^{n+1}\\|\beta|=1}}A_{\alpha_i,\beta,\gamma}(\p^\gamma\widetilde{\mq}_{\sigma,k,l}^{\mathrm{f},j-1,\alpha_1,\ldots,\alpha_{i-1}})\right)\mq_{\sigma,k,l}^{\mathrm{f},1}
	\end{align*}
	for $\lfloor j/2\rfloor<i<j-1$,
	\begin{align*}
	\mq_{\sigma,k,l}^{\mathrm{f},j,\alpha_1,\ldots,\alpha_{j-1}}&=-\sum_{\substack{\beta\in\N^{n+1}\\|\beta|=1}}A_{\alpha_{j-1},\beta,\gamma}(\p^\gamma\widetilde{\mq}_{\sigma,k,l}^{\mathrm{f},j-1,\alpha_1,\ldots,\alpha_{j-2}})\mq_{\sigma,k,l}^{\mathrm{f},1}
	\end{align*}
	and
	\begin{align*}
	\bk_{\sigma,k,l,\bullet}^{\mathrm{f},N}&=\sum_{i=\lfloor(N+1)/2\rfloor}^{N-1}\sum_{\alpha_1,\ldots,\alpha_i\in\mathfrak{N}_{n+1}}\p^{\alpha_1+\ldots+\alpha_i}(\chi\phi_\bullet)\sum_{\substack{\beta\in\N^{n+1}\\|\beta|=0}}\sum_{\gamma\in\mathfrak{N}_{2n+3}}A_{\alpha_i,\beta,\gamma}(\p^\gamma\widetilde{\mq}_{\sigma,k,l}^{\mathrm{f},N-1,\alpha_1,\ldots,\alpha_{i-1}})\\
	&\qquad+\sum_{i=\lceil(N+1)/2\rceil}^{N}\sum_{\alpha_1,\ldots,\alpha_i\in\mathfrak{N}_{n+1}}\p^{\alpha_1+\ldots+\alpha_i}(\chi\phi_\bullet)\sum_{\substack{\beta\in\N^{n+1}\\|\beta|\leq1}}\sum_{\gamma\in\mathfrak{N}_{2n+3}}A_{\alpha_i,\beta,\gamma}(\p^\gamma\widetilde{\mq}_{\sigma,k,l}^{\mathrm{f},N,\alpha_1,\ldots,\alpha_{i-1}}).
	\end{align*}
	Indeed, defining $\mq_{\sigma,k,l,\bullet}^{\mathrm{f},N+1}$ according to the above formula, we get with \eqref{Eq P_sigma f ...}:
	\begin{align*}
	\bP_\sigma&\mathrm{Op}_{\CM,\CF}[F_{k,l}(\mq_{\sigma,k,l,\bullet}^{\mathrm{f},1}+\ldots+\mq_{\sigma,k,l,\bullet}^{\mathrm{f},N+1})]v_{k,l,\bullet}=\chi\phi_\bullet u\\
	&+\sum_{i=\lceil(N+1)/2\rceil}^{N}\sum_{\alpha_1,\ldots,\alpha_i\in\mathfrak{N}_{n+1}}\p^{\alpha_1+\ldots+\alpha_i}(\chi\phi_\bullet)\sum_{\substack{\beta\in\N^{n+1}\\|\beta|=0}}\sum_{\gamma\in\mathfrak{N}_{2n+3}}\mathrm{Op}_{\CM,\CF}[F_{k,l}A_{\alpha_i,\beta,\gamma}(\p^\gamma\widetilde{\mq}_{\sigma,k,l}^{\mathrm{f},N,\alpha_1,\ldots,\alpha_{i-1}})]v_{k,l,\bullet}\\
	&+\sum_{i=\lfloor(N+1)/2\rfloor+1}^{N+1}\sum_{\alpha_1,\ldots,\alpha_i\in\mathfrak{N}_{n+1}}\p^{\alpha_1+\ldots+\alpha_i}(\chi\phi_\bullet)\sum_{\substack{\beta\in\N^{n+1}\\|\beta|\leq1}}\sum_{\gamma\in\mathfrak{N}_{2n+3}}\mathrm{Op}_{\CM,\CF}[F_{k,l}A_{\alpha_i,\beta,\gamma}(\p^\gamma\widetilde{\mq}_{\sigma,k,l}^{\mathrm{f},N+1,\alpha_1,\ldots,\alpha_{i-1}})]v_{k,l,\bullet}.
	\end{align*}
	There has been $2N-3-\lfloor(N+1)/2\rfloor-\lceil(N+1)/2\rceil=2N-4+\delta_{N+1\in2\N}-2\lfloor(N+1)/2\rfloor$ terms cancelled in the remainder $\bk_{\sigma,k,l,\bullet}^{\mathrm{f},N}$; it remains to notice that $\lceil(N+1)/2\rceil=\lfloor(N+2)/2\rfloor$ and $\lfloor(N+1)/2\rfloor+1=\lceil(N+2)/2\rceil$ to complete the induction proof of the above formula.
	\begin{lemma}
	\label{Lem reg Q^f}
	For all $N\in\N\setminus\{0\}$, $\bQ_{\sigma,k,l}^{\mathrm{f},N}\in\CB(\CH^{k,l},\CX^{k+1,l})$.
	\end{lemma}
	\begin{proof}
		We can show (\emph{e.g.} using Fa\`a di Bruno formula) that, for all $N_1,N_2\in\N\setminus\{0\}$ and $\lceil N_2/2\rceil\leq N_2'\leq N_2$, $\rho^{N_2'}\p_{\omega_j}^{N_2}\rho^{N_1}\p_{\rho}^{N_1}\widetilde{\mq}_{\sigma,k,l}^{\mathrm{f},1}$ is a sum of terms controlled on the support of $\chi$ by
		\begin{align*}
		f(\rho,\omega,\xi,\eta)&:=\left|\frac{\rho^{s+t'}(\xi_{k,l}^{2s}+|\eta|_{\cancel{a}}^{2s})(\xi_{k,l}^{2t}+\rho^t|\eta|_{\cancel{a}}^{2t})}{(a\xi_{k,l}^2+\rho|\eta|_{\cancel{a}}^2)^{s+t+1}}\right|
		\end{align*}
		where $0\leq s\leq N_1$, $0\leq t\leq N_2$ and $\lceil t/2\rceil\leq t'\leq t$; the $s$ terms come from the derivatives $\rho^s\p_{\rho}^s$ (recall from \eqref{Eq PQ_sigma expression f} that each derivative $\p_{\rho}$ is multiplied on the left by $\rho$ for terms in front of $\chi$) while the $t,t'$ terms come from the derivatives $\rho^{t'}\p_{\omega_j}^{t}$ (as we have $\rho\p_{\omega_j}$ and $\rho\p_{\omega_j}^2$ in \eqref{Eq PQ_sigma expression f} for terms in front of $\chi$). Recall that the denominator of $f$ does not cancel on the support of $\chi$ by hypothesis on $a$ and $\cancel{a}$. If $|\eta|_{\cancel{a}}\leq\rho^{-1/2}$, then
		\begin{align*}
		f(\rho,\omega,\xi,\eta)&\leq\frac{\rho^{t'}(\rho^{s+t'}|\xi_{k,l}|^{2s}+1)(|\xi_{k,l}|^{2t}+1)}{\big|a\xi_{k,l}^2+\rho|\eta|_{\cancel{a}}^2\big|^{s+t+1}}\lesssim1
		\end{align*}
		as the total order in $\xi$ and $|\eta|_{\cancel{a}}$ is $-2$ (we use here that $|\eta_j|\lesssim\big||\eta|_{\cancel{a}}\big|$ by hypothesis). On the other hand, if $|\eta|_{\cancel{a}}\geq\rho^{-1/2}$, then
		\begin{align}
		\label{Eq with t'}
		f(\rho,\omega,\xi,\eta)&\leq\frac{\left(\frac{|\xi_{k,l}|^{2s}}{\big||\eta|_{\cancel{a}}\big|^{2s}}+1\right)\left(\frac{|\xi_{k,l}|^{2t}}{\rho^t\big||\eta|_{\cancel{a}}\big|^{2t}}+1\right)}{\left|\frac{a\xi_{k,l}^2}{\rho|\eta|_{\cancel{a}}^2}+1\right|^{s+t+1}}\frac{\rho^{s+t+t'}\big||\eta|_{\cancel{a}}\big|^{2(s+t)}}{\rho^{s+t+1}\big||\eta|_{\cancel{a}}\big|^{2(s+t+1)}}\lesssim1
		\end{align}
		as again the total order in $\xi$ and $|\eta|_{\cancel{a}}$ is $-2$. The same conclusion obviously holds for terms on the support of $\chi'$ and $\chi''$ (this time derivatives in \eqref{Eq PQ_sigma expression f} are not necessarily multiplied on the left by $\rho$, but $\rho>0$ there).
		
		Using the inductive construction of the $\mq_{\sigma,k,l,\bullet}^{\mathrm{f},j}$, we see that all the $\mq_{\sigma,k,l}^{\mathrm{f},j,\alpha_1,\ldots,\alpha_i}$ are written in terms of $A_{\alpha,\beta,\gamma}(\p^\gamma\mq_{\sigma,k,l}^{\mathrm{f},j-2,\alpha_1,\ldots,\alpha_{i-1}})$ with $|\beta|=0$ (this means no $\xi_{k,l}$ nor $\eta_1,\ldots,\eta_n$ in $A_{\alpha,\beta,\gamma}$) and of \linebreak$A_{\alpha,\beta,\gamma}(\p^\gamma\mq_{\sigma,k,l}^{\mathrm{f},j-1,\alpha_1,\ldots,\alpha_{i-1}})$ with $|\beta|=1$ (this means one power of $\xi_{k,l}$ or $\eta_1,\ldots,\eta_n$ in $A_{\alpha,\beta,\gamma}$); the total order in $\xi$ and $|\eta|_{\cancel{a}}$ of $\mq_{\sigma,k,l,\bullet}^{\mathrm{f},j}$ is therefore $-(j+1)$. The derivatives computation above of course applies to $\mathrm{\mq}_{\sigma,k,l}^{\mathrm{f},1}$ (the absence of $t'\geq1$ in \eqref{Eq with t'} is compensated by the extra $\rho$ in $\mathrm{\mq}_{\sigma,k,l}^{\mathrm{f},1}$). We thus deduce:
		\begin{align*}
		(|\xi|^{N_1}+|\eta_j|^{N_2})|\p_{\omega_j}^{N_2}(\rho\p_{\rho})^{N_1}\mq_{\sigma,k,l,\bullet}^{\mathrm{f},j}|&\lesssim1.
		\end{align*}
		Passing to the $x:=\ln\rho$ variable, the mixed quantization $\mathrm{Op}_{\CM,\CF}$ becomes a standard Fourier type quantization and $\R\times\CM\times\R\times\R^n\ni(x,\omega,\xi,\eta)\mapsto\mq_{\sigma,k,l,\bullet}^{\mathrm{f},j}(\mathrm{e}^x,\omega,\xi,\eta)$ is a (Fourier type) order 0 symbol (as $\p_x=\rho\p_{\rho}$) applied to $(x,\omega)\mapsto v_{k,l,\bullet}(\mathrm{e}^x,\omega)\in L^2(\R\times\CM,\mathrm{d}x\mathrm{d}\omega)$: the standard theory of pseudo-differential calculus implies that this term remains bounded in $L^2(\R\times\CM,\mathrm{d}x\mathrm{d}\omega)$ (the term $\rho^l=\mathrm{e}^{lx}$ being controlled in the support of $\chi$), which means that, back to the $\rho$ variable, $\bQ_{\sigma,k,l}^{\mathrm{f},N}u\in\CH^{0,l}$ \emph{i.e.} $\bQ_{\sigma,k,l}^{\mathrm{f},N}\in\CB(\CH^{k,l},\CH^{0,l})$.
		
		Next, we take derivatives of $\bQ_{\sigma,k,l}^{\mathrm{f},N}u$. First of all, $\p_{\rho}^{k+1}\bQ_{\sigma,k,l}^{\mathrm{f},N}u$ involves the terms
		\begin{align*}
		(\p_{\rho}^{k+1}\rho^{\mathrm{i}\xi}F_{k,l})(\p_{\rho}^{k+1-s}\mq_{\sigma,k,l,\bullet}^{\mathrm{f},j})&=\frac{\rho^{\mathrm{i}\xi+k+l-s}(\mathrm{i}\xi+k+l)_{(s)}}{(\mathrm{i}\xi+k+l)_{(k)}}(\p_{\rho}^{k+1-s}\mq_{\sigma,k,l,\bullet}^{\mathrm{f},j})\\
		&=\rho^l\rho^{\mathrm{i}\xi}\frac{(\mathrm{i}\xi+k+l)_{(s)}}{(\mathrm{i}\xi+k+l)_{(k)}\langle\xi\rangle}\frac{\langle\xi\rangle}{\rho}(\rho^{k+1-s}\p_{\rho}^{k+1-s}\mq_{\sigma,k,l,\bullet}^{\mathrm{f},j})
		\end{align*}
		for $1\leq j\leq N$ and $0\leq s\leq k+1$; the term $\rho^l$ gives the suitable weight for the target space $\CX^{k+1,l}\subset\CH^{k+1,l}$, $\rho^{\mathrm{i}\xi}$ is used in the inverse Mellin transform, $\frac{(\mathrm{i}\xi+k+l)_{(s)}}{(\mathrm{i}\xi+k+l)_{(k)}\langle\xi\rangle}$ is bounded uniformly in $\xi\in\R$, and the same argument as at the beginning of the proof sets that $(\rho^{k+1-s}\p_{\rho}^{k+1-s}\mq_{\sigma,k,l,\bullet}^{\mathrm{f},j})$ is bounded uniformly in $(\rho,\omega,\xi,\eta)\in\mathrm{Supp}(\chi)\times\mathrm{Supp}(\phi_\bullet)\times\R\times\R^n$ and is of order 1 in $\rho$ (because of the presence of $\mq_{\sigma,k,l}^{\mathrm{f},1}$) and of order $-1$ in $\xi$, which controls the term $\frac{\langle\xi\rangle}{\rho}$. It follows that, after passing to the $x:=\ln\rho$ variable, $\p_{\rho}^{k+1}\bQ_{\sigma,k,l}^{\mathrm{f},N}u$ is the Fourier type quantization of a order 0 symbol applied to $(x,\omega)\mapsto v_{k,l,\bullet}(\mathrm{e}^x,\omega)\in L^2(\R\times\CM,\mathrm{d}x\mathrm{d}\omega)$: as above, we deduce that this term in bounded in $\CH^{0,l}$, that is, $\p_{\rho}^{k+1}\bQ_{\sigma,k,l}^{\mathrm{f},N}\in\CB(\CH^{k,l},\CH^{0,l})$. Next, $\p_{\omega_i}^{k+1}\bQ_{\sigma,k,l}^{\mathrm{f},N}u$ involves the terms
		\begin{align*}
		(\p_{\omega_i}^{k+1}\mathrm{e}^{\mathrm{i}\omega\cdot\eta}F_{k,l})(\p_{\omega_i}^{k+1-s}\mq_{\sigma,k,l,\bullet}^{\mathrm{f},j})&=\rho^l\mathrm{e}^{\mathrm{i}\omega\cdot\eta}\frac{(\mathrm{i}\eta_i)^s}{(\mathrm{i}\xi+k+l)_{(k)}\langle\eta\rangle}\rho^s\frac{\langle\eta\rangle}{\rho}(\rho^{k+1-s}\p_{\omega_i}^{k+1-s}\mq_{\sigma,k,l,\bullet}^{\mathrm{f},j}).
		\end{align*}
		for all $1\leq i\leq n$; all the terms $\rho^l$, $\frac{1}{(\mathrm{i}\xi+k+l)_{(k)}}$ and $\frac{\langle\eta\rangle}{\rho}(\rho^{k+1-s}\p_{\omega_i}^{k+1-s}\mq_{\sigma,k,l,\bullet}^{\mathrm{f},j})$ are treated as for the $\rho$ derivatives above, $\rho^s\leq\rho_0^s$ ($\rho_0$ as in \eqref{Eq chi def}), while $\int_{\R^n}\mathrm{e}^{\mathrm{i}\omega\cdot\eta}(\mathrm{i}\eta_i)^s\langle\eta\rangle^{-1}\CF[v_{k,l,\bullet}]\mathrm{d}\eta$ gives an element of $\CH$. We conclude that $\p_{\omega_i}^{k+1}\bQ_{\sigma,k,l}^{\mathrm{f},N}\in\CB(\CH^{k,l},\CH^{0,l})$ as well. The mixed derivatives $\p_{\rho}^{s_1}\p_{\omega}^{s_2}$, $s_1+|s_2|=k+1$ are dealt with similarly.
	
		At last, it is clear that $\bP_\sigma\bQ_{\sigma,k,l}^{\mathrm{f},N}\in\CB(\CH^{k,l})$ since, by construction, $\bP_\sigma\bQ_{\sigma,k,l}^{\mathrm{f},N}=\chi+\bK_{\sigma,k,l}^{\mathrm{f},N}$ and $\bk_{\sigma,k,l,\bullet}^{\mathrm{f},N}$ has the same form as $\mq_{\sigma,k,l,\bullet}^{\mathrm{f},N}$ (with more decay in $\xi$ and $\eta$). This proves that $\bQ_{\sigma,k,l}^{\mathrm{f},N}\in\CB(\CH^{k,l},\CX^{k+1,l})$ and thus completes the proof.
	\end{proof}
	\begin{remark}
	\label{Rmk loss of 1 derivative}
	Notice the loss of one derivative with respect to the Fourier elliptic case: taking $k+2$ derivatives in $\rho$ would produce an extra term $\rho^{-1}$ which would target the wrong weighted Sobolev space.
	\end{remark}
	\subsection{Decay in the basis}
	\label{Decay in the basis}
	We keep the same notations as in Sect. \ref{Decay in the fiber}. The remainder term is not yet compact due to the lack of decay in the basis of the cotangent space, so we add corrections to the boundary parametrix $\bQ_{\sigma,k,l}^{\mathrm{f},N}$ in order to fix it.
	
	We set
	\begin{align*}
	\breve{w}_0&:=\breve{w}(0),\qquad \breve{w}(\rho):=\frac{w(\rho)-w_0}{\rho}
	\end{align*}
	for any function $w\in\CC^1([0,+\infty),\C)$, and
	\begin{align*}
	\mathfrak{N}_{d}^0&:=\big\{\alpha\in\mathfrak{N}_{d}\ \big\vert\ \text{the first coordinate of $\alpha$ is $0$}\big\},\\
	\mathfrak{N}_{d}^1&:=\big\{\alpha\in\mathfrak{N}_{d}\ \big\vert\ \text{the first coordinate of $\alpha$ is not $0$}\big\}.
	\end{align*}
	Given $\alpha_1,\ldots,\alpha_i\in\mathfrak{N}_{d}$, we will write $\boldsymbol{\alpha}_i:=(\alpha_1,\ldots,\alpha_i)\in\mathfrak{N}_{id}$ and use the notation $:\boldsymbol{\alpha}_i:$ for $\alpha_1+\ldots+\alpha_i\in\mathfrak{N}_{d}$.
	
	We define
	\begin{align}
	\label{Eq Indicial family 1}
	\mtp_{\sigma,k,l}(\xi)&:=a_0\xi_{k,l}^2+(b_0-2\mathrm{i}a_0)\xi_{k,l}-\mathrm{i}b_0-a_0+c_0
	\end{align}
	the \textbf{indicial family of $P_\sigma$} at $0$. It is the indicial polynomial from Frobenius method applied to $P_\sigma$ with the variable $\rho^{-\mathrm{i}\xi_{k,l}+1}$ instead of $\rho$, where $+1$ compensates the $\rho^{-1}$ factor of $P_\sigma$ and $k$, $l$ are the underlying Sobolev regularity and weight. In other words, if $\mathrm{Ind}_\sigma$ is the indicial polynomial\footnote{Recall that the indicial polynomial associated to $\rho^2\p_\rho^2+p(\rho)\p_\rho+q(\rho)$ is $\xi\mapsto\xi(\xi-1)+p(0)\xi+q(0)$.} associated to $P_\sigma$ then $\mtp_{\sigma,k,l}(\xi)=\mathrm{Ind}_\sigma(-\mathrm{i}\xi_{k,l}+1)$). \textbf{We assume that $0\notin\mtp_{\sigma,k,l}(\R)$}. Next, we define
	\begin{align*}
	\bQ_{\sigma,k,l}^{\mathrm{b},N,1}u&:=\sum_{\bullet}\mathrm{Op}_{\CM,\CF}[F_{k,l}\mq_{\sigma,k,l,\bullet}^{\mathrm{b},N,1}]v_{k,l,\bullet}
	\end{align*}
	with\footnote{This term will cancel the parts of $\bk_{\sigma,k,l,\bullet}^{\mathrm{f},N}$ containing $\chi$. We could remove from $\mq_{\sigma,k,l}^{\mathrm{b},N,1,\bullet}$ all the terms in $\bk_{\sigma,k,l}^{\mathrm{f},N,\bullet}$ of order $\rho$; this however would complicate the correction process below as we would have to track the dependence in $\rho$ of the coefficients of $\bk_{\sigma,k,l,\bullet}^{\mathrm{f},N}$.}
	\begin{align*}
	\mq_{\sigma,k,l,\bullet}^{\mathrm{b},N,1}&:=\sum_{i=\lfloor(N+1)/2\rfloor}^{N}\sum_{\boldsymbol{\alpha}_i\in\mathfrak{N}_{i(n+1)}^0}\p^{:\boldsymbol{\alpha}_i:}(\chi\phi_\bullet)\mq_{\sigma,k,l}^{\mathrm{b},N,1,\alpha_1,\ldots,\alpha_i},\\
	\mq_{\sigma,k,l}^{\mathrm{b},N,1,\alpha_1,\ldots,\alpha_{\lfloor(N+1)/2\rfloor}}&:=-\delta_{N\in2\N}\sum_{\substack{\beta\in\N^{n+1}\\|\beta|=0}}\sum_{\gamma\in\mathfrak{N}_{2n+3}}A_{\alpha_{\lfloor(N+1)/2\rfloor},\beta,\gamma}(\p^\gamma\widetilde{\mq}_{\sigma,k,l}^{\mathrm{f},N-1,\alpha_1,\ldots,\alpha_{\lfloor(N+1)/2\rfloor-1}})\frac{\rho}{\mtp_{\sigma,k,l}},
	\end{align*}
	for all $\lceil(N+1)/2\rceil\leq i\leq N-1$,
	\begin{align*}
	\mq_{\sigma,k,l}^{\mathrm{b},N,1,\alpha_1,\ldots,\alpha_i}&:=-\sum_{\substack{\beta\in\N^{n+1}\\|\beta|=0}}\sum_{\gamma\in\mathfrak{N}_{2n+3}}A_{\alpha_i,\beta,\gamma}(\p^\gamma\widetilde{\mq}_{\sigma,k,l}^{\mathrm{f},N-1,\alpha_1,\ldots,\alpha_{i-1}})\frac{\rho}{\mtp_{\sigma,k,l}}\\
	&\qquad-\sum_{\substack{\beta\in\N^{n+1}\\|\beta|\leq1}}\sum_{\gamma\in\mathfrak{N}_{2n+3}}A_{\alpha_i,\beta,\gamma}(\p^\gamma\widetilde{\mq}_{\sigma,k,l}^{\mathrm{f},N,\alpha_1,\ldots,\alpha_{i-1}})\frac{\rho}{\mtp_{\sigma,k,l}},\\
	\mq_{\sigma,k,l}^{\mathrm{b},N,1,\alpha_1,\ldots,\alpha_N}&:=-\sum_{\substack{\beta\in\N^{n+1}\\|\beta|\leq1}}\sum_{\gamma\in\mathfrak{N}_{2n+3}}A_{\alpha_N,\beta,\gamma}(\p^\gamma\widetilde{\mq}_{\sigma,k,l}^{\mathrm{f},N,\alpha_1,\ldots,\alpha_{N-1}})\frac{\rho}{\mtp_{\sigma,k,l}}.
	\end{align*}
	Note that $\mq_{\sigma,k,l,\bullet}^{\mathrm{b},N,1}$ is of order $\rho$. We again drop the indices $\sigma,k,l$ off the notations in the following computation to simplify them. Using the beginning of Sect. \ref{Decay in the fiber}, we have:
	\begin{align*}
	&\quad\bP_\sigma\rho^{\mathrm{i}\xi+k+l}\mathrm{e}^{\mathrm{i}\omega\cdot\eta}\mq^{\mathrm{b},N,1}\\
	&=\rho^{\mathrm{i}\xi+k+l}\mathrm{e}^{\mathrm{i}\omega\cdot\eta}\bigg\{\!\left(a_0\xi_{k,l}^2+(b_0-2\mathrm{i}a_0)\xi_{k,l}-\mathrm{i}b_0-a_0+c_0\right)\widetilde{\mq}^{\mathrm{b},N,1}\\
	&\qquad\qquad\qquad\quad+\rho\left(\breve{a}\xi_{k,l}^2+(\breve{b}-2\mathrm{i}\breve{a})\xi_{k,l}-\mathrm{i}\breve{b}-\breve{a}+\breve{c}\right)\widetilde{\mq}^{\mathrm{b},N,1}-\rho(2\mathrm{i}a\xi_{k,l}+\mathrm{i}b+3a)(\p_{\rho}\widetilde{\mq}^{\mathrm{b},N,1})-a\rho^2(\p_{\rho}^2\widetilde{\mq}^{\mathrm{b},N,1})\\
	&\qquad\qquad\qquad\quad+\rho\sum_{j=1}^n\left((\cancel{a}_j\eta_j^2+\cancel{b}_j\eta_j)\widetilde{\mq}^{\mathrm{b},N,1}-\mathrm{i}(2\cancel{a}_j\eta_j+\cancel{b}_j)(\p_{\omega_j}\widetilde{\mq}^{\mathrm{b},N,1})-\cancel{a}_j(\p_{\omega_j}^2\widetilde{\mq}^{\mathrm{b},N,1})\right)\!\bigg\}.
	\end{align*}
	In order to further simplify notations for the next computation, we simply write $\widetilde{\mq}^{\mathrm{b},N,1,\boldsymbol{\alpha}_i}$ for $\widetilde{\mq}^{\mathrm{b},N,1,\alpha_1,\ldots,\alpha_i}$. It follows
	\begin{align*}
	&\quad\bP_\sigma\mathrm{Op}_{\CM,\CF}[F_{k,l}\p^{:\boldsymbol{\alpha}_i:}(\chi\phi_\bullet)\mq^{\mathrm{b},N,1,\boldsymbol{\alpha}_i}]\nonumber\\
	&=\p^{:\boldsymbol{\alpha}_i:}(\chi\phi_\bullet)\mathrm{Op}_{\CM,\CF}[F_{k,l}\mtp_{\sigma,k,l}\widetilde{\mq}^{\mathrm{b},N,1,\boldsymbol{\alpha}_i}]\nonumber\\
	&\quad+\p^{:\boldsymbol{\alpha}_i:}(\chi\phi_\bullet)\mathrm{Op}_{\CM,\CF}\bigg[F_{k,l}\bigg\{\!\rho\left(\breve{a}\xi_{k,l}^2+(\breve{b}-2\mathrm{i}\breve{a})\xi_{k,l}-\mathrm{i}\breve{b}-\breve{a}+\breve{c}\right)\widetilde{\mq}^{\mathrm{b},N,1,\boldsymbol{\alpha}_i}\nonumber\\
	&\qquad\qquad\qquad\qquad\qquad\quad\quad-\rho(2\mathrm{i}a\xi_{k,l}+\mathrm{i}b+3a)(\p_{\rho}\widetilde{\mq}^{\mathrm{b},N,1,\boldsymbol{\alpha}_i})-a\rho^2(\p_{\rho}^2\widetilde{\mq}^{\mathrm{b},N,1,\boldsymbol{\alpha}_i})\nonumber\\
	&\qquad\qquad\qquad\qquad\qquad\quad\quad+\rho\sum_{j=1}^n\left((\cancel{a}_j\eta_j^2+\cancel{b}_j\eta_j)\widetilde{\mq}^{\mathrm{b},N,1,\boldsymbol{\alpha}_i}-\mathrm{i}(2\cancel{a}_j\eta_j+\cancel{b}_j)(\p_{\omega_j}\widetilde{\mq}^{\mathrm{b},N,1,\boldsymbol{\alpha}_i})-\cancel{a}_j(\p_{\omega_j}^2\widetilde{\mq}^{\mathrm{b},N,1,\boldsymbol{\alpha}_i})\right)\!\bigg\}\bigg]\nonumber\\
	&\quad+\mathrm{Op}_{\CM,\CF}\bigg[F_{k,l}\bigg\{\!\left(\left((2a\rho+\mathrm{i}b+a)\p_{\rho}\p^{:\boldsymbol{\alpha}_i:}(\chi\phi_\bullet)(1+\mathrm{i}\xi_{k,l})+a\rho\p_{\rho}^2\p^{:\boldsymbol{\alpha}_i:}(\chi\phi_\bullet)\right)\widetilde{\mq}^{\mathrm{b},N,1,\boldsymbol{\alpha}_i}\right.\nonumber\\
	&\qquad\qquad\qquad\qquad\quad\left.+(2a\rho+\mathrm{i}b+a)\p_{\rho}\p^{:\boldsymbol{\alpha}_i:}(\chi\phi_\bullet)(\p_{\rho}\widetilde{\mq}^{\mathrm{b},N,1,\boldsymbol{\alpha}_i})\right)\nonumber\\
	&\qquad\qquad\qquad\qquad\quad+\rho\sum_{j=1}^{n}\left((2\cancel{a}_j+\mathrm{i}\cancel{b}_j)\p_{\omega_j}\p^{:\boldsymbol{\alpha}_i:}(\chi\phi_\bullet)\mathrm{i}\eta_j\widetilde{\mq}^{\mathrm{b},N,1,\boldsymbol{\alpha}_i}+\cancel{a}_j\p_{\omega_j}^2\p^{:\boldsymbol{\alpha}_i:}(\chi\phi_\bullet)\widetilde{\mq}^{\mathrm{b},N,1,\boldsymbol{\alpha}_i}\right.\nonumber\\[-5mm]
	&\qquad\qquad\qquad\qquad\qquad\qquad\quad\left.+(2\cancel{a}_j+\mathrm{i}\cancel{b}_j)\p_{\omega_j}\p^{:\boldsymbol{\alpha}_i:}(\chi\phi_\bullet)(\p_{\omega_j}\widetilde{\mq}^{\mathrm{b},N,1,\boldsymbol{\alpha}_i})\right)\!\bigg\}\bigg]
	\end{align*}
	\emph{i.e.}
	\begin{align}
	\label{Eq P_sigma b ...}
	\bP_\sigma&\mathrm{Op}_{\CM,\CF}[F_{k,l}\mq_{\sigma,k,l,\bullet}^{\mathrm{b},N,1}]=\sum_{i=\lfloor(N+1)/2\rfloor}^{N}\sum_{\boldsymbol{\alpha}_i\in\mathfrak{N}_{i(n+1)}^0}\p^{:\boldsymbol{\alpha}_i:}(\chi\phi_\bullet)\mathrm{Op}_{\CM,\CF}[F_{k,l}\mtp_{\sigma,k,l}\widetilde{\mq}_{\sigma,k,l}^{\mathrm{b},N,1,\alpha_1,\ldots,\alpha_i}]\nonumber\\
	&+\sum_{i=\lfloor(N+1)/2\rfloor}^{N}\sum_{\substack{\boldsymbol{\alpha}_i\in\mathfrak{N}_{i(n+1)}^0\\\alpha_{i+1}\in\mathfrak{N}_{n+1}}}\p^{:\boldsymbol{\alpha}_{i+1}:}(\chi\phi_\bullet)\sum_{\beta\in\mathfrak{N}_{n+1}}\sum_{\gamma\in\mathfrak{N}_{2n+3}}\sum_{0\leq\delta\leq2}\mathrm{Op}_{\CM,\CF}[F_{k,l}B_{\alpha_1,\ldots,\alpha_{i+1},\beta,\gamma,\delta}(\p^\gamma\widetilde{\mq}_{\sigma,k,l}^{\mathrm{b},N,1,\alpha_1,\ldots,\alpha_i})]
	\end{align}
	where $B_{\alpha_1,\ldots,\alpha_{i+1},\beta,\gamma,\delta}$ denotes the coefficients in \eqref{Eq P_sigma b ...} in front of $\p^{:\boldsymbol{\alpha}_{i+1}:}(\chi\phi_\bullet)(\p^\gamma\widetilde{\mq}_{\sigma,k,l}^{\mathrm{b},N,1,\alpha_1,\ldots,\alpha_i})$ containing a term $\rho^\delta\xi_{k,l}^{\beta'}\eta_1^{\beta_1''}\ldots\eta_n^{\beta_n''}$ with $|\beta|=\beta'+|\beta''|$ (note that all the $\beta'$, $\beta''_j$ are zero but one which can be 0,1 or 2 because $\beta\in\mathfrak{N}_{n+1}$). \textbf{Notice that $B_{\alpha_1,\ldots,\alpha_{i+1},\beta,\gamma,\delta}$ is either of order $\rho$, $\rho^2$} (when $\alpha_{i+1}\in\mathfrak{N}_{n+1}^0$) \textbf{or contains $\chi'$ or $\chi''$} (when $\delta=0$). Note also that the remainder (the term on the last line) in \eqref{Eq P_sigma b ...} is of order $\rho$.
	
	To increase the total order in $\rho$ of the remainder, we define
	\begin{align*}
	\mtp_{\sigma,k,l,2}(\xi)&:=a_0\xi_{k,l}^2+(b_0-4\mathrm{i}a_0)\xi_{k,l}-2\mathrm{i}b_0-4a_0+c_0=\mtp_{\sigma,k+1,l}(\xi)
	\end{align*}
	and \textbf{we assume that $0\notin\mtp_{\sigma,k,l,2}(\R)$}. We then set
	\begin{align*}
	\bQ_{\sigma,k,l}^{\mathrm{b},N,2}u&:=\sum_{\bullet}\mathrm{Op}_{\CM,\CF}[F_{k,l}\mq_{\sigma,k,l,\bullet}^{\mathrm{b},N,2}]v_{k,l,\bullet}
	\end{align*}
	with
	\begin{align*}
	\mq_{\sigma,k,l,\bullet}^{\mathrm{b},N,2}&:=\sum_{i=\lfloor(N+1)/2\rfloor}^{N}\sum_{\boldsymbol{\alpha}_{i+1}\in\mathfrak{N}_{(i+1)(n+1)}^0}\p^{:\boldsymbol{\alpha}_{i+1}:}(\chi\phi_\bullet)\mq_{\sigma,k,l}^{\mathrm{b},N,2,\alpha_1,\ldots,\alpha_{i+1}},\\
	\mq_{\sigma,k,l}^{\mathrm{b},N,2,\alpha_1,\ldots,\alpha_{i+1}}&:=-\sum_{\beta\in\mathfrak{N}_{n+1}}\sum_{\gamma\in\mathfrak{N}_{2n+3}}\sum_{1\leq\delta\leq2}\mathrm{Op}_{\CM,\CF}[F_{k,l}B_{\alpha_1,\ldots,\alpha_{i+1},\beta,\gamma,\delta}(\p^\gamma\widetilde{\mq}^{\mathrm{b},N,1,\alpha_1,\ldots,\alpha_i})]\frac{\rho}{\mtp_{\sigma,k,l,2}}.
	\end{align*}
	Once again, $B_{\alpha_1,\ldots,\alpha_{i+1},\beta,\gamma,0}=0$ when $\alpha_1,\ldots,\alpha_{i+1}\in\mathfrak{N}_{n+1}^0$ so that the above term will cancel the remainder of order $\rho$. Note that $\mq_{\sigma,k,l,\bullet}^{\mathrm{b},N,2}$ is of order $\rho^2$ but is of order $-N+2$ in $|\eta|_{\cancel{a}}$ (contrary to the first iteration, we had here to cancel terms containing $\rho|\eta|_{\cancel{a}}^2$). If $\widetilde{\widetilde{\mq}}^{\mathrm{b},N,2}:=\rho^{-2}\mq^{\mathrm{b},N,2}$ then:
	\begin{align*}
	&\quad\bP_\sigma\rho^{\mathrm{i}\xi+k+l}\mathrm{e}^{\mathrm{i}\omega\cdot\eta}\mq^{\mathrm{b},N,2}\\
	&=\rho^{\mathrm{i}\xi+k+l}\mathrm{e}^{\mathrm{i}\omega\cdot\eta}\bigg\{\rho\left(a_0\xi_{k,l}^2+(b_0-4\mathrm{i}a_0)\xi_{k,l}-2\mathrm{i}b_0-4a_0+c_0\right)\widetilde{\widetilde{\mq}}^{\mathrm{b},N,2}\\
	&\qquad\qquad\qquad\quad\!\!+\rho^2\left(\breve{a}\xi_{k,l}^2+(\breve{b}-4\mathrm{i}\breve{a})\xi_{k,l}-2\mathrm{i}\breve{b}-4\breve{a}+\breve{c}\right)\widetilde{\widetilde{\mq}}^{\mathrm{b},N,2}-\rho^2(2\mathrm{i}a\xi_{k,l}+\mathrm{i}b+5a)(\p_{\rho}\widetilde{\widetilde{\mq}}^{\mathrm{b},N,2})-a\rho^3(\p_{\rho}^2\widetilde{\widetilde{\mq}}^{\mathrm{b},N,2})\\
	&\qquad\qquad\qquad\quad\!\!+\rho^2\sum_{j=1}^n\left((\cancel{a}_j\eta_j^2+\cancel{b}_j\eta_j)\widetilde{\widetilde{\mq}}^{\mathrm{b},N,2}-\mathrm{i}(2\cancel{a}_j\eta_j+\cancel{b}_j)(\p_{\omega_j}\widetilde{\widetilde{\mq}}^{\mathrm{b},N,2})-\cancel{a}_j(\p_{\omega_j}^2\widetilde{\widetilde{\mq}}^{\mathrm{b},N,2})\right)\!\bigg\}.
	\end{align*}
	The same computations as in \eqref{Eq P_sigma b ...} thus give
	\begin{align*}
	&\quad\bP_\sigma\mathrm{Op}_{\CM,\CF}[F_{k,l}\mq_{\sigma,k,l,\bullet}^{\mathrm{b},N,2}]\\
	&=\sum_{i=\lfloor(N+1)/2\rfloor}^{N}\sum_{\boldsymbol{\alpha}_{i+1}\in\mathfrak{N}_{(i+1)(n+1)}^0}\p^{:\boldsymbol{\alpha}_{i+1}:}(\chi\phi_\bullet)\sum_{\beta\in\mathfrak{N}_{n+1}}\sum_{\gamma\in\mathfrak{N}_{2n+3}}\sum_{1\leq\delta\leq2}\mathrm{Op}_{\CM,\CF}[F_{k,l}B_{\alpha_1,\ldots,\alpha_{i+1},\beta,\gamma,\delta}(\p^\gamma\widetilde{\mq}_{\sigma,k,l}^{\mathrm{b},N,1,\alpha_1,\ldots,\alpha_i})]\\
	&\quad+\sum_{i=\lfloor(N+1)/2\rfloor}^{N}\sum_{\substack{\boldsymbol{\alpha}_{i+1}\in\mathfrak{N}_{(i+1)(n+1)}^0\\\alpha_{i+2}\in\mathfrak{N}_{n+1}}}\p^{:\boldsymbol{\alpha}_{i+2}:}(\chi\phi_\bullet)\sum_{\beta\in\mathfrak{N}_{n+1}}\sum_{\gamma\in\mathfrak{N}_{2n+3}}\sum_{0\leq\delta\leq2}\mathrm{Op}_{\CM,\CF}[F_{k,l}B_{\alpha_1,\ldots,\alpha_{i+2},\beta,\gamma,\delta}(\p^\gamma\widetilde{\mq}_{\sigma,k,l}^{\mathrm{b},N,2,\alpha_1,\ldots,\alpha_i})]
	\end{align*}
	with $B_{\alpha_1,\ldots,\alpha_{i+2},\beta,\gamma,\delta}$ denotes the coefficients in the remainder in front of $\p^{:\boldsymbol{\alpha}_{i+2}:}(\chi\phi_\bullet)(\p^\gamma\widetilde{\mq}_{\sigma,k,l}^{\mathrm{b},N,2,\alpha_1,\ldots,\alpha_i})$ containing a term $\rho^\delta\xi_{k,l}^{\beta'}\eta_1^{\beta_1''}\ldots\eta_n^{\beta_n''}$ with $|\beta|=\beta'+|\beta''|$. Since $\mq_{\sigma,k,l,\bullet}^{\mathrm{b},N,2}$ is of order $\rho^2$, the remainder above is of order $\rho^2$.
	
	For all $j\in\N\setminus\{0\}$, we define
	\begin{align}
	\label{Eq Indicial family j}
	\mtp_{\sigma,k,l,j}(\xi)&:=a_0\xi_{k,l}^2+(b_0-2\mathrm{i}ja_0)\xi_{k,l}-j(\mathrm{i}b_0+ja_0)+c_0=\mtp_{\sigma,k+j-1,l}(\xi)
	\end{align}
	and \textbf{we assume that $0\notin\mtp_{\sigma,k,l,j}(\R)$}. By induction over $N'\in\N\setminus\{0,1,2\}$, we show that
	\begin{align*}
	\bP_\sigma\bQ_{\sigma,k,l}^{\mathrm{b},N,N'}u&=\sum_{i=\lfloor(N+1)/2\rfloor}^{N}\sum_{\boldsymbol{\alpha}_i\in\mathfrak{N}_{i(n+1)}^0}\p^{:\boldsymbol{\alpha}_i:}(\chi\phi_\bullet)\mathrm{Op}_{\CM,\CF}[F_{k,l}\mtp_{\sigma,k,l}\widetilde{\mq}_{\sigma,k,l}^{\mathrm{b},N,1,\alpha_1,\ldots,\alpha_i}]v_{k,l,\bullet}+\bK_{\sigma,k,l}^{\mathrm{b},N,N'}u
	\end{align*}
	where
	\begin{align*}
	\bQ_{\sigma,k,l}^{\mathrm{b},N,N'}u&=\sum_{j=1}^{N'}\left(\sum_{\bullet}\mathrm{Op}_{\CM,\CF}[F_{k,l}\mq_{\sigma,k,l,\bullet}^{\mathrm{b},N,j}]\right)v_{k,l,\bullet},\qquad\bK_{\sigma,k,l}^{\mathrm{b},N,N'}u=\sum_{j=1}^{N'}\left(\sum_{\bullet}\mathrm{Op}_{\CM,\CF}[F_{k,l}\bk_{\sigma,k,l,\bullet}^{\mathrm{b},N,j}]v_{k,l,\bullet}\right)
	\end{align*}
	with
	\begin{align*}
	\mq_{\sigma,k,l,\bullet}^{\mathrm{b},N,j}&:=\sum_{i=\lfloor(N+1)/2\rfloor}^{N}\sum_{\boldsymbol{\alpha}_{i+j-1}\in\mathfrak{N}_{(i+j-1)(n+1)}^0}\p^{:\boldsymbol{\alpha}_{i+j-1}:}(\chi\phi_\bullet)\mq_{\sigma,k,l}^{\mathrm{b},N,j,\alpha_1,\ldots,\alpha_{i+j-1}},\\
	\mq_{\sigma,k,l}^{\mathrm{b},N,j,\alpha_1,\ldots,\alpha_{i+j-1}}&:=-\sum_{\beta\in\mathfrak{N}_{n+1}}\sum_{\gamma\in\mathfrak{N}_{2n+3}}\sum_{1\leq\delta\leq2}\mathrm{Op}_{\CM,\CF}[F_{k,l}B_{\alpha_1,\ldots,\alpha_{i+j-1},\beta,\gamma,\delta}(\p^\gamma\widetilde{\mq}^{\mathrm{b},N,j-1,\alpha_1,\ldots,\alpha_{i+j-2}})]\frac{\rho}{\mtp_{\sigma,k,l,j}}.
	\end{align*}
	and
	\begin{align*}
	\bk_{\sigma,k,l,\bullet}^{\mathrm{b},N,j}&=\sum_{i=\lfloor(N+1)/2\rfloor}^{N}\sum_{\substack{\boldsymbol{\alpha}_{i+j-1}\in\mathfrak{N}_{(i+j)(n+1)}^0\\\alpha_{i+j}\in\mathfrak{N}_{n+1}}}\p^{:\boldsymbol{\alpha}_{i+j}:}(\chi\phi_\bullet)\sum_{\beta\in\mathfrak{N}_{n+1}}\sum_{\gamma\in\mathfrak{N}_{2n+3}}\sum_{0\leq\delta\leq2}\mathrm{Op}_{\CM,\CF}[F_{k,l}B_{\alpha_1,\ldots,\alpha_{i+j},\beta,\gamma,\delta}(\p^\gamma\widetilde{\mq}_{\sigma,k,l}^{\mathrm{b},N,j,\alpha_1,\ldots,\alpha_i})].
	\end{align*}
	Besides, $\mq_{\sigma,k,l,\bullet}^{\mathrm{b},N,j}$ and $\bk_{\sigma,k,l,\bullet}^{\mathrm{b},N,j}$ are of order $j$ in $\rho$ and of order $-N+2(j-1)$ in $|\eta|_{\cancel{a}}$. Indeed, we can check that
	\begin{align*}
	&\quad\bP_\sigma\rho^{\mathrm{i}\xi+k+l}\mathrm{e}^{\mathrm{i}\omega\cdot\eta}\mq^{\mathrm{b},N,j}\\
	&=\rho^{\mathrm{i}\xi+k+l}\mathrm{e}^{\mathrm{i}\omega\cdot\eta}\bigg\{\rho^{j-1}\left(a_0\xi_{k,l}^2+(b_0-2\mathrm{i}ja_0)\xi_{k,l}-j(\mathrm{i}b_0-ja_0)+c_0\right)(\rho^{-j}\mq^{\mathrm{b},N,j})\\
	&\qquad\qquad\qquad\quad\!\!+\rho^{j}\left(\breve{a}\xi_{k,l}^2+(\breve{b}-2\mathrm{i}j\breve{a})\xi_{k,l}-j(\mathrm{i}\breve{b}-j\breve{a})+\breve{c}\right)(\rho^{-j}\mq^{\mathrm{b},N,j})\\
	&\qquad\qquad\qquad\quad\!\!-\rho^{j}(2\mathrm{i}a\xi_{k,l}+\mathrm{i}b+(2j+1)a)(\p_{\rho}(\rho^{-j}\mq^{\mathrm{b},N,j}))-a\rho^{j+1}(\p_{\rho}^2(\rho^{-j}\mq^{\mathrm{b},N,j}))\\
	&\qquad\qquad\qquad\quad\!\!+\rho^{j}\sum_{j=1}^n\left((\cancel{a}_j\eta_j^2+\cancel{b}_j\eta_j)(\rho^{-j}\mq^{\mathrm{b},N,j})-\mathrm{i}(2\cancel{a}_j\eta_j+\cancel{b}_j)(\p_{\omega_j}(\rho^{-j}\mq^{\mathrm{b},N,j}))-\cancel{a}_j(\p_{\omega_j}^2(\rho^{-j}\mq^{\mathrm{b},N,j}))\right)\!\bigg\}.
	\end{align*}
	so that defining inductively $\mq_{\sigma,k,l,\bullet}^{\mathrm{b},N,j}$ for all $1\leq j\leq N'+1$ according to the above formula, $\mq_{\sigma,k,l,\bullet}^{\mathrm{b},N,N'+1}$ is of order $\rho^{N'+1}$; the remainder $\bk_{\sigma,k,l,\bullet}^{\mathrm{b},N,N'+1}$ is also of order $\rho^{N'+1}$ as the term of order $\rho^{N'}$ in $\bP_\sigma\rho^{\mathrm{i}\xi+k+l}\mathrm{e}^{\mathrm{i}\omega\cdot\eta}\mq_{\sigma,k,l,\bullet}^{\mathrm{b},N,N'+1}$ precisely cancels the one in $\bk_{\sigma,k,l,\bullet}^{\mathrm{b},N,N'}$. As for the total order in $|\eta|_{\cancel{a}}$, each iteration increase it by a factor 2 (starting from $j=2$) as $\mtp_{\sigma,k,l,j}(\xi)^{-1}$ do not produce any decay in the $\eta$ directions of the fiber.
	
	The same proof as for Lem. \ref{Lem reg Q^b} proves the following result:
	\begin{lemma}
		\label{Lem reg Q^b}
		For all $N,N'\in\N\setminus\{0\}$, $\bQ_{\sigma,k,l}^{\mathrm{b},N,N'}\in\CB(\CH^{k,l},\CX^{k+1,l})$ provided that $N>2(N'-1)$ and $0\notin\bigcup_{j=1}^{N'}\mtp_{\sigma,k+j-1,l}(\R)$.
	\end{lemma}
	\subsection{Global boundary parametrices}
	\label{Global boundary parametrices}
	We prove in this section Prop. \ref{Prop Local Parametrix}.
	
	We first define a global boundary right parametrix for $\bP_{\sigma}$. Let $N,N'\in\N\setminus\{0\}$ such that $N>2(N'-1)$ and assume that $0\notin\bigcup_{j=1}^{N'}\mtp_{\sigma,k+j-1,l}(\R)$. Let
	\begin{align*}
	\bQ_{\sigma,k,l}^{N,N'}&:=\bQ_{\sigma,k,l}^{\mathrm{f},N}+\bQ_{\sigma,k,l}^{\mathrm{b},N,N'}\in\CB(\CH^{k,l},\CX^{k+1,l}).
	\end{align*}
	The mapping property $\CH^{k,l}\to\CX^{k+1,l}$ of $\bQ_{\sigma,k,l}^{N,N'}$ has been established in Lem. \ref{Lem reg Q^f} and Lem. \ref{Lem reg Q^b}. The computations carried out in Sect. \ref{Decay in the fiber} and Sect. \ref{Decay in the basis} show that
	\begin{align*}
	\bP_\sigma\bQ_{\sigma,k,l}^{N,N'}u&=\chi u+\bK_{\sigma,k,l}^{N,N'}u
	\end{align*}
	where
	\begin{align}
	\label{Eq K^{N,N'}}
	\bK_{\sigma,k,l}^{N,N'}u=\sum_{\bullet}\mathrm{Op}_{\CM,\CF}[F_{k,l}\bk_{\sigma,k,l,\bullet}^{N,N'}]v_{k,l,\bullet},\qquad\bk_{\sigma,k,l,\bullet}^{N,N'}=\bk_{\sigma,k,l,\bullet}^{\mathrm{f},N,1}+\sum_{j=1}^{N'}\mathrm{Op}_{\CM,\CF}[F_{k,l}\bk_{\sigma,k,l,\bullet}^{\mathrm{b},N,j}]
	\end{align}
	with
	\begin{align*}
	\bk_{\sigma,k,l,\bullet}^{\mathrm{f},N,1}&=\sum_{i=\lfloor(N+1)/2\rfloor}^{N-1}\sum_{\boldsymbol{\alpha}_i\in\mathfrak{N}_{i(n+1)}^1}\p^{:\boldsymbol{\alpha}_i:}(\chi\phi_\bullet)\sum_{\substack{\beta\in\N^{n+1}\\|\beta|=0}}\sum_{\gamma\in\mathfrak{N}_{2n+3}}A_{\alpha_i,\beta,\gamma}(\p^\gamma\widetilde{\mq}_{\sigma,k,l}^{\mathrm{f},N-1,\alpha_1,\ldots,\alpha_{i-1}})\\
	&\qquad+\sum_{i=\lceil(N+1)/2\rceil}^{N}\sum_{\boldsymbol{\alpha}_i\in\mathfrak{N}_{i(n+1)}^1}\p^{:\boldsymbol{\alpha}_i:}(\chi\phi_\bullet)\sum_{\substack{\beta\in\N^{n+1}\\|\beta|\leq1}}\sum_{\gamma\in\mathfrak{N}_{2n+3}}A_{\alpha_i,\beta,\gamma}(\p^\gamma\widetilde{\mq}_{\sigma,k,l}^{\mathrm{f},N,\alpha_1,\ldots,\alpha_{i-1}}).
	\end{align*}
	Above, $\bk_{\sigma,k,l,\bullet}^{\mathrm{f},N,1}$ is nothing but the terms in $\bk_{\sigma,k,l,\bullet}^{\mathrm{f},N}$ from Sect. \ref{Decay in the fiber} containing derivatives of $\chi$ (hence supported far away from $\{0\}_\rho$) and $\bk_{\sigma,k,l,\bullet}^{\mathrm{b},N,j}$ are from Sect. \ref{Decay in the basis}.
	
	It remains to prove:
	\begin{lemma}
	\label{Lem Compactness K}
	The operator $\bK_{\sigma,k,l}^{N,N'}\in\CB(\CH^{k,l})$ is compact and if $N\geq3N'-2$, then \linebreak$\bK_{\sigma,k,l}^{N,N'}\in\CB(\CH^{k,l},\CH^{k+N-2(N'-1),l})$.
	\end{lemma}
	\begin{proof}
		The boundedness part of the statement is clear from Lem. \ref{Lem reg Q^f} and Lem. \ref{Lem reg Q^b} as $\bk_{\sigma,k,l,\bullet}^{\mathrm{f},N,N'}$ is of order $-(N-2(N'-1))$ in $\xi$ and $|\eta|_{\cancel{a}}$ (the $-2(N'-1)$ term coming from the correction in the basis) and of order $N'$ in $\rho$: as we gain one derivative whenever one power of $\rho$ is associated to one power of $\frac{\rho}{a\xi^2+\rho|\eta|_{\cancel{a}}^2}$, we can increase the regularity of the target space provided that $N-2(N'-1)\geq N'$, that is $N\geq 3N'-2$.
		
		To show compactness, let us pick a bounded sequence $(u_t)_{t\in\N}\subset\CH^{k,l}$. Using the beginning of the proof of Lem. \ref{Lem reg Q^f}, we can check that $\R\times\CM\times\R\times\R^n\ni(x:=\ln\rho,\omega,\xi,\eta)\mapsto\bk_{\sigma,k,l,\bullet}^{N,N'}(\mathrm{e}^x,\omega,\xi,\eta)$ is a decaying Fourier symbol:
		\begin{align*}
		(|\xi|^{N+\alpha}+|\eta_j|^{N+\beta})|\p_{\omega_j}^\alpha(\rho\p_{\rho}^\beta)\bk_{\sigma,k,l,\bullet}^{N,N'}|&=\CO(\rho^{N'}).
		\end{align*}
		It follows from \cite[Thm. 4.28]{Z} (using the standard quantization $\mathrm{Op}_0$ as denoted therein) that we can extract a subsequence $(u_{t'})_{t'\in\N}$ from $(u_t)_{t\in\N}$ such that\footnote{We first make $\rho^{-l}\bK_{\sigma,k,l}^{N,N'}u_{t'}$ converge in the appropriate unweighted $L^2$ space thanks to \cite[Thm. 4.28]{Z}, then we can multiply by $\rho^l$ to get convergence in $\CH^{0,l}$.}
		\begin{align*}
		\lim_{t'\to+\infty}\bK_{\sigma,k,l}^{N,N'}u_{t'}&=u_\infty\in\CH^{0,l}.
		\end{align*}
		The proof of Lem. \ref{Lem reg Q^f} also shows that $\p_{\rho}^{k_1}\p_{\omega}^{k_2}\bK_{\sigma,k,l}^{N,N'}u=\mathrm{Op}_{\CM,\CF}[\breve{\bk}_{\sigma,k,l,\bullet}^{N,N'}]u$ with $k_1+|k_2|=k$, where $\R\times\CM\times\R\times\R^n\ni(x:=\ln\rho,\omega,\xi,\eta)\mapsto\breve{\bk}_{\sigma,k,l,\bullet}^{N,N'}(\mathrm{e}^x,\omega,\xi,\eta)$ is also a decaying Fourier symbol; as a consequence, \cite[Thm. 4.28]{Z} entails that we can extract a subsequence $(u_{t''})_{t''\in\N}$ from $(u_{t'})_{t'\in\N}$ such that
		\begin{align*}
		\lim_{t''\to+\infty}\p_{\rho}^{k_1}\p_{\omega}^{k_2}\bK_{\sigma,k,l}^{N,N'}u_{t'}&=w_\infty\in\CH^{0,l}.
		\end{align*}
		But then, $\rho^{-l}\p_{\rho}^{k_1}\p_{\omega}^{k_2}v_\infty=\rho^{-l}w_\infty$ in the distribution sense; as $\rho^{-l}w_\infty\in L^2(X,\frac{\mathrm{d}\rho}{\rho}\mathrm{d}\omega)$, we discover that $\rho^{-l}\p_{\rho}^{k_1}\p_{\omega}^{k_2}v_\infty\in L^2(X,\frac{\mathrm{d}\rho}{\rho}\mathrm{d}\omega)$ as well, that is $v_\infty\in\CH^{k,l}$. This completes the proof.
	\end{proof}

	We now construct a global boundary left parametrix for $\bP_{\sigma}$. The construction is very similar to the one for the right parametrix above; as we will not explicitly use the left parametrix in the sequel, we will only give the main lines of its construction. Let again $N,N'\in\N\setminus\{0\}$ such that $N>2(N'-1)$ and assume that $0\notin\bigcup_{j=1}^{N'}\mtp_{\sigma,k+j-1,l}(\R)$. Then define
	\begin{align*}
	\widetilde{\bQ}_{\sigma,k,l}^{N,N'}&:=\widetilde{\bQ}_{\sigma,k,l}^{\mathrm{f},N}+\widetilde{\bQ}_{\sigma,k,l}^{\mathrm{b},N,N'}\in\CB(\CH^{k,l},\CX^{k+1,l})
	\end{align*}
	so that
	\begin{align*}
	\widetilde{\bQ}_{\sigma,k,l}^{N,N'}\bP_\sigma&=\chi\mathds{1}_{\CX^{k+1,l}}+\widetilde{\bK}_{\sigma,k,l}^{N,N'}.
	\end{align*}
	Above $\widetilde{\bQ}_{\sigma,k,l}^{\mathrm{f},N}$ first produces decay in the fiber, $\widetilde{\bQ}_{\sigma,k,l}^{\mathrm{b},N,N'}$ then creates decay in the basis and $\widetilde{\bK}_{\sigma,k,l}^{N,N'}$ is compact.
	
	To construct $\widetilde{\bQ}_{\sigma,k,l}^{\mathrm{f},N}$, we first let
	\begin{align*}
	&(\widetilde{\bQ}_{\sigma,k,l,\bullet}^{\mathrm{f},1}u)(\rho)\\
	&\quad:=\frac{1}{(2\pi)^{n+1}}\int_{\R_\xi}\int_{(0,+\infty)_y}\int_{\R^n_\eta}\int_{\R^n_\zeta}\left(\frac{\rho}{y}\right)^{\mathrm{i}\xi}\mathrm{e}^{\mathrm{i}(\omega-\zeta)\cdot\eta}F_{k,l}(\rho,\xi)\frac{\chi(\rho)\phi_\bullet(\omega)y}{a(y,\zeta,\sigma)\xi_{k,l}^2+y|\eta|_{\cancel{a}}^2}(y^{-l}\p_y^kv_{k,l,\bullet})(y,\zeta)\mathrm{d}\zeta\frac{\mathrm{d}y}{y}\mathrm{d}\eta\mathrm{d}\xi
	\end{align*}
	where $v_{k,l,\bullet}(\rho,\omega):=\widetilde{\phi}_{\bullet}(\omega)(\p_{\rho}^k u)(\rho,\omega)$ (once again, $\CM\cap\mathrm{Supp}(\widetilde{\phi}_\bullet)$ is identified as an open subset of $\R^n$ so that the Fourier transform with respect to $\zeta\in\CM\cap\mathrm{Supp}(\widetilde{\phi}_\bullet)$ is well defined). Then:
	\begin{align*}
	&\quad\int_{(0,+\infty)_y}\int_{\R^n_\zeta}\left(\frac{\rho}{y}\right)^{\mathrm{i}\xi}\mathrm{e}^{\mathrm{i}(\omega-\zeta)\cdot\eta}\frac{y^{-l}F_{k,l}(\rho,\xi)y}{a(y)\xi_{k,l}^2+y|\eta|_{\cancel{a}}^2}\left(\left(-\frac{a(y\p_y)^2}{y}-\sum_{j=1}^{n}\cancel{a}_j\p_{\zeta_j}^2\right)\p_y^kv_{k,l,\bullet}\right)\!(y,\eta)\mathrm{d}\zeta\frac{\mathrm{d}y}{y}\\
	&=\int_{(0,+\infty)_y}\int_{\R^n_\zeta}\left(\frac{\rho}{y}\right)^{\mathrm{i}\xi}\mathrm{e}^{\mathrm{i}(\omega-\zeta)\cdot\eta}\frac{y^{-l}F_{k,l}(\rho,\xi)}{a(y)\xi_{k,l}^2+y|\eta|_{\cancel{a}}^2}\left(\left(-a\p_yy\p_y-\sum_{j=1}^{n}\cancel{a}_j\p_{\zeta_j}^2\right)\p_y^kv_{k,l,\bullet}\right)\!(y,\eta)\mathrm{d}\zeta\mathrm{d}y\\
	%
	%
	&=\int_{(0,+\infty)_y}\int_{\R^n_\zeta}\left(\frac{\rho}{y}\right)^{\mathrm{i}\xi}\mathrm{e}^{\mathrm{i}(\omega-\zeta)\cdot\eta}\frac{y^{-l}F_{k,l}(\rho,\xi)}{a(y)\xi_{k,l}^2+y|\eta|_{\cancel{a}}^2}\left(\left(\frac{a\xi_{k,l}^2}{y}+|\eta|_{\cancel{a}}^2+\mathrm{subprincipal\ terms}\right)\p_y^kv_{k,l,\bullet}\right)\!(y,\eta)\mathrm{d}\zeta\mathrm{d}y\\
	&=\int_{(0,+\infty)_y}\int_{\R^n_\zeta}\left(\frac{\rho}{y}\right)^{\mathrm{i}\xi}\mathrm{e}^{\mathrm{i}(\omega-\zeta)\cdot\eta}y^{-l}F_{k,l}(\rho,\xi)(\p_y^kv_{k,l,\bullet})(y,\eta)\mathrm{d}\zeta\frac{\mathrm{d}y}{y}+\mathrm{remainder}.
	\end{align*}
	Above, the subprincipal terms come from commutators (we write $-a\p_yy\p_y-\sum_{j=1}^{n}\cancel{a}_j\p_{\zeta_j}^2=-\p_yy\p_ya-[a,\p_y]y\p_y-\p_yy[a,\p_y]-\sum_{j=1}^{n}(\cancel{a}_j\p_{\zeta_j}^2+[\cancel{a}_j,\p_{\zeta_j}^2])$ and so on) and the remainder term is the integral of terms of order $1$ in $y$ and $-1$ in $\xi$ and $|\eta|_{\cancel{a}}$. Using Lem. \ref{Lem F_{k,l}}, we obtain
	\begin{align*}
	\widetilde{\bQ}_{\sigma,k,l,\bullet}^{\mathrm{f},1}\bP_\sigma&=\chi\mathds{1}_{\CX^{k+1,l}}+\widetilde{\bK}_{\sigma,k,l,\bullet}^{\mathrm{f},1}
	\end{align*}
	where $\widetilde{\bK}_{\sigma,k,l}^{\mathrm{f},1}$ has the same properties as the remainder term in the last line above. That $\widetilde{\bQ}_{\sigma,k,l}^{\mathrm{f},1}\in\CB(\CH^{k,l},\CX^{k+1,l})$ then follows using the same arguments as Lem. \ref{Lem reg Q^f}. We then inductively construct $\widetilde{\bQ}_{\sigma,k,l,\bullet}^{\mathrm{f},2},\widetilde{\bQ}_{\sigma,k,l,\bullet}^{\mathrm{f},3},\ldots$ to decrease the total order in $\xi$ and $|\eta|_{\cancel{a}}$ of the remainder, then we define
	\begin{align*}
	\widetilde{\bQ}_{\sigma,k,l}^{\mathrm{f},N}&:=\sum_{j=1}^N\sum_{\bullet}\widetilde{\bQ}_{\sigma,k,l,\bullet}^{\mathrm{f},j}.
	\end{align*}

	To construct $\widetilde{\bQ}_{\sigma,k,l}^{\mathrm{b},N}$, we first let
	\begin{align*}
	&(\widetilde{\bQ}_{\sigma,k,l,\bullet}^{\mathrm{b},1}u)(\rho)\\
	&\quad:=\frac{1}{(2\pi)^{n+1}}\int_{\R_\xi}\int_{(0,+\infty)_y}\int_{\R^n_\eta}\int_{\R^n_\zeta}\left(\frac{\rho}{y}\right)^{\mathrm{i}\xi}\mathrm{e}^{\mathrm{i}(\omega-\zeta)\cdot\eta}F_{k,l}(\rho,\xi)\frac{\chi(\rho)\phi_\bullet(\omega)y\widetilde{\bk}_{\sigma,k,l,\bullet}(y,\zeta)}{\mtp_{\sigma,k,l}^1(\xi)}(y^{-l}\p_y^kv_{k,l,\bullet})(y,\zeta)\mathrm{d}\zeta\frac{\mathrm{d}y}{y}\mathrm{d}\eta\mathrm{d}\xi
	\end{align*}
	where is defined in \eqref{Eq Indicial family 1} and $\widetilde{\bk}^{\mathrm{f},N,1}_{\sigma,k,l,\bullet}(y,\zeta)$ is such that
	\begin{align*}
	\sum_{\bullet}\chi\phi_{\bullet}\mathrm{Op}_{\CM,\CF}[F_{k,l}](\widetilde{\bk}^{\mathrm{f},N,1}_{\sigma,k,l,\bullet}v_{k,l,\bullet})&=-\widetilde{\bK}^{\mathrm{f},N}_{\sigma,k,l}u
	\end{align*}
	As we can write
	\begin{align}
	\label{Eq Increment by 2 in eta}
	\bP_\sigma&=\frac{-a_0(\rho\p_\rho)^2+(b_0-2\mathrm{i}a_0)(\rho\p_\rho)-\mathrm{i}b_0-a_0+c_0}{\rho}+\CO(\p_\rho,\p_\rho^2,\rho\p_{\omega_1},\rho\p_{\omega_1}^2,\ldots,\rho\p_{\omega_n},\rho\p_{\omega_n}^2),
	\end{align}
	integrations by parts yield
	\begin{align*}
	\sum_{\bullet}\widetilde{\bQ}_{\sigma,k,l,\bullet}^{\mathrm{b},1}\bP_\sigma u&=-\widetilde{\bK}^{\mathrm{f},N}_{\sigma,k,l}u+\mathrm{remainder}
	\end{align*}
	with a remainder having as decay in $\xi$ and $|\eta|_{\cancel{a}}$ as the $\widetilde{\bk}_{\sigma,k,l,\bullet}^{\mathrm{f},N,1}$ but one more power of $y$. We then inductively construct $\widetilde{\bQ}_{\sigma,k,l,\bullet}^{\mathrm{b},N,2},\widetilde{\bQ}_{\sigma,k,l,\bullet}^{\mathrm{b},N,3},\ldots$ to increase the total order in $y$ of the remainder using the polynomials $\mtp_{\sigma,k,l,j}(\xi)$ defined in \eqref{Eq Indicial family j}; note that each iteration now increases the total order of $|\eta|_{\cancel{a}}$ by 2 since we have to cancel the remainder in \eqref{Eq Increment by 2 in eta} that does not simplify with the terms $\mtp_{\sigma,k,l,j}(\xi)$. Then we define
	\begin{align*}
	\widetilde{\bQ}_{\sigma,k,l}^{\mathrm{b},N,N'}&:=\sum_{j=1}^{N'}\sum_{\bullet}\widetilde{\bQ}_{\sigma,k,l,\bullet}^{\mathrm{b},N,j}.
	\end{align*}

	We have then:
	\begin{align*}
	\widetilde{\bQ}_{\sigma,k,l}^{N,N'}\bP_\sigma&=\chi\mathds{1}_{\CX^{k+1,l}}+\widetilde{\bK}_{\sigma,k,l}^{N,N'}.
	\end{align*}
	Compactness of $\widetilde{\bK}_{\sigma,k,l}^{N,N'}\in\CB(\CX^{k+1,l})$ then follows again from \cite[Thm. 4.28]{Z} (using the quantization $\mathrm{Op}_1$ therein) with the arguments of Lem. \ref{Lem Compactness K} above. Furthermore, using the well known property of the Mellin transform $\CM[y\mapsto y^{N'}U(y)](\xi)=\CM[U](\xi+\mathrm{i}N')$, the extra term $y^{N'}$ from the basis correction can be transformed into $\rho^{N'}$ upon shifting the contour intgral in $\xi$ from $\R$ to $\R-\mathrm{i}N'$, ensuring the Sobolev regularity $\CH^{k+1+N-2(N'-1),l}$; this amounts to changing $\xi_{k,l}$ to $\xi_{k+N',l}$, whence the extra assumption $0\notin\bigcup_{0\leq\alpha\leq N'}\bigcup_{j=1}^{N'}\mtp_{\sigma,k+j-1,l}(\R-\mathrm{i}\alpha)$. As a result, $\widetilde{\bK}^{N,N'}_{\sigma,k,l}\in\CB(\CX^{k+1+N-2(N'-1),l})$.
	
	Finally, the regularity in $\sigma$ statement for both the right and left parametrices is a mere consequence of the parametrix construction which uses rational functions of the coefficients of $\bP_\sigma$ (the denominators never cancel far away from the critical strips).

	\subsection{Local index 0 Fredholm property}
	\label{Local index 0 Fredholm property}
	We finally prove Prop. \ref{Prop Invertibility Q}.
	
		For the sake of definiteness, we assume that $a(\rho,\omega,\sigma)>0$ and $\cancel{a}_j(\rho,\omega,\sigma)>0$ on $\overline{X'}\times\CM$. Let $C>0$ and $N,N'\in\N\setminus\{0\}$ such that $N>2(N'-1)$ as in Prop \ref{Prop Local Parametrix}. We define $\bQ_{\sigma,k,l}^{N,N'}$ and $\widetilde{\bQ}_{\sigma,k,l}^{N,N'}$ with $\mq_{\sigma,k,l}^{\mathrm{f},1}$ replaced by $\frac{\rho}{a\xi_{k,l}^2+\rho|\eta|_{\cancel{a}}^2+\rho C}$. We claim that $\bQ_{\sigma,k,l}^{N,N'}\in\CB(\CH'^{k,l},\CX'^{k+1,l})$ and $\widetilde{\bQ}_{\sigma,k,l}^{N,N'}\in\CB(\CH'^{k,l},\CX'^{k+1,l})$ are invertible for $C$ sufficiently large. To see this, we use the above computations to obtain
		\begin{align}
		\label{Eq P+C}
		(\bP_\sigma+\rho C)\bQ_{\sigma,k,l}^{N,N'}&=\mathds{1}_{\CH'^{k,l}}+\bK_{\sigma,k,l}^{N,N'},\\
		\label{Eq P+C BIS}
		\widetilde{\bQ}_{\sigma,k,l}^{N,N'}(\bP_\sigma+\rho C)&=\mathds{1}_{\CX'^{k+1,l}}+\widetilde{\bK}_{\sigma,k,l}^{N,N'}
		\end{align}
		for some $\bK_{\sigma,k,l}^{N,N'}\in\CB(\CH'^{k,l})$ and $\widetilde{\bK}_{\sigma,k,l}^{N,N'}\in\CB(\CX'^{k+1,l})$ with operator norms of order $o_{C\to+\infty}(1)$. Indeed, let us we write $\bK_{\sigma,k,l}^{N,N'}u=\sum_\bullet\mathrm{Op}_{\CM,\CF}[\bk_{\sigma,k,l,\bullet}]v_{k,l,\bullet}$ with $v_{k,l,\bullet}(\rho,\omega):=\chi(y)\widetilde{\phi}_\bullet(\omega)(\rho^{-l}\p_\rho^k u)(\rho,\omega)$ (recall that $\chi_{\vert X'}=\mathds{1}_{X'}$; here we implicitly use a Sobolev extension of $u$ to $\mathrm{Supp}(\chi)$ so that we can consider $u$ as a function defined on $(0,+\infty)$ vanishing on a neighborhood of $+\infty$).
		
		On the one hand,
		\begin{align*}
		&\quad\bigg|\int_{\R_\xi}\int_{(0,+\infty)_y}\int_{\R^n_\eta}\int_{\R^n_\zeta}\mathds{1}_{\rho/y\leq1}\left(\frac{\rho}{y}\right)^{\mathrm{i}\xi}\mathrm{e}^{\mathrm{i}(\omega-\zeta)\cdot\eta}\bk_{\sigma,k,l,\bullet}(\rho,y,\omega,\zeta,\xi,\eta)v_{k,l,\bullet}(y,\zeta)\mathrm{d}\zeta\frac{\mathrm{d}y}{y}\mathrm{d}\eta\mathrm{d}\xi\bigg|\\
		&\leq\int_{\R_\xi}\int_{(0,+\infty)_y}\int_{\R^n_\eta}\int_{\R^n_\zeta}\mathds{1}_{\rho/y\leq1}\left|\frac{\bk_{\sigma,k,l,\bullet}(\rho,y,\omega,\zeta,\xi,\eta)}{\rho}\right||v_{k,l,\bullet}(y,\zeta)|\mathrm{d}\zeta\frac{\rho}{y}\mathrm{d}y\mathrm{d}\eta\mathrm{d}\xi\\
		&\leq\int_{\R_\xi}\int_{\R^n_\eta}\sup_{\rho/y\leq1/2}\left|\frac{\bk_{\sigma,k,l,\bullet}(\rho,y,\omega,\zeta,\xi,\eta)}{\rho}\right|\mathrm{d}\eta\mathrm{d}\xi\sqrt{\int_{(0,+\infty)_y}\int_{\R^n_\zeta}|v_{k,l,\bullet}(y,\zeta)|^2\mathrm{d}\zeta\frac{\mathrm{d}y}{y}}\sqrt{\int_{\mathrm{Supp}(\chi)}y\mathrm{d}y}
		\end{align*}
		by Cauchy-Schwartz inequality; since $\bk_{\sigma,k,l,\bullet}/\rho$ is uniformly bounded in $(\rho,\omega)\in[0,+\infty)\times\CM$ and analytic in $\xi$ and in $\eta$, we can use the residue theorem (choosing $N$ large enough for integrability) to see that the integral over $X'\times\CM$ of the square of the left hand side above is bounded multiple of $\|u\|_{\CH'^{k,l}}$ of order $o_{C\to+\infty}(1)$ (integrals with respect to $\rho$ and $\omega$ are harmless for the estimate as we integrate over bounded domains).
		
		On the other hand, when $\rho/y>1$, we can use the residue theorem for the integral over $\R_\xi$ and sum up over all the residues of $\xi\mapsto\left(\frac{\rho}{y}\right)^{\mathrm{i}\xi}\bk_{\sigma,k,l,\bullet}(\rho,y,\omega,\zeta,\xi,\eta)$ in the upper half plane $\C^+$ (since $\left|\left(\frac{\rho}{y}\right)^{\mathrm{i}\xi}\right|$ decays exponentially fast as a function of $\xi$ with $\mathrm{Im}(\xi)>0$ when $\rho/y>1$); the residues comes from the falling factorial term $(\mathrm{i}\xi+k+l)_{(k)}$ in the denominator of $F_{k,l}$ (whose the nearest pole to $\R$ is $\xi=\mathrm{i}/2$), from the denominator $\frac{\rho}{a\xi_{k,l}^2+\rho|\eta|_{\cancel{a}}^2+\rho C}$ and from the indicial polynomials (both yielding poles with imaginary part greater or equal to $\mathrm{i}(k+l)$). Since $k+l>0$ by assumption, we can factor out of the sum of the residues a term $y^{\epsilon}$ with $\epsilon>0$:
		\begin{align*}
		&\quad\bigg|\int_{\R_\xi}\int_{(0,+\infty)_y}\int_{\R^n_\eta}\int_{\R^n_\zeta}\mathds{1}_{\rho/y>1}\left(\frac{\rho}{y}\right)^{\mathrm{i}\xi}\mathrm{e}^{\mathrm{i}(\omega-\zeta)\cdot\eta}\bk_{\sigma,k,l,\bullet}(\rho,y,\omega,\zeta,\xi,\eta)v_{k,l,\bullet}(y,\zeta)\mathrm{d}\zeta\frac{\mathrm{d}y}{y}\mathrm{d}\eta\mathrm{d}\xi\bigg|\\
		&\leq\int_{(0,+\infty)_y}\int_{\R^n_\eta}\int_{\R^n_\zeta}\mathds{1}_{\rho/y>1}\sum\left|\underset{\,\C^+}{\mathrm{Res}}\left(\xi\mapsto\left(\frac{\rho}{y}\right)^{\mathrm{i}\xi}\frac{\bk_{\sigma,k,l,\bullet}(\rho,y,\omega,\zeta,\xi,\eta)}{y^{\epsilon}}\right)\right||v_{k,l,\bullet}(y,\zeta)|\mathrm{d}\zeta\frac{y^{\epsilon}}{y}\mathrm{d}y\mathrm{d}\eta\\
		&\leq\int_{\R^n_\eta}\sup_{\rho/y>1}\sum\left|\underset{\,\C^+}{\mathrm{Res}}\left(\xi\mapsto\left(\frac{\rho}{y}\right)^{\mathrm{i}\xi}\frac{\bk_{\sigma,k,l,\bullet}(\rho,y,\omega,\zeta,\xi,\eta)}{y^{\epsilon}}\right)\right|\mathrm{d}\eta\sqrt{\int_{(0,+\infty)_y}\int_{\R^n_\zeta}|v_{k,l,\bullet}(y,\zeta)|^2\mathrm{d}\zeta\frac{\mathrm{d}y}{y}}\\
		&\qquad\qquad\qquad\qquad\qquad\qquad\qquad\qquad\qquad\qquad\qquad\qquad\qquad\qquad\qquad\times\sqrt{\int_0^\rho\frac{y^{2\epsilon}}{y}\mathrm{d}y}.
		\end{align*}
		It remains to integrate over $X'\times\CM$ (putting $y^\epsilon$ into the denominator of $\bk_{\sigma,k,l,\bullet}$ takes away the term $\rho^{-\epsilon}$ from $\left(\frac{\rho}{y}\right)^{\mathrm{i}\xi_*}=\rho^{-\epsilon}\left(\frac{\rho}{y}\right)^{\mathrm{i}\xi_*-\epsilon}y^\epsilon$ where $\xi_*\in\C^+$ is a pole; since the integral over $(0,\rho)$ above is equal to $\rho^{2\epsilon}/(2\epsilon)$, integrability near 0 is ensured). As above, the denominator of $\bk_{\sigma,k,l,\bullet}$ makes that we get a multiple of $\|u\|_{\CH'^{k,l}}$ of order $o_{C\to+\infty}(1)$. 
		
		In the same spirit, we can establish the same estimate for the $\CH$ norm of derivatives of $\bK_{\sigma,k,l}u$, concluding that $\|\bK_{\sigma,k,l}\|_{\CB(\CH'^{k,l})}=o_{C\to+\infty}(1)$. Moreover, if $\widetilde{\bK}_{\sigma,k,l}^{N,N'}u=\sum_\bullet\mathrm{Op}_{\CM,\CF}[\widetilde{\bk}_{\sigma,k,l,\bullet}]v_{k,l,\bullet}$ then $\widetilde{\bk}_{\sigma,k,l,\bullet}/y$ is uniformly bounded so that we can multiply by $y$ in the integral with respect to $y$, removing once again the troublesome $\frac{1}{y}$ term before using Cauchy-Schwartz inequality (this time we do not need to integrate by parts using $\p_\xi$). The conclusion is $\|\widetilde{\bK}_{\sigma,k,l}\|_{\CB(\CX'^{k+1,l})}=o_{C\to+\infty}(1)$ as well.
		
		Now we choose $C$ lare enough (depending on $|\sigma|$, $k$ and $|l|$) so that the right hand sides in \eqref{Eq P+C} and \eqref{Eq P+C BIS} are invertible: this proves that $\bP_\sigma+C\in\CB(\CX'^{k+1,l},\CH'^{k,l})$ is invertible. Back into \eqref{Eq P+C}, we get the invertibility of $\bQ_{\sigma,k,l}\in\CB(\CH'^{k,l},\CX'^{k+1,l})$.
		
		Finally, Sect. \ref{Global boundary parametrices} tells us that
		\begin{align*}
		\bP_\sigma\bQ_{\sigma,k,l}^{N,N'}&=\mathds{1}_{\CH'^{k,l}}+\bK_{\sigma,k,l}^{N,N'}
		\end{align*}
		with $\bQ_{\sigma,k,l}^{N,N'}\equiv\bQ_{\sigma,k,l}^{N,N'}(C)$ still defined with the aforementioned modification of $\mq_{\sigma,k,l}^{\mathrm{f},1}$ and $\bK_{\sigma,k,l}^{N,N'}\equiv\bK_{\sigma,k,l}^{N,N'}(C)$ compact on $\CH'^{k,l}$ (but not necessarily invertible as $C\to+\infty$). Since both $\bQ_{\sigma,k,l}^{N,N'}$ and $\mathds{1}_{\CH'^{k,l}}+\bK_{\sigma,k,l}^{N,N'}$ are index 0 Fredholm operators, so is $\bP_\sigma$. The proof is complete.

	\section{The charged KG operator near cosmological accelerating and rotating charged black holes}
	\label{The charged KG operator near cosmological accelerating and rotating charged black holes}
	We apply in this section the results obtained in Sect. \ref{General boundary parametrix construction} to charged KG scalar fields near cosmological accelerating and rotating charged black holes.

	\subsection{The metric in Boyer-Lindquist coordinates}
	\label{The metric in Boyer-Lindquist coordinates}
	We follow \emph{cf.} \cite{PG} and \cite[Sect. 3.2 \& 3.3]{H1}. Let $M,\Lambda>0$, $a,Q,\alpha\in\R$ and define the horizon function:
	\begin{align*}
	\mu(r)&:=(r^2+a^2)\left(1-\frac{\Lambda r^2}{3}\right)+\left(-2Mr+Q^2\left(1+\frac{\Lambda a^2}{3}\right)^2\right)(1-\alpha^2r^2).
	\end{align*}
	On $\mathbb{R}_t\times(r_{-},r_{+})_r\times\mathbb{S}^2_{\theta,\varphi}$ with $r_{\pm}>0$ to be fixed below, we define the DSKN metric in Boyer-Lindquist coordinates
	\begin{align}
	\label{Eq metric far from poles}
	g&=\frac{1}{\Omega^2}\left(\frac{\mu}{(1+\lambda)^2\rho^2}\big(\mathrm{d}t-a\sin^2\theta\mathrm{d}\varphi\big)^2-\frac{\kappa\sin^2\theta}{(1+\lambda)^2\rho^2}\big(a\mathrm{d}t-(r^2+a^2)\mathrm{d}\varphi\big)^2-\rho^2\left(\frac{\mathrm{d}r^2}{\mu}+\frac{\mathrm{d}\theta^2}{\kappa}\right)\right)
	\end{align}
	where
	\begin{align*}
	\Omega&=1-\alpha r\cos\theta,\qquad\quad\lambda=\frac{\Lambda a^2}{3},\qquad\quad\rho^2=r^2+a^2\cos^2\theta,\\
	\kappa&=1-2\alpha M\cos\theta+\left(\alpha^2(a^2+Q^2(1+\lambda)^2)+\lambda\right)\cos^2\theta.
	\end{align*}
	(Notice that we use the coordinates $\bar{t}$ and $\bvarphi$ in \cite[eq. (4)]{PG} while $e=Q(1+\lambda)$ and $g=0$.) If
	\begin{align*}
	A&:=-\frac{Qr}{\rho^2}(\mathrm{d}t-a\sin^2\theta\mathrm{d}\varphi),
	\end{align*}
	then $(g,A)$ solves the coupled Einstein-Maxwell equation with electromagnetic potential $\mathrm{d}A$. The volume form $\mathrm{dVol}_{g}$ associated to $g$ is given by:
	\begin{align*}
	\mathrm{dVol}_{g}&=\frac{1}{\Omega^4}\frac{\rho^2\sin\theta}{(1+\lambda)^2}\mathrm{d}t\mathrm{d}r\mathrm{d}\theta\mathrm{d}\varphi.
	\end{align*}
	The coefficients of the inverse metric $g^{-1}$ are:
	\begin{align*}
	g^{tt}&=-\Omega^2\frac{(1+\lambda)^2(\mu a^2\sin^2\theta-(r^2+a^2)^2\kappa)}{\mu\kappa\rho^2},\qquad g^{t\varphi}=-\Omega^2\frac{a(1+\lambda)^2(\mu-(r^2+a^2)\kappa)}{\mu\kappa\rho^2},\\
	g^{\varphi\varphi}&=-\Omega^2\frac{(1+\lambda)^2(\mu-a^2\kappa\sin^2\theta)}{\mu\kappa\rho^2\sin^2\theta},\qquad g^{rr}=-\Omega^2\frac{\mu}{\rho^2},\qquad g^{\theta\theta}=-\Omega^2\frac{\kappa}{\rho^2}.
	\end{align*}

	This metric describes the exterior of a cosmological, accelerating and rotating charged black hole. $M$ is the mass of the black hole, $Q$ is its electric charge, $a$ is its angular velocity, $\alpha$ its acceleration and $\Lambda$ is the cosmological constant. The polar coordinates $(\theta,\varphi)$ are chosen such that the rotation of the black hole is azimuthal. When $\alpha=0$, $g$ boils down to the De Sitter-Kerr-Newman metric.
	\begin{assumptions}
	\label{Assumptions metric g}
	\begin{enumerate}
		\item [1.] We assume the there exist real numbers $0<r_{-}<r_{+}$ such that $\mu$ is positive in $(r_{-},r_{+})$ and $\mu(r_{\pm})=0$.
		\item [2.] We assume that $\Omega,\kappa>0$ on $[r_{-},r_{+}]_r\times\mathbb{S}^2_{\theta,\varphi}$.
	\end{enumerate}
	\end{assumptions}
	\begin{remark}
	When $a=\alpha=0$, \cite[eq. (5)]{Mo17} or equivalently \cite[eq. (3.7)]{H1} gives necessary and sufficient conditions on $M,\Lambda$ and $|Q|$ for $\mu$ to have four simple real roots $r_{n}<0<r_{c}<r_{-}<r_{+}$ and to be positive in $(r_{-},r_{+})$; the implicit function theorem ensures that $\mu$ still has four simple real roots, still denoted by $r_{n}<0<r_{c}<r_{-}<r_{+}$, when $a,\alpha$ are small enough, and therefore Assumption 1 is verified in this setting.
	
	As for $\kappa$, we can check that $\kappa=0$ if and only if
	\begin{align*}
	\cos\theta&=\frac{\alpha M}{\alpha^2(a^2+Q^2(1+\lambda)^2)+\lambda}\left(1\pm\sqrt{1-\frac{\alpha^2(a^2+Q^2(1+\lambda)^2)+\lambda}{\alpha^2M^2}}\right)
	\end{align*}
	when $\alpha\neq0$. Consequently, Assumption 2 is verified when $\alpha$ is small enough with respect to $r_{-}$ (in particular, it is sufficient to assume that $|\alpha|\ll1/2M$ in the limits $M\to+\infty$ and $\Lambda,a,Q\to0$, as then $r_{-}\to2M$).
	\end{remark}
	\begin{remark}
	\label{Remark kappa}
		It will be useful to notice that $\kappa\equiv\kappa(\cos\theta)=\kappa(0)+\widetilde{\kappa}_{N}\sin^2\theta$ where
		\begin{align*}
		\widetilde{\kappa}_{N}(\theta)&:=\frac{\kappa(\theta)-\kappa(0)}{\sin^2\theta}=\frac{2\alpha M(1-\cos\theta)}{\sin^2\theta}-\left(\alpha^2(a^2+Q^2(1+\lambda)^2)+\lambda\right)
		\end{align*}
		is smooth on a neighborhood of the north pole $\{\theta=0\}$. We similarly define
		\begin{align*}
		\widetilde{\kappa}_{S}(\theta)&:=\frac{\kappa(\theta)-\kappa(\pi)}{\sin^2\theta}=-\frac{2\alpha M(1+\cos\theta)}{\sin^2\theta}-\left(\alpha^2(a^2+Q^2(1+\lambda)^2)+\lambda\right)
		\end{align*}
		near the south pole $\{\theta=\pi\}$.
	\end{remark}
	The set	$\R_t\times\{r_{-}\}_r\times\mathbb{S}^2_{\theta,\varphi}$ is the event (or black hole) horizon, $\R_t\times\{r_{+}\}_r\times\mathbb{S}^2_{\theta,\varphi}$ is the cosmological horizon ($r_{+}$ thus being the radius of the observable universe -- this model describes no expansion of the universe); in the case where $r_{n,a}$ and $r_{c,a}$ also exist, $\R_t\times\{r_{c,a}\}_r\times\mathbb{S}^2_{\theta,\varphi}$ is the Cauchy (or inner) horizon while the negative root $r_{n,a}$ is associated to no physical set.

	\subsection{The metric in $*$-coordinates}
	\label{The metric in $*$ coordinates}
	The expression \eqref{Eq metric far from poles} degenerates at $\{r_{-}\}$, $\{r_{+}\}$ as well as at the poles of $\mathbb{S}^2$ due to coordinate singularities there.
	
	We first smoothen the metric coefficients near the horizons $r_{\pm}$ following \cite[Sect. 3.2]{HV}. We introduce the new coordinates $t_*:=t-T(r)$ and $\varphi_*:=\varphi-\Phi(r)$ where $T,\Phi$ satisfy near $r_{\pm}$:
	\begin{align*}
	T'(r)&:=\pm\left(\frac{(1+\lambda)(r^2+a^2)}{\mu(r)}+c_{\pm}(r)\right),\qquad\Phi'(r):=\pm\left(\frac{a(1+\lambda)}{\mu(r)}+\widetilde{c}_{\pm}(r)\right).
	\end{align*}
	To explicit the functions $c_{\pm}$ and $\widetilde{c}_{\pm}$, we first introduce
	\begin{align*}
	\mathfrak{c}^{-2}&:=\max\{\mu_0(r)\ \vert\ r\in(r_{-,0},r_{+,0})\}=:\mu_0(\mathfrak{r}),\qquad\nu(r):=\begin{cases}
	\sqrt{1-\mu_0(r)\mathfrak{c}^2}&\text{if $r\leq\mathfrak{r}$,}\\
	-\sqrt{1-\mu_0(r)\mathfrak{c}^2}&\text{if $r\geq\mathfrak{r}$}.
	\end{cases}
	\end{align*}
	Observe that $\nu$ is smooth in a neighborhood of $[r_{-,0},r_{+,0}]$. Then taking $\varpi\in\CC^\infty(\R,[0,1])$ such that $\varpi(r)=1$ if $r\leq r_1$ and $\varpi(r)=0$ if $r\geq r_2$ for some $r_{-}<r_1<r_2<r_{+}$, we let:
	\begin{align*}
	c_{-}(r)&:=\frac{-1\mp\nu(r)}{\mu(r)},\qquad c_{+}(r):=-\left(\frac{2(1+\lambda)(r^2+a^2)}{\mu(r)}+\frac{-1\mp\nu(r)}{\mu(r)}\right)\varpi(r)+\nu(r)(1-\varpi(r)),\\
	\widetilde{c}_{-}(r)&:=0,\qquad\widetilde{c}_{+}(r):=-\frac{2a(1+\lambda)}{\mu(r)}\varpi(r).
	\end{align*}
	This ensures that $T,\Phi\in\CC^\infty((r_{-},r_{+}),\R)$ are well-defined (up to an additive constant) and the expressions of $T'$ and $\Phi'$ coincide in $[r_1,r_2]$. In the $(t_*,r,\theta,\varphi_*)$ coordinates with $\theta\in(0,\pi)$, we have $\mathrm{dVol}_g=\frac{\rho^2\sin\theta}{\Omega^4(1+\lambda)^2}\mathrm{d}t_*\mathrm{d}r\mathrm{d}\theta\mathrm{d}\varphi_*$ and the new coefficients of the inverse metric read ('$+$' sign for $r>r_1$, '$-$' sign for $r<r_2$, \emph{cf.} \cite[eq. (3.17)]{HV}):
	\begin{align*}
	\Omega^{-2}\rho^2g^{t_*t_*}&=-\mu c_{\pm}^2-2(1+\lambda)(r^2+a^2)c_{\pm}-\frac{(1+\lambda)^2a^2\sin^2\theta}{\kappa},\\
	\Omega^{-2}\rho^2g^{t_*r}&=\pm\mu c_{\pm}\pm(1+\lambda)(r^2+a^2),\\
	\Omega^{-2}\rho^2g^{t_*\varphi_*}&=-\mu c_{\pm}\widetilde{c}_{\pm}-a(1+\lambda)c_{\pm}-(1+\lambda)(r^2+a^2)\widetilde{c}_{\pm}-\frac{a(1+\lambda)^2}{\kappa},\\
	\Omega^{-2}\rho^2g^{r\varphi_*}&=\pm\mu\widetilde{c}_{\pm}\pm a(1+\lambda),\\
	\Omega^{-2}\rho^2g^{\varphi_*\varphi_*}&=-\mu\widetilde{c}_{\pm}^2-2a(1+\lambda)\widetilde{c}_{\pm}-\frac{(1+\lambda)^2}{\kappa\sin^2\theta}.
	\end{align*}

	In order to remove the singularities at the poles of $\mathbb{S}^2$, we set\footnote{When $\alpha=0$, no change of variable is needed as it is explained below eq. (3.17) in \cite{HV}.}
	\begin{align*}
	\bvarphi_{*,N}&:=\frac{\kappa(0)}{1+\lambda}\varphi_*\quad\text{near the north pole},\qquad\bvarphi_{*,S}:=\frac{\kappa(\pi)}{1+\lambda}\varphi_*\quad\text{near the south pole},
	\end{align*}
	then
	\begin{align*}
	x_{*,\bullet}&:=\sin\theta\cos\bvarphi_{*,\bullet},\qquad y_{*,\bullet}:=\sin\theta\sin\bvarphi_{*,\bullet},\qquad\bullet\in\{N,S\}
	\end{align*}
	near the appropriate pole. \textbf{We henceforth work near the north pole} as computations are analogous near the south pole (we essentially replace $\kappa(0)$ by $\kappa(\pi)$); we omit the subscript $N$ to lighten notations. We have
	\begin{align*}
	\begin{pmatrix}
	\mathrm{d}x_{*}\\[1mm]\mathrm{d}y_{*}
	\end{pmatrix}&=\begin{pmatrix}
	\cos\theta\cos\bvarphi_{*}&-y_{*}\\[1mm]\cos\theta\sin\bvarphi_{*}&x_{*}
	\end{pmatrix}\begin{pmatrix}
	\mathrm{d}\theta\\[1mm]\mathrm{d}\bvarphi_{*}
	\end{pmatrix},\qquad\begin{pmatrix}
	\mathrm{d}\theta\\[1mm]\mathrm{d}\bvarphi_{*}
	\end{pmatrix}=\begin{pmatrix}
	\frac{\cos\bvarphi_{*}}{\cos\theta}&\frac{\sin\bvarphi_{*}}{\cos\theta}\\[1mm]-\frac{\sin\bvarphi_{*}}{\sin\theta}&\frac{\cos\bvarphi_{*}}{\sin\theta}
	\end{pmatrix}\begin{pmatrix}
	\mathrm{d}x_{*}\\[1mm]\mathrm{d}y_{*}
	\end{pmatrix}
	\end{align*}
	and $\mathrm{dVol}_g=\frac{\rho2}{\Omega^4\kappa(0)(1+\lambda)\cos\theta}\mathrm{d}t_*\mathrm{d}r\mathrm{d}x_*\mathrm{d}y_*$. 
	%
	%
	%
	%
	We compute (using Rmk. \ref{Remark kappa}) the new inverse metric coefficients:
	\begin{align*}
	\Omega^{-2}\rho^2g^{t_*x_*}&=\left(\mu c_{\pm}\widetilde{c}_{\pm}+a(1+\lambda)c_{\pm}+(1+\lambda)(r^2+a^2)\widetilde{c}_{\pm}+\frac{a(1+\lambda)^2}{\kappa}\right)\frac{\kappa(0)}{1+\lambda}y_*,\\
	\Omega^{-2}\rho^2g^{t_*y_*}&=-\left(\mu c_{\pm}\widetilde{c}_{\pm}+a(1+\lambda)c_{\pm}+(1+\lambda)(r^2+a^2)\widetilde{c}_{\pm}+\frac{a(1+\lambda)^2}{\kappa}\right)\frac{\kappa(0)}{1+\lambda}x_*,\\
	\Omega^{-2}\rho^2g^{rx_*}&=\mp\left(\mu\widetilde{c}_{\pm}+a(1+\lambda)\right)\frac{\kappa(0)}{1+\lambda}y_*,\\
	\Omega^{-2}\rho^2g^{ry_*}&=\pm\left(\mu\widetilde{c}_{\pm}+a(1+\lambda)\right)\frac{\kappa(0)}{1+\lambda}x_*,\\
	%
	\Omega^{-2}\rho^2g^{x_*x_*}&=-\frac{\kappa(0)^2}{\kappa}+\frac{\kappa^2-(\kappa+\kappa(0))\widetilde{\kappa}_N}{\kappa}x_{*}^2-\frac{\kappa(0)^2(\mu\widetilde{c}_{\pm}^2+a(1+\lambda)\widetilde{c}_{\pm})}{(1+\lambda)^2}y_{*}^2,\\
	%
	\Omega^{-2}\rho^2g^{y_*y_*}&=-\frac{\kappa(0)^2}{\kappa}+\frac{\kappa^2-(\kappa+\kappa(0))\widetilde{\kappa}_N}{\kappa}y_{*}^2-\frac{\kappa(0)^2(\mu\widetilde{c}_{\pm}^2+a(1+\lambda)\widetilde{c}_{\pm})}{(1+\lambda)^2}x_{*}^2,\\
	%
	\Omega^{-2}\rho^2g^{x_*y_*}&=\frac{\kappa^2-(\kappa+\kappa(0))\widetilde{\kappa}_N}{\kappa}x_{*}y_{*}+\frac{\kappa(0)^2(\mu\widetilde{c}_{\pm}^2+a(1+\lambda)\widetilde{c}_{\pm})}{(1+\lambda)^2}x_{*}y_{*}.
	\end{align*}
	Since $\cos\theta=\sqrt{1-x_{*}^2-y_{*}^2}$ is smooth near $\{x_{*}=y_{*}=0\}$, $g$ is indeed smooth and non degenerate on $\R\times(r_{-},r_{+})\times\mathbb{S}^2$.

	\subsection{Spectral family}
	\label{Spectral family}
	Fix $\ell\in\Z$ and $k\in\N\setminus\{0\}$. Let $L^2:=L^2((r_{-},r_{+})\times\mathbb{S}^2,\mathrm{d}r\mathrm{d}\omega)$, $H^k\equiv H^k(\ell):=H^k((r_{-},r_{+})\times\mathbb{S}^2,\mathrm{d}r\mathrm{d}\omega)\cap\ker(D_\varphi+\ell)$ with norm\footnote{Far away from the poles of $\mathbb{S}^2$, we can take $(\omega_1,\omega_2)=(\theta,\varphi)$ while near the north/south pole, we can take $(\omega_1,\omega_2)=(x_{*,N/S},y_{*,N/S})$.}
	\begin{align*}
	\|u\|_{H^k}^2&:=\|u\|_{L^2}^2+\|\p_r^ku\|_{L^2}^2+\sum_{j=1}^2\|\p_{\omega_j}^ku\|_{L^2}^2.
	\end{align*}
	We define the (harmonic) spectral family associated to the charged KG operator
	\begin{align*}
	\bP_{\sigma}&:=\Omega^{-2}\rho^2\mathrm{e}^{\mathrm{i}(\sigma t_*+\ell\varphi+qR(r))}(\CBox+m^2)\mathrm{e}^{-\mathrm{i}(\sigma t_*+\ell\varphi+qR(r))}
	\end{align*}
	where
	\begin{align*}
	\CBox&=(\nabla_\gamma-\mathrm{i}qA_\gamma)(\nabla^\gamma-\mathrm{i}qA^\gamma).
	\end{align*}
	Above $\sigma\in\C$ and $R\in\CC^\infty((r_{-},r_{+}),\R)$ is defined (up to an additive constant) by
	\begin{align*}
	R'(r)&=\frac{Qr}{\rho^2}(T'(r)-a\sin^2\theta\Phi'(r)),
	\end{align*}
	removing thereby the singular term $A_r\equiv A_r(r)$ from the potential $A$ written in the $(t_*,r,\theta,\varphi_*)$ coordinates (thus $\p_r-\mathrm{i}qA_r$ has become $\p_r$); notice that the term $-\mathrm{i}\Omega^2(g^{t_*r}\sigma+g^{r\varphi_*}\ell)\p_r\!\left(\frac{\rho^2}{\Omega^4}\right)$ is missing in \cite{BeHa20}, where $\alpha=0$ and thus $\Omega=1$). We realize $\bP_\sigma$ on the following domain:
	\begin{align*}
	\CX^k&:=\{u\in H^k\ \vert\ \bP_\sigma u\in H^{k-1}\},\qquad\|u\|_{\CX^k}^2:=\|u\|_{\CH^k}^2+\|\bP_\sigma u\|_{\CH^{k-1}}^2.
	\end{align*}

	In the $(t_*,r,\theta,\varphi_*)$ coordinates with $\theta\in(0,\pi)$, we have
	\begin{align*}
	A&=A_t\mathrm{d}t_*+(A_t T'(r)+A_\varphi\Phi'(r))\mathrm{d}r+A_\varphi\mathrm{d}\varphi=:A_{t_*}\mathrm{d}t_*+A_r\mathrm{d}r+A_{\varphi_*}\mathrm{d}\varphi_*
	\end{align*}
	and (using $\p_{\varphi_*}=\p_\varphi$)
	%
	%
	\begin{align*}
	\bP_{\sigma}&=\Omega^2D_r\frac{\mu}{\Omega^2}D_r+\frac{\Omega^2}{\sin\theta}D_{\theta}\frac{\kappa\sin\theta}{\Omega^2}D_\theta+\frac{\rho^2}{\Omega^2}\left[g^{t_*r}(\sigma+qA_{t_*})+g^{r\varphi_*}(\ell+qA_{\varphi_*}),D_r\right]_{+}\\
	&\quad-\frac{\rho^2}{\Omega^2}g^{t_*t_*}(\sigma+qA_{t_*})^2-\frac{\rho^2}{\Omega^2}g^{\varphi_*\varphi_*}(\ell+qA_{\varphi_*})^2-2\frac{\rho^2}{\Omega^2}g^{t_*\varphi_*}(\sigma+qA_{t_*})(\ell+qA_{\varphi_*})\\
	&\quad-\mathrm{i}\Omega^2(g^{t_*r}(\sigma+qA_{t_*})+g^{r\varphi_*}(\ell+qA_{\varphi_*}))\p_r\!\left(\frac{\rho^2}{\Omega^4}\right)+m^2\frac{\rho^2}{\Omega^2}.
	\end{align*}

	In the $(t_*,r,x_*,y_*)$ coordinates with $x_*^2+y_*^2<1$, the potential vector $A$ becomes:
	\begin{align*}
	A
	&=A_{t_*}\mathrm{d}t_*+A_r\mathrm{d}r-A_{\varphi_*}\frac{1+\lambda}{\kappa(0)}\frac{\sin\bvarphi_*}{\sin\theta}\mathrm{d}x_*+A_{\varphi_*}\frac{1+\lambda}{\kappa(0)}\frac{\cos\bvarphi_*}{\sin\theta}\mathrm{d}y_*\\
	&=:A_{t_*}\mathrm{d}t_*+A_r\mathrm{d}r+A_{x_*}\mathrm{d}x_*+A_{y_*}\mathrm{d}y_*.
	\end{align*}
	Note that $A_{x_*}$ and $A_{y_*}$ are smooth near the poles thanks to a term $\sin^2\theta$ in $A_{\varphi_*}=A_\varphi$. The spectral family reads (\emph{cf.} App. \ref{App computation P_sigma} for details):
	\begin{align*}
	\bP_\sigma&=-\frac{\rho^2}{\Omega^2}(D_{x_{*,\bullet}}-qA_{x_{*,\bullet}})(g^{{x_{*,\bullet}}{x_{*,\bullet}}}(D_{x_{*,\bullet}}-qA_{x_{*,\bullet}})+g^{{x_{*,\bullet}}{y_{*,\bullet}}}(D_{y_{*,\bullet}}-qA_{y_{*,\bullet}}))\\
	&\quad-\frac{\rho^2}{\Omega^2}(D_{y_{*,\bullet}}-qA_{y_{*,\bullet}})(g^{{x_{*,\bullet}}{y_{*,\bullet}}}(D_{x_{*,\bullet}}-qA_{x_{*,\bullet}})+g^{{y_{*,\bullet}}{y_{*,\bullet}}}(D_{y_{*,\bullet}}-qA_{y_{*,\bullet}}))\\
	&\quad+\Omega^2D_r\frac{\mu}{\Omega^2}D_r+\frac{\rho^2}{\Omega^2}[g^{t_*r}(\sigma+qA_{t_*})+g^{t_*\varphi_*}(\ell+qA_{\varphi_*}),D_r]_{+}\\
	&\quad+\mathrm{i}\Omega^2\cos\theta\p_{x_{*,\bullet}}\!\left(\frac{\rho^2}{\Omega^4\cos\theta}\right)(g^{{x_{*,\bullet}}{x_{*,\bullet}}}(D_{x_{*,\bullet}}-qA_{x_{*,\bullet}})+g^{{x_{*,\bullet}}{y_{*,\bullet}}}(D_{y_{*,\bullet}}-qA_{y_{*,\bullet}}))\\
	&\quad+\mathrm{i}\Omega^2\cos\theta\p_{y_{*,\bullet}}\!\left(\frac{\rho^2}{\Omega^4\cos\theta}\right)(g^{{x_{*,\bullet}}{y_{*,\bullet}}}(D_{x_{*,\bullet}}-qA_{x_{*,\bullet}})+g^{{y_{*,\bullet}}{y_{*,\bullet}}}(D_{y_{*,\bullet}}-qA_{y_{*,\bullet}}))\\
	&\quad-\mathrm{i}\Omega^2\p_r\!\left(\frac{\rho^2}{\Omega^4}\right)(g^{t_*r}(\sigma+qA_{t_*})+g^{r\varphi_*}(\ell+qA_{\varphi_*}))\\
	&\quad-\frac{\rho^2}{\Omega^2}g^{t_*t_*}(\sigma+qA_{t_*})^2-2\frac{\rho^2}{\Omega^2}g^{t_*\varphi_*}(\sigma+qA_{t_*})(\ell+qA_{\varphi_*})+m^2\frac{\rho^2}{\Omega^2}.
	\end{align*}
	\subsection{Verification of Assumptions \ref{Assumptions a,b,c...}}
	\label{Verification of Assumptions Assumptions a,b,c...}
	We now put the operator $\bP_\sigma$ into the framework of Sect. \ref{General boundary parametrix construction}.
	
	Let $\chi_\pm,\widetilde{\chi}_\pm\in\CC^\infty([r_{-},r_{+}],[0,1])$ such that $\chi_\pm(r_{\pm})=1$, $\chi_-+\chi_+=1$ and $\widetilde{\chi}_\pm\chi_\pm=\chi_\pm$. We then identify $(r_{-},r_{+})\cap\mathrm{Supp}(\widetilde{\chi}_\pm)$ with a neighborhood of $0$ on the half axis $(0,+\infty)$ using the boundary defining functions $\brho:=|r-r_{\pm}|$ (we will use bold symbols for these in order to emphasize the difference with the function $\rho$ introduced in Sect. \ref{The metric in Boyer-Lindquist coordinates}). We will abusively write $\chi_\pm(\brho)$ instead of $\chi_\pm(r_{\pm}\mp\brho)$ then consider $\chi_\pm$ as defined on all $(0,+\infty)$ by taking $\chi_\pm(\brho)=0$ whenever $\brho\geq r_{+}-r_{-}$ (and similarly for $\widetilde{\chi}_\pm$). Notice that
	\begin{align*}
	D_r&=\mp D_{\brho},\qquad\mu(r)=\mu_\pm(\brho)\brho
	\end{align*}
	where\footnote{To see this, we write $\mu(r)=-\frac{\Lambda}{3}(r-r_{n,a})(r-r_{c,a})(r-r_{-})(r-r_{+})$ to discover that $\frac{\mu(r)}{r-r_{-}}>0$ for all $r\in(r_{c,a},r_{+})$ and $\frac{\mu(r)}{r_{+}-r}>0$ for all $r\in(r_{-},+\infty)$. In the case where $\mu$ only has two real roots $0<r_{-}<r_{+}$, then we can write $\mu(r)=-\frac{\Lambda}{3}f(r)(r-r_{-})(r-r_{+})$ with $f(r)>0$ for $r\in[r_{-},r_{+}]$, and a similar argument applies.} $\mu_\pm>0$ on $\overline{\mathrm{Supp}(\chi_\pm)}$.

	Let also $\phi_\bullet,\widetilde{\phi}_\bullet\in\CC^\infty(\mathbb{S}^2,[0,1])$ with $\bullet\in\{N,E,E',S\}$ ($N$ for north pole, $E,E'$ for equator -- we need two open sets to cover this region of the sphere -- and $S$ for south pole) such that $\phi_\bullet\equiv1$ near the north pole/equator/south pole of $\mathbb{S}^2$, $\phi_N+\phi_E+\phi_{E'}+\phi_S=1$ and $\widetilde{\phi}_\bullet\phi_\bullet=\phi_\bullet$ (thus the compact smooth manifold $\CM$ in Sect. \ref{General boundary parametrix construction} is $\mathbb{S}^2$). We then identify $\mathbb{S}^2\cap\mathrm{Supp}(\widetilde{\phi}_\bullet)$ with an open set of $\R^2$. While it needs no particular care to construct the radial cut-offs, the angular cut-offs must respect some additional properties in order to ensure ellipticity of $\cancel{P}_\sigma$. For $\bullet\in\{N,S\}$, we require:
	\begin{align*}
	\phi_\bullet(x_{*,\bullet},y_{*,\bullet})&=0\qquad\text{if }x_{*,\bullet}^2+y_{*,\bullet}^2\geq1.
	\end{align*}

	On $\mathrm{Supp}(\chi_\pm)\times\mathrm{Supp}(\phi_\bullet)$ with $\bullet\in\{E,E'\}$, we have $\bP_\sigma=\brho^{-1}P_\sigma+\cancel{P}_{\!\sigma}$ where
	\begin{align*}
	P_\sigma&=a(\brho)(\brho D_{\brho})^2+b(\brho,\theta,\sigma)\brho D_{\brho}+c(\brho,\theta,\sigma),\qquad\cancel{P}_{\!\sigma}=\cancel{a}_\theta(\brho,\theta)D_\theta^2+\cancel{b}_\theta(\brho,\theta)D_\theta,
	\end{align*}
	and (the expressions of $c$ and $c_\bullet$ below have to be slightly modified, \emph{cf.} the end of the section)
	\begin{align*}
	a(\brho)&=\mu_\pm(\brho),\\
	b(\brho,\theta,\sigma)&=\mp2\frac{\rho^2}{\Omega^2}\left(g^{t_*r}(\sigma+qA_{t_*})+g^{r\varphi_*}(\ell+qA_{\varphi_*})\right)\pm\mathrm{i}\brho\Omega^2\p_r\!\left(\frac{\mu_\pm}{\Omega^2}\right),\\
	c(\brho,\theta,\sigma)&=\brho\bigg(-\mathrm{i}\Omega^2\p_r\!\left(\frac{\rho^2}{\Omega^4}\left(g^{t_*r}(\sigma+qA_{t_*})+g^{r\varphi_*}(\ell+qA_{\varphi_*})\right)\right)+m^2\frac{\rho^2}{\Omega^2}\\
	&\qquad\quad\;\;-\frac{\rho^2}{\Omega^2}g^{t_*t_*}(\sigma+qA_{t_*})^2-2\frac{\rho^2}{\Omega^2}g^{t_*\varphi_*}(\sigma+qA_{t_*})(\ell+qA_{\varphi_*})-\frac{\rho^2}{\Omega^2}g^{\varphi_*\varphi_*}(\ell+qA_{\varphi_*})^2\bigg),\\
	\cancel{a}_\theta(\brho,\theta)&=\kappa,\\
	\cancel{b}_\theta(\brho,\theta)&=-\mathrm{i}\frac{\Omega^2}{\sin\theta}\p_\theta\!\left(\frac{\sin\theta}{\Omega^2}\right).
	\end{align*}
	Notice that $a(0)$, $b(0,\theta,\sigma)$ and $c(0,\theta,\sigma)=0$ are purely radial for all $(\theta,\sigma)\in(0,\pi)\times\C$.
	
	On $\mathrm{Supp}(\chi_\pm)\times\mathrm{Supp}(\phi_\bullet)$ with $\bullet\in\{N,S\}$, we have $\bP_\sigma=\brho^{-1}P_\sigma+\cancel{P}_{\!\sigma}$ where
	\begin{align*}
	P_\sigma&=a(\brho)(\brho D_{\brho})^2+b(\brho,\omega,\sigma)\brho D_{\brho}+c_\bullet(\brho,\omega,\sigma),\\
	\cancel{P}_{\!\sigma}&=\cancel{a}_{x_{*,\bullet}}(x_{*,\bullet},y_{*,\bullet})D_{x_{*,\bullet}}^2+\cancel{b}_{x_{*,\bullet}}(\brho,x_{*,\bullet},y_{*,\bullet})D_{x_{*,\bullet}}+\cancel{a}_{y_{*,\bullet}}(x_{*,\bullet},y_{*,\bullet})D_{y_{*,\bullet}}^2+\cancel{b}_{y_{*,\bullet}}(\brho,x_{*,\bullet},y_{*,\bullet})D_{y_{*,\bullet}}
	\end{align*}
	and
	\begin{align*}
	c_\bullet(\brho,\theta,\sigma)&=\brho\bigg(-\mathrm{i}\Omega^2\p_r\!\left(\frac{\rho^2}{\Omega^4}\left(g^{t_*r}(\sigma+qA_{t_*})+g^{r\varphi_*}(\ell+qA_{\varphi_*})\right)\right)+m^2\frac{\rho^2}{\Omega^2}\\
	&\qquad\quad\;\;-\frac{\rho^2}{\Omega^2}g^{t_*t_*}(\sigma+qA_{t_*})^2-2\frac{\rho^2}{\Omega^2}g^{t_*\varphi_*}(\sigma+qA_{t_*})(\ell+qA_{\varphi_*})-\frac{\rho^2}{\Omega^2}\breve{g}^{x_{*,\bullet}y_{*,\bullet}}\ell^2\\
	&\qquad\quad\;\;-\mathrm{i}q\Omega^2\cos\theta\p_{x_{*,\bullet}}\!\left(\frac{\rho^2}{\Omega^4\cos\theta}(g^{x_{*,\bullet}x_{*,\bullet}}A_{x_{*,\bullet}}+g^{x_{*,\bullet}y_{*,\bullet}}A_{y_{*,\bullet}})\right)\\
	&\qquad\quad\;\;-\mathrm{i}q\Omega^2\cos\theta\p_{y_{*,\bullet}}\!\left(\frac{\rho^2}{\Omega^4\cos\theta}(g^{x_{*,\bullet}y_{*,\bullet}}A_{x_{*,\bullet}}+g^{y_{*,\bullet}y_{*,\bullet}}A_{y_{*,\bullet}})\right)\\
	&\qquad\quad\;\;-q^2\frac{\rho^2}{\Omega^2}\left(g^{x_{*,\bullet}x_{*,\bullet}}A_{x_{*,\bullet}}^2+2g^{x_{*,\bullet}y_{*,\bullet}}A_{x_{*,\bullet}}A_{y_{*,\bullet}}+g^{y_{*,\bullet}y_{*,\bullet}}A_{y_{*,\bullet}}^2\right)\bigg),\\
	\cancel{a}_{x_{*,\bullet}}(x_{*,\bullet},y_{*,\bullet})&=\cancel{a}_{y_{*,\bullet}}(x_{*,\bullet},y_{*,\bullet})=\kappa\cos^2\theta,\\
	\cancel{b}_{x_{*,\bullet}}(\brho,x_{*,\bullet},y_{*,\bullet})&=\mathrm{i}\Omega^2\cos\theta\left(\p_{x_{*,\bullet}}\!\left(\frac{\rho^2g^{x_{*,\bullet}x_{*,\bullet}}}{\Omega^4\cos\theta}\right)+\p_{y_{*,\bullet}}\!\left(\frac{\rho^2g^{x_{*,\bullet}y_{*,\bullet}}}{\Omega^4\cos\theta}\right)\right)-\mathrm{i}\frac{\rho^2}{\Omega^2}\breve{g}^{x_{*,\bullet}y_{*,\bullet}}x_{*,\bullet}\\
	&\quad+2q\frac{\rho^2}{\Omega^2}(g^{x_{*,\bullet}x_{*,\bullet}}A_{x_{*,\bullet}}+g^{x_{*,\bullet}y_{*,\bullet}}A_{y_{*,\bullet}}),\\
	\cancel{b}_{y_{*,\bullet}}(\brho,x_{*,\bullet},y_{*,\bullet})&=\mathrm{i}\Omega^2\cos\theta\left(\p_{y_{*,\bullet}}\!\left(\frac{\rho^2g^{y_{*,\bullet}y_{*,\bullet}}}{\Omega^4\cos\theta}\right)+\p_{x_{*,\bullet}}\!\left(\frac{\rho^2g^{x_{*,\bullet}y_{*,\bullet}}}{\Omega^4\cos\theta}\right)\right)-\mathrm{i}\frac{\rho^2}{\Omega^2}\breve{g}^{x_{*,\bullet}y_{*,\bullet}}y_{*,\bullet}\\
	&\quad+2q\frac{\rho^2}{\Omega^2}(g^{x_{*,\bullet}y_{*,\bullet}}A_{x_{*,\bullet}}+g^{y_{*,\bullet}y_{*,\bullet}}A_{y_{*,\bullet}}).
	\end{align*}
	Above,
	\begin{align*}
	\breve{g}^{x_*y_*}&:=\frac{g^{x_*y_*}}{x_{*}y_{*}}=\frac{\Omega^2}{\rho^2}\left(\frac{\kappa^2-(\kappa+\kappa(\bullet))\widetilde{\kappa}_\bullet}{\kappa}-\frac{\kappa(\bullet)^2(\mu\widetilde{c}_{\pm}^2+a(1+\lambda)\widetilde{c}_{\pm})}{(1+\lambda)^2}\right)
	\end{align*}
	with $\kappa(N):=\kappa(0)$ and $\kappa(S):=\kappa(\pi)$. We then used that
	\begin{align*}
	-\frac{\rho^2}{\Omega^2}\left(g^{x_{*,\bullet}x_{*,\bullet}}+\breve{g}^{x_{*,\bullet}y_{*,\bullet}}y_{*,\bullet}^2\right)
	&=-\frac{\rho^2}{\Omega^2}\left(g^{y_{*,\bullet}y_{*,\bullet}}+\breve{g}^{x_{*,\bullet}y_{*,\bullet}}x_{*,\bullet}^2\right)
	=\frac{\kappa(\bullet)^2}{\kappa}-\frac{\kappa^2-(\kappa+\kappa(\bullet))\widetilde{\kappa}_\bullet}{\kappa}\sin^2\theta=\kappa\cos^2\theta.
	\end{align*}
	We also used the identity
	\begin{align*}
	-2x_{*,\bullet}y_{*,\bullet}D_{x_{*,\bullet}}D_{y_{*,\bullet}}&=D_{\bvarphi_*}^2-y_{*,\bullet}^2D_{x_{*,\bullet}}^2-x_{*,\bullet}^2D_{y_{*,\bullet}}^2-\mathrm{i}(x_{*,\bullet}D_{x_{*,\bullet}}+y_{*,\bullet}D_{y_{*,\bullet}})
	\end{align*}
	in order to remove any mixed term $D_{x_{*,\bullet}}D_{y_{*,\bullet}}$ from $\cancel{P}_\sigma$, entering thereby the abstract setting of Sect. \ref{General boundary parametrix construction}.
	
	We now check that Assumptions \ref{Assumptions a,b,c...} are fulfilled. To do this, we need to slighlty modify the symbol $\mq_{\sigma,k,-1/2}^{\mathrm{f},1}$ in Sect. \ref{Decay in the fiber} and replace its denominator by $a(\brho)(\xi_{k,-1/2}^2+(k-1/2+\epsilon)^2)+\brho|\eta|_{\cancel{a}}^2$ for any $\epsilon>0$, and then respectively replace $c$, $c_\bullet$ above by $c-a(k-1/2)^2$, $c_\bullet-a(k-1/2)^2$. As we have
	\begin{align*}
	&a(\brho)>0\text{ for all }\brho\in\overline{\mathrm{Supp}(\chi_\pm)}\subset[0,r_{+}-r_{-}),\\
	&\cancel{a}_{x_{*,\bullet}}(x_{*,\bullet},y_{*,\bullet})=\cancel{a}_{y_{*,\bullet}}(x_{*,\bullet},y_{*,\bullet})=\kappa(1-x_{*,\bullet}^2-y_{*,\bullet}^2)>0\text{ for all }(x_{*,\bullet},y_{*,\bullet})\in\overline{\mathrm{Supp}(\phi_\bullet)}\subset\{x_{*,\bullet}^2+y_{*,\bullet}^2<1\},
	\end{align*}
	the new denominator of $\mq_{\sigma,k,-1/2}^{\mathrm{f},1}$ cancels for no $(\xi,\eta)\in\R^3$, and the last two assumptions on the principal symbol follow from the positivity of $a$, $\cancel{a}_{x_{*,\bullet}}$ and $\cancel{a}_{y_{*,\bullet}}$ as noticed in Rmk. \ref{Rmk on Assumptions 1}.

	\subsection{Proof of Thm. \ref{Thm Fredholm index 0 cKG op DSKN}}
	\label{Proof of Thm. Thm Fredholm index 0 cKG op DSKN}
	We use the abstract setting developped in Sect. \ref{General boundary parametrix construction} to prove Thm. \ref{Thm Fredholm index 0 cKG op DSKN}. In Sect. \ref{General boundary parametrix construction}, we replace $\rho$ by $\brho$, $\chi$ by $\chi_\pm$, we define $\CX_\pm^{k,-1/2}$ and $\CH_\pm^{k,-1/2}$ accordingly to the boundary considered (near $\{\brho=0\}$), then we replace $u\in\CH_\pm^{k,-1/2}$ by $\widetilde{\chi}_\pm u$ with $u\in H^k$ -- we take $l=-1/2$ because $\|\widetilde{\chi}_\pm u\|_{L^2((r_{-},r_{+})\times\mathbb{S}^2,\mathrm{d}r\mathrm{d}\omega)}=\|\brho^{1/2}\widetilde{\chi}_\pm u\|_{L^2((0,+\infty),\frac{\mathrm{d}\brho}{\brho})}$.
	
	By Prop. \ref{Prop Local Parametrix}, for all $N,N'\in\N\setminus\{0\}$ such that $N>2(N'-1)$, there exist a right boundary parametrix $\bQ^{N,N'}_{\sigma,k-1,-1/2,\pm}\in\CB(\CH_\pm^{k-1,-1/2},\CX_\pm^{k,-1/2})$ as well as a left boundary parametrix $\widetilde{\bQ}^{N,N'}_{\sigma,k-1,-1/2,\pm}\in\CB(\CH_\pm^{k-1,-1/2},\CX_\pm^{k,-1/2})$ such that:
	\begin{align*}
	\bP_\sigma\bQ^{N,N'}_{\sigma,k-1,-1/2,\pm}&=\chi_\pm\mathds{1}_{\CH_\pm^{k-1,-1/2}}+\bK^{N,N'}_{\sigma,k-1,-1/2,\pm},\\
	\widetilde{\bQ}^{N,N'}_{\sigma,k-1,-1/2,\pm}\widetilde{\bP}_\sigma&=\chi_\pm\mathds{1}_{\CX_\pm^{k,-1/2}}+\widetilde{\bK}^{N,N'}_{\sigma,k-1,-1/2,\pm}.
	\end{align*}
	The remainders $\bK^{N,N'}_{\sigma,k-1,-1/2,\pm}\in\CB(\CH_\pm^{k-1,-1/2})$ and $\widetilde{\bK}^{N,N'}_{\sigma,k-1,-1/2,\pm}\in\CB(\CX_\pm^{k,-1/2})$ are compact. We then glue together the boundaries parametrices and compact remainders:
	\begin{align*}
	\bQ^{N,N'}_{\sigma,k-1}&:=\bQ^{N,N'}_{\sigma,k-1,-1/2,-}+\bQ^{N,N'}_{\sigma,k-1,-1/2,+},\qquad\bK^{N,N'}_{\sigma,k-1}:=\bK_{\sigma,k,-1/2,-}+\bK^{N,N'}_{\sigma,k-1,-1/2,+},\\
	\widetilde{\bQ}^{N,N'}_{\sigma,k-1}&:=\widetilde{\bQ}^{N,N'}_{\sigma,k-1,-1/2,-}+\widetilde{\bQ}^{N,N'}_{\sigma,k-1,-1/2,+},\qquad\widetilde{\bK}^{N,N'}_{\sigma,k-1}:=\widetilde{\bK}^{N,N'}_{\sigma,k-1,-1/2,-}+\widetilde{\bK}^{N,N'}_{\sigma,k-1,-1/2,+}.
	\end{align*}
	Using the facts that $\chi_-+\chi_+=1$, $\chi_\pm\CH_\pm^{k-1,-1/2}\hookrightarrow H^{k-1}$ and $\chi_\pm\CX_\pm^{k,-1/2}\hookrightarrow\CX^k$ (as $\bP_\sigma\chi_\pm=[\bP_\sigma,\chi_\pm]+\chi_\pm\bP_\sigma$), we finally obtain:
	\begin{align}
	\label{Eq Parametrix}
	\bP_\sigma\bQ^{N,N'}_{\sigma,k-1}&=\mathds{1}_{H^k}+\bK^{N,N'}_{\sigma,k-1},\\
	\widetilde{\bQ}^{N,N'}_{\sigma,k-1}\bP_\sigma&=\mathds{1}_{\CX^k}+\widetilde{\bK}^{N,N'}_{\sigma,k-1}.\nonumber
	\end{align}
	It is well known that operators of the form "identity $+$ compact" are index 0 Fredholm operators (\emph{cf.} \cite[Cor. 19.1.8]{Ho}); as a result, $\bP_\sigma$ as well as $\bQ^{N,N'}_{\sigma,k-1}$ and $\widetilde{\bQ}^{N,N'}_{\sigma,k-1}$ are Fredholm operators (\emph{cf.} \cite[Cor. 19.1.9]{Ho}). We show that the index of $\bP_\sigma$ is 0 as in the proof of Prop. \ref{Prop Invertibility Q} (we use there that $a,\cancel{a}_\theta,\cancel{a}_{x_*,\bullet},\cancel{a}_{y_*,\bullet}>0$). Since $\bP_\sigma$ depends analytically on $\sigma\in\C$, Prop. \ref{Prop Local Parametrix} entails that $\bK^{N,N'}_{\sigma,k-1}$ and $\widetilde{\bK}^{N,N'}_{\sigma,k-1}$ are analytic in $\sigma$ provided that $0\notin\bigcup_{j=1}^{N'}\mtp_{\sigma,k-1+j-1,-1/2,\pm}(\R)$; by Rmk. \ref{Rmk Roots Indicial Family},
	\begin{align*}
	\mtp_{\sigma,k-1+j-1,-1/2,\pm}(\xi)&=|(\p_r\mu)(r_{\pm})|(\xi-\mathrm{i}(k-1/2+j-1))\left(\xi-\mathrm{i}(k-1/2+j-1)+\frac{b(0,\theta,\sigma)}{a(0)}\right)
	\end{align*}
	where
	\begin{align*}
	\frac{b(0,\theta,\sigma)}{a(0)}&=-\frac{2}{|(\p_r\mu)(r_{\pm})|}\left((1+\lambda)(r_\pm^2+a^2)\left(\sigma-\frac{qQr_\pm}{r_\pm^2+a^2\cos^2\theta}\right)+ a(1+\lambda)\left(\ell+\frac{qQr_\pm a\sin^2\theta}{r_\pm^2+a^2\cos^2\theta}\right)\right).
	\end{align*}
	The indicial polynomial $\mtp_{\sigma,k-1+j-1,-1/2,\pm}$ can thus cancel at some $\xi\in\R$ if and only if $\sigma$ is such that $-(k-1/2+j-1)-\frac{2(1+\lambda)(r_\pm^2+a^2)\mathrm{Im}(\sigma)}{|(\p_r\mu)(r_{\pm})|}=0$,
	whence the restriction of $\sigma$ to $\C_{k}$ for $N'=1$.
	
	Notice finally that, when $a=0$, $\bP_\sigma$ is elliptic inside $(r_-,r_+)\times\mathbb{S}^2$ (\emph{cf.} \cite[Sect. 2]{BeHa20}) thus fits into the abstract setting of \ref{General boundary parametrix construction} without any harmonical restriction. This completes the proof.
	\begin{remark}
	\label{Rmk on multiplication by rho^2/Omega^2}
		Observe that multiplication by $\rho^2\Omega^{-2}$ in the definition of $\bP_\sigma$ does not change the critical strip as only the ratio $\frac{b(r_\pm,\theta,\sigma)}{a(r_\pm)}$ appears in $\mtp_{\sigma,k-1+j-1,-1/2,\pm}(\xi)$.
	\end{remark}
	\begin{remark}
		The exact structure of the characteristic set of the charged KG operator in the cotangent bundle, especially the existence of radial sets at the cospheres at infinity above the horizons $\{r=r_\pm\}$, is not relevant for our purpose. We refer to \cite[Sect. 2.2]{V} for the definition of the radial sets in the De Sitter-Kerr setting; see also \cite[Sect. 2.4]{V} for propagation estimates therein (as it turns out, one obtains the same estimates at the characteristic set whether one microlocalises the analysis near the radial sets or not).
	\end{remark}
	%

	%
	%
	%
	%

	%
	%
	\section{Numerical scheme}
	\label{Numerical scheme}
	In this last section, we present a method to numerically compute resonances using the right parametrix previously constructed for $\bP_\sigma$ (the angular momentum $\ell\in\Z$ is still fixed here, \emph{cf.} Section \ref{Spectral family}). We will use the notations of Sect. \ref{Proof of Thm. Thm Fredholm index 0 cKG op DSKN}. We place ourselves in the context of Thm. \ref{Thm error estimate for numerical scheme}, that is, we assume that $\bP_{\sigma_0}^{-1}$ exists for some $\sigma_0\in\C_k$ with $k\in\N\setminus\{0\}$ (so that Thm. \ref{Thm Fredholm index 0 cKG op DSKN} combined with analytic Fredholm theory ensure that $\C_k\ni\sigma\mapsto(\bP_\sigma^{-1},\CB(H^{k-1},\CX^k))$ is meromorphic).
	
	
	%
	\subsection{Trace class property}
	\label{Trace class property}
	In this section, we prove the following result:
	\begin{lemma}
		\label{Lem Trace Class Property}
		Let $k,N,N'\in\N\setminus\{0\}$ be such that $N\geq3N'-2$ (so that Prop. \ref{Prop Local Parametrix} applies) and $N-2(N'-1)>3$. Then for all $\sigma\in\C_{k,N'}$ (where $\C_{k,N'}$ is defined in \eqref{Def C_k,l,N'}), the realization $(\bK_{\sigma,k-1}^{N,N'},\CB(H^{k-1}))$ is of trace class.
	\end{lemma}
	Recall that the \emph{singular values} $s_0\geq s_1\geq\ldots\geq s_j\geq\ldots$ of a compact operator $K\in\CB(H_1,H_2)$ acting between two Hilbert spaces $H_1$ and $H_2$ are the (positive) square roots of the eigenvalues of the compact, non-negative self-adjoint operator $K^*K$ (\emph{cf.} \cite[Prop. B.13]{DZ}), and then $K$ is of trace class if and only if
	\begin{align*}
	\|K\|_1&:=\sum_{j=0}^{+\infty}s_j<+\infty.
	\end{align*}
	As for finite rank operators, the trace is then defined by
	\begin{align*}
	K&\longmapsto\sum_{j=0}^{+\infty}\sum_{k=0}^{+\infty}\langle f_k,K e_j\rangle_{H_2}
	\end{align*}
	for any orthonormal Hilbert bases $(e_j)_{j\in\N}\subset H_1$ and $(f_k)_{k\in\N}\subset H_2$ (the trace being a linear thus bounded functional, this invariance property extends to the space of all trace class operators $H_1\to H_2$).
	\begin{proof}[Proof of Lem. \ref{Lem Trace Class Property}]
		The assumption $\sigma\in\C_{k,N'}$ ensures that $(\bK_{\sigma,k-1}^{N,N'},\CB(H^{k-1}))$ is well-defined. Put $\bar{H}^{k-1}:=H^{k-1}([r_{-},r_{+}]\times\mathbb{S}^2,\mathrm{d}r\mathrm{d}\omega)$. Let us define $\bar{\bK}_{\sigma,k-1}^{N,N'}\in\CB(\bar{H}^{k-1})$ by
		\begin{align*}
		(\bar{\bK}_{\sigma,k-1}^{N,N'}u)(r,\omega)&:=\begin{cases}
		(\bK_{\sigma,k-1}^{N,N'}u)(r,\omega)&\text{if $(r,\omega)\in(r_{-},r_{+})\times\mathbb{S}^2$},\\
		0&\text{if $r=r_{\pm}$}.
		\end{cases}
		\end{align*}	
		The boundary values (defined as limits as $r\to r_{\pm}$) are consistent with the fact that the total order in $|r-r_\pm|$ in $\bK_{\sigma,k-1}^{N,N'}$ is $N'>0$. We show that the realization $(\bar{\bK}_{\sigma,k-1}^{N,N'},\CB(\bar{H}^{k-1}))$ is trace class; as we suppose that $N\geq3N'-2$, we have $\bar{\bK}_{\sigma,k-1}^{N,N'}\in\CB(\bar{H}^{k-1},\bar{H}^{k-1+N-2(N'-1)})$. We will then deduce from Lem. \ref{Lem Equivalence for Trace Class} that $(\bK_{\sigma,k-1}^{N,N'},\CB(H^{k-1}))$ is trace class as well.
		
		Let us denote by $\widetilde{\bar{\bK}}_{\sigma,k-1}^{N,N'}$ the realization $(\bar{\bK}_{\sigma,k-1}^{N,N'},\CB(\bar{H}^{k-1},\bar{H}^{k-1+N-2(N'-1)}))$.	Let $(x_1,x_2,x_3):=(r,\theta,\varphi)\in[r_{-},r_{+}]\times[0,\pi]\times[0,2\pi]$, $(\ell_1,\ell_2,\ell_3):=(r_{+}-r_{-},\pi,2\pi)$, $\omega:=(\omega_1,\omega_2,\omega_3):=(\frac{2\pi}{\ell_1},\frac{2\pi}{\ell_2},\frac{2\pi}{\ell_3})$ then define
		\begin{align*}
		\mathrm{e}_{\mathfrak{n},k}(x)&:=\frac{1}{\langle(\mathfrak{n}\omega)^k\rangle}\prod_{j=1}^{3}\frac{\mathrm{e}^{\mathrm{i}\mathfrak{n}_j\omega_jx_j}}{\sqrt{\ell_j}}\qquad\qquad\forall\mathfrak{n}:=(\mathfrak{n}_1,\mathfrak{n}_2,\mathfrak{n}_3)\in\Z^3.
		\end{align*}
		Above, $\langle(\mathfrak{n}\omega)^k\rangle^2:=1+(\mathfrak{n}_1\omega_1)^{2k}+(\mathfrak{n}_2\omega_2)^{2k}+(\mathfrak{n}_3\omega_3)^{2k}$. Following the proof of \cite[Prop. B.21]{DZ}, we show that $\iota:\bar{H}^{k-1+N-2(N'-1)}\hookrightarrow\bar{H}^{k-1}$ is trace class with\footnote{For all $(x_1,x_2,x_3)\in\R^3$ and $a,b\in\N$ such that $b-a\geq1$, we have
				\begin{align*}
				(x_1^2+x_2^2+x_3^2)^{b-a}&\leq3^{b-a-1}(x_1^{2(b-a)}+x_2^{2(b-a)}+x_3^{2(b-a)})
				\end{align*}
				so that
				\begin{align*}
				(x_1^{2a}+x_2^{2a}+x_3^{2a})(x_1^2+x_2^2+x_3^2)^{b-a}&\leq3^{b-a-1}(x_1^{2a}+x_2^{2a}+x_3^{2a})(x_1^{2(b-a)}+x_2^{2(b-a)}+x_3^{2(b-a)})\leq3^{b-a}(x_1^{2b}+x_2^{2b}+x_3^{2b})
				\end{align*}
				by Young's inequality for products. Consequently: 
				\begin{align*}
				\|\iota\|_{1}&\leq1+\sqrt{2}\sum_{\mathfrak{n}\in\Z^3\setminus\{0_{\Z^3}\}}\sqrt{\frac{(\mathfrak{n}_1\omega_1)^{2(k-1)}+(\mathfrak{n}_2\omega_2)^{2(k-1)}+(\mathfrak{n}_3\omega_3)^{2(k-1)}}{(\mathfrak{n}_1\omega_1)^{k-1+N-2(N'-1)}+(\mathfrak{n}_2\omega_2)^{k-1+N-2(N'-1)}+(\mathfrak{n}_3\omega_3)^{k-1+N-2(N'-1)}}}\\
				&\leq1+\sqrt{2\cdot3^{N-2(N'-1)}}\sum_{\mathfrak{n}\in\Z^3\setminus\{0_{\Z^3}\}}\frac{1}{((\mathfrak{n}_1\omega_1)^2+(\mathfrak{n}_2\omega_2)^2+(\mathfrak{n}_3\omega_3)^2)^{(N-2(N'-1))/2}}.
				\end{align*}
			}:
		\begin{align*}
		\|\iota\|_1&=\sum_{\mathfrak{n}\in\Z^3}\frac{\langle(\mathfrak{n}\omega)^{k-1}\rangle}{\langle(\mathfrak{n}\omega)^{k-1+N-2(N'-1)}\rangle}<+\infty.
		\end{align*}
		This comes from the decomposition (\emph{cf.} \cite[eq. (B.3.3)]{DZ})
		\begin{align*}
		\iota&=\sum_{\mathfrak{n}\in\Z^3}s_{\mathfrak{n}}\left\langle\mathrm{e}_{\mathfrak{n},k-1+N-2(N'-1)},\bullet\right\rangle_{\bar{H}^{k-1+N-2(N'-1)}}\mathrm{e}_{\mathfrak{n},k-1}
		\end{align*}
		from which we get:
		\begin{align*}
		s_{\mathfrak{n}}&=\left\|\iota\mathrm{e}_{\mathfrak{n},k-1+N-2(N'-1)}\right\|_{\bar{H}^{k-1}}=\left\|\mathrm{e}_{\mathfrak{n},k-1+N-2(N'-1)}\right\|_{\bar{H}^{k-1}}=\frac{\langle(\mathfrak{n}\omega)^{k-1}\rangle}{\langle(\mathfrak{n}\omega)^{k-1+N-2(N'-1)}\rangle}.
		\end{align*}
		That the series converges follows from the assumption $N-2(N'-1)>3$. It remains to use \cite[eq. (B.4.7)]{DZ} to obtain:
		\begin{align}
		\label{Eq sing val embedding estimate}
		\|\bar{\bK}_{\sigma,k-1}^{N,N'}\|_1&=\|\iota\widetilde{\bar{\bK}}_{\sigma,k-1}^{N,N'}\|_1\leq\|\iota\|_1\|\widetilde{\bar{\bK}}_{\sigma,k-1}^{N,N'}\|_{\CB(\bar{H}^{k-1},\bar{H}^{k-1+N-2(N'-1)})}<+\infty.
		\end{align}
		This completes the proof.
	\end{proof}
	\begin{remark}
		Observe that assuming $N-2(N'-1)>3$ is necessary because of the coordinate singularities at the poles of $\mathbb{S}^2$; if we could use only the Boyer-Lindquist coordinates then $N-2(N'-1)>2$ would have been sufficient for the remainder to be trace class.
		
		Observe also that increasing $N'$ creates new critical strips in $\C$; however each increment of $N'$ adds such a strip below all the others. As a consequence, when one deforms contour integration from the upper half plane up to the upper critical strip below $\R$ to obtain resonances expansion \eqref{Expansion res} (as \emph{e.g.} in \cite[Sect. 3.2]{BoHa}), the exact value of $N'>0$ does not matter.
	\end{remark}
	\begin{remark}
	\label{Rmk number eigenvalues K}
	Let $\varepsilon>0$ and denote by $N_\varepsilon$ the number of eigenvalues $\lambda$ of $\bK_{\sigma,k-1}^{N,N'}$ satisfying $|\lambda|\geq\varepsilon$. Then
	\begin{align*}
	\|\bK_{\sigma,k-1}^{N,N'}\|_1&\geq\sum_{\lambda\text{ eigenvalue}}|\lambda|\geq\sum_{\substack{\lambda\text{ eigenvalue}\\|\lambda|\geq\varepsilon}}|\lambda|\geq\varepsilon N_\varepsilon,
	\end{align*}
	hence
	\begin{align*}
	N_\varepsilon&\leq\frac{1}{\varepsilon}\|\bK_{\sigma,k-1}^{N,N'}\|_1.
	\end{align*}
	\end{remark}
	\begin{lemma}
		\label{Lem Equivalence for Trace Class}
		We use the notations of the proof of Lem. \ref{Lem Trace Class Property}. The realization $(\bar{\bK}_{\sigma,k-1}^{N,N'},\CB(\bar{H}^{k-1}))$ is trace class if and only if $(\bK_{\sigma,k-1}^{N,N'},\CB(H^{k-1}))$ is. Moreover, both the operators have the same singular values.
	\end{lemma}
	\begin{proof}
		Let $(\bar{e}_j)_{j\in\N}$ be an orthonormal Hilbert basis of $\bar{H}^{k-1}$ and set $e_j:=(\bar{e}_j)_{\vert(r_{-},r_{+})\times\mathbb{S}^2}$; then $(e_j)_{j\in\N}$ is clearly an orthonormal Hilbert basis of $H^{k-1}$. The converse is also true: given $e_j\in H^k$ an element of a basis of $H^{k-1}$, we can form a basis for $\bar{H}^{k-1}$ by extending $e_j$ to a neighborhood of $[r_{-},r_{+}]\times\mathbb{S}^2$.
		
		Let us denote by $\langle\bullet,\bullet\rangle$ the standard inner product of $L^2((r_{-},r_{+})\times\mathbb{S}^2,\mathrm{d}r\mathrm{d}\omega)$; as $\{r_{\pm}\}\times\mathbb{S}^2$ are measure zero sets, $\langle\bullet,\bullet\rangle$ coincides with the standard inner product of $L^2([r_{-},r_{+}]\times\mathbb{S}^2,\mathrm{d}r\mathrm{d}\omega)$.
		
		If $\bar{\bK}_{\sigma,k-1}^{N,N'}$ is trace class, then we can write
		\begin{align*}
		\bar{\bK}_{\sigma,k-1}^{N,N'}&=\sum_{j\in\N}s_j\langle\bar{e}_j,\bullet\rangle\bar{f}_j
		\end{align*}
		where
		\begin{align*}
		\sum_{j\in\N}s_j&<+\infty,
		\end{align*}
		$(\bar{e}_j)_{j\in\N}$ is an orthonormal Hilbert basis of $\bar{H}^{k-1}$ composed of eigenvectors of $(\bar{\bK}_{\sigma,k-1}^{N,N'})^*\bar{\bK}_{\sigma,k-1}^{N,N'}$ and $\bar{f}_j=s_j^{-1}\bar{\bK}_{\sigma,k-1}^{N,N'}\bar{e}_j$ when $s_j\neq0$, and $\bar{f}_j=0$ otherwise (\emph{cf.} the proof of \cite[Prop. B.13]{DZ}). Then for all $u\in H^{k-1}$ and all $(r,\omega)\in(r_{-},r_{+})\times\mathbb{S}^2$,
		\begin{align*}
		(\bK_{\sigma,k-1}^{N,N'}u)(r,\omega)&=(\bar{\bK}_{\sigma,k-1}^{N,N'}u)(r,\omega)=\sum_{j\in\N}s_j\langle\bar{e}_j,u\rangle\bar{f}_j(r,\omega)=\sum_{j\in\N}s_j\langle e_j,u\rangle f_j(r,\omega)
		\end{align*}
		and $\bK_{\sigma,k-1}^{N,N'}$ is trace class with singular values $(s_j)_{j\in\N}$. The converse is true for $f_j(r_{\pm},\omega)=0$ for all $\omega\in\mathbb{S}^2$ by construction of $\bar{\bK}_{\sigma,k-1}^{N,N'}$.
	\end{proof}
	\begin{remark}
		It is a matter of fact that Lem. \ref{Lem Trace Class Property} is not necessary for the numerical scheme below: indeed, it is enough to work with the extension $\bar{\bK}_{\sigma,k-1}^{N,N'}$ on $H^{k-1}([r_{-},r_{+}]\times\mathbb{S}^2,\mathrm{d}r\mathrm{d}\omega)$ and observe that $-1$ is an eigenvalue of $\bar{\bK}_{\sigma,k-1}^{N,N'}$ if and only if it is for $\bK_{\sigma,k-1}^{N,N'}$; it then follows that the zeros (above the suitable critical strip) of $\det(\mathds{1}_{H^{k-1}([r_{-},r_{+}]\times\mathbb{S}^2,\mathrm{d}r\mathrm{d}\omega)}+\bar{\bK}_{\sigma,k-1}^{N,N'})$ are exactly the resonances of $\bP_\sigma$. Note that either approach yields the same computations in practice.
	\end{remark}
	\subsection{Numerical approximation of resonances}
	\label{Numerical approximation of resonances}
	In this second section, we gather all the previous results to approximate resonances and estimate the associated numerical error (thus proving Thm. \ref{Thm error estimate for numerical scheme}).
	
	Let $k\in\N\setminus\{0\}$. Under the assumptions of Lem. \ref{Lem Trace Class Property} above, $(\bK_{\sigma,k-1}^{N,N'},\CB(H^{k-1}))$ is of trace class therefore the Fredholm determinant of $\mathds{1}_{H^{k-1}}+\bK_{\sigma,k-1}^{N,N'}$ is well defined and we have (\emph{cf.} \cite[Prop. B.28]{DZ}):
	\begin{align*}
	\mathds{1}_{H^{k-1}}+\bK_{\sigma,k-1}^{N,N'}\text{ is invertible}\qquad&\Longleftrightarrow\qquad D(\sigma):=\det(\mathds{1}_{H^{k-1}}+\bK_{\sigma,k-1}^{N,N'})\neq0.
	\end{align*}
	In virtue of \eqref{Eq Parametrix} and the index 0 property of $(\bP_\sigma,\CX^k)$, we deduce that\footnote{One usually calls resonances the poles of $\bP_\sigma^{-1}$ lying above the upper critical strip; in the present paper, we indifferently use the term resonances for all these poles lying outside the critical strips.}
	\begin{align*}
	\text{$\sigma\in\C_{k,N'}$ is a resonance}\qquad&\Longleftrightarrow\qquad D(\sigma)=0
	\end{align*}
	which proves the first part of Thm. \ref{Thm error estimate for numerical scheme}. The right hand side above is an analytic complex valued function whose zeros can be approximated or counted inside a fixed domain in $\C$.
	
	To see how, consider a family $(\Pi_R)_{R\in\N\setminus\{0\}}$ of projectors onto some $R$ dimensional linear subspaces of $H^{k-1}$, identified to $\C^R$. As $R\to+\infty$, $\|(\mathds{1}_{H^{k-1}}-\Pi_R)\bK_{\sigma,k-1}^{N,N'}\|_{\CB(H^{k-1})}$ by compactness; moreover, the proof of \cite[Prop. B.20]{DZ} shows that (recall that $s_j(A)$ are the singular values of $A$)
	\begin{align}
	\label{Trace inequality}
	\mathfrak{T}_R(\sigma)&:=\|(\mathds{1}_{H^{k-1}}-\Pi_R)\bK_{\sigma,k-1}^{N,N'}\|_{1}\leq\sum_{j\in\N}\min\big\{s_j(\bK_{\sigma,k-1}^{N,N'}),\|(\mathds{1}_{H^{k-1}}-\Pi_R)\bK_{\sigma,k-1}^{N,N'}\|_{\CB(H^{k-1})}\big\}\longrightarrow0
	\end{align}
	and if we set
	\begin{align*}
	D_R(\sigma)&:=\det(\mathds{1}_{\C^R}+\Pi_R\bK_{\sigma,k-1}^{N,N'}),
	\end{align*}
	then (see \emph{e.g.} the remark below \cite[Thm. 5.1]{S}):
	\begin{align}
	\label{Determinant inequality}
	\mathfrak{D}_R(\sigma)&:=|D(\sigma)-D_R(\sigma)|\leq\|(\mathds{1}_{H^{k-1}}-\Pi_R)\bK_{\sigma,k-1}^{N,N'}\|_{1}\mathrm{e}^{1+\max\{\|\bK_{\sigma,k-1}^{N,N'}\|_1,\,\|\Pi_R\bK_{\sigma,k-1}^{N,N'}\|_1\}}.
	\end{align}
	Combining \eqref{Trace inequality} and \eqref{Determinant inequality} together, we come with the following error estimate:
	\begin{align}
	\label{Determinant inequality BIS}
	\mathfrak{D}_R(\sigma)&\leq \mathfrak{T}_R(\sigma)\mathrm{e}^{1+\|\bK_{\sigma,k-1}^{N,N'}\|_1+\mathfrak{T}_R(\sigma)}.
	\end{align}
	We can now work with $D_R(\sigma)$ which is nothing but the determinant of a matrix (so that we can evaluate it with a computer). We finally take advantage of the analiticity in $\sigma\in\C_{k,N'}$: for any positively oriented contour $\Gamma\subset\C_{k,N'}$ which does not intersect any resonances, let $N(\Gamma)$ and $S(\Gamma,n)$ ($n\in\N\setminus\{0\}$) be respectively the number of resonances lying inside $\Gamma$ and the sum of the $n$-th power of these resonances (counted with their multiplicity):
	\begin{align*}
	N(\Gamma)&:=\frac{1}{2\pi\mathrm{i}}\oint_{\Gamma}\frac{\p_\sigma D(\sigma)}{D(\sigma)}\mathrm{d}\sigma,\qquad S(\Gamma,n):=\frac{1}{2\pi\mathrm{i}}\oint_{\Gamma}\frac{\p_\sigma D(\sigma)}{D(\sigma)}\sigma^n\mathrm{d}\sigma.
	\end{align*}
	If $N_R(\Gamma)$ and $S_R(\Gamma,n)$ are defined as above but with $D_R(\sigma)$ insted of $D(\sigma)$, then:
	\begin{align*}
	|N(\Gamma)-N_R(\Gamma)|&\leq\frac{1}{2\pi}\oint_{\Gamma}\frac{1}{|D_R(\sigma)|}\left(\frac{|\p_\sigma D(\sigma)||D(\sigma)-D_R(\sigma)|}{|D(\sigma)|}+|\p_\sigma D_R(\sigma)-\p_\sigma D(\sigma)|\right)|\mathrm{d}\sigma|,\\
	|S(\Gamma,n)-S_R(\Gamma,n)|&\leq\frac{1}{2\pi}\oint_{\Gamma}\frac{1}{|D_R(\sigma)|}\left(\frac{|\p_\sigma D(\sigma)||D(\sigma)-D_R(\sigma)|}{|D(\sigma)|}+|\p_\sigma D_R(\sigma)-\p_\sigma D(\sigma)|\right)|\sigma|^n|\mathrm{d}\sigma|.
	\end{align*}
	Using Cauchy formula for derivatives, we see that
	\begin{align*}
	|\p_\sigma D(\sigma)|&\leq\frac{1}{\delta(\sigma)}\max_{|\sigma'-\sigma|\leq\delta(\sigma)}\{|D(\sigma')|\},\qquad |\p_\sigma D_R(\sigma)-\p_\sigma D(\sigma)|\leq\frac{1}{\delta(\sigma)}\max_{|\sigma'-\sigma|\leq\delta(\sigma)}\{\mathfrak{D}_R(\sigma')\}
	\end{align*}
	where $\delta(\sigma)>0$ is such that $\{\sigma+\delta\mathrm{e}^{\mathrm{i}t}\ \vert\ 0\leq t\leq 2\pi\}\subset\C_{k,N'}$ for all $\sigma\in\Gamma$; on the other hand, \cite[eq. (B.5.11)]{DZ} shows that:
	\begin{align*}
	|D(\sigma)|&\leq\mathrm{e}^{\|\bK_{\sigma,k-1}^{N,N'}\|_1}.
	\end{align*}
	Consequently, we eventually get:
	\begin{align*}
	|N(\Gamma)-N_R(\Gamma)|&\leq\frac{1}{2\pi}\oint_{\Gamma}\frac{1}{|D_R(\sigma)|}\frac{1}{\delta(\sigma)}\max_{|\sigma'-\sigma|\leq\delta(\sigma)}\left\{\frac{\mathrm{e}^{\|\bK_{\sigma',k-1}^{N,N'}\|_1}}{|D(\sigma)|}\mathfrak{D}_R(\sigma)+\mathfrak{D}_R(\sigma')\right\}|\mathrm{d}\sigma|,\\
	|S(\Gamma,n)-S_R(\Gamma,n)|&\leq\frac{1}{2\pi}\oint_{\Gamma}\frac{1}{|D_R(\sigma)|}\frac{1}{\delta(\sigma)}\max_{|\sigma'-\sigma|\leq\delta(\sigma)}\left\{\frac{\mathrm{e}^{\|\bK_{\sigma',k-1}^{N,N'}\|_1}}{|D(\sigma)|}\mathfrak{D}_R(\sigma)+\mathfrak{D}_R(\sigma')\right\}|\sigma|^n|\mathrm{d}\sigma|.
	\end{align*}
	The rougher estimates of Thm. \ref{Thm error estimate for numerical scheme} are obtained as follows: we pick any contour $\Gamma'\subset\C_{k,N'}$ containing $\Gamma$, denote by $\delta'>0$ the distance between $\Gamma$ and $\Gamma'$, then we take the maximum and minimum respectively of the numerator and denominator of the integrand over $\Gamma$ (for the $\sigma$ variable) or $\Gamma'$ (for the $\sigma'$ variable). The constants of the latter estimate are thus:
	\begin{align}
	\label{Eq Constants Numerics}
	C_R(\Gamma)&=\frac{\mathrm{lg}(\Gamma)}{2\pi\delta'\min_{\sigma\in\Gamma}\{|D_R(\sigma)|\}}\left(\frac{\max_{|\sigma'-\sigma|\leq\delta'}\{\mathrm{e}^{\|\bK_{\sigma',k-1}^{N,N'}\|_1}\}}{\min_{\sigma\in\Gamma}\{|D_R(\sigma)|\}-\max_{\sigma\in\Gamma}\{\mathfrak{D}_R(\sigma)\}}\max_{\sigma\in\Gamma}\{\mathfrak{D}_R(\sigma)\}+\max_{|\sigma'-\sigma|\leq\delta'}\{\mathfrak{D}_R(\sigma')\}\right),\\\nonumber
	\widetilde{C}_R(\Gamma,n)&=C_R(\Gamma)|\Gamma|^n.
	\end{align}
	Above $\mathrm{lg}(\Gamma)$ denotes the length of $\Gamma$ and $|\Gamma|:=\max\{|\sigma|\ \vert\ \sigma\in\Gamma\}$. Notice that $C_R(\Gamma),\widetilde{C}_R(\Gamma,n)=\CO(\|(\mathds{1}_{H^{k-1}}-\Pi_R)\bK_{\sigma,k-1}^{N,N'}\|_{\CB(H^{k-1})})$ as we can deduce from \eqref{Trace inequality} and \eqref{Determinant inequality BIS}.
	
	We briefly explain how to evaluate $C_R(\Gamma)$ and $\widetilde{C}_R(\Gamma,n)$: the minimum over $\Gamma$ of $|D_R|$ which can be numerically evaluated, as well as maxima over $\Gamma$ and $\Gamma'$ of $\mathfrak{D}_R$; by \eqref{Determinant inequality BIS}, the latter can be bounded if we have an estimate on the singular values $s_j(\bK_{\sigma,k-1}^{N,N'})$ and $\|(\mathds{1}_{H^{k-1}}-\Pi_R)\bK_{\sigma,k-1}^{N,N'}\|_{\CB(H^{k-1})}$. As the singular values can themselves be estimated using \eqref{Eq sing val embedding estimate} (replacing $\|\iota\|_{1}$ by $s_j(\iota)$) and a uniform bound of\footnote{That both these norms coincide follows from the fact that after $k-1$ derivaitves, we consider the quantization of a symbol containing at least one power of $|r-r_\pm|$ near $\{r=r_\pm\}$.} $\|\widetilde{\bar{\bK}}_{\sigma,k-1}^{N,N'}\|_{\CB(\bar{H}^{k-1},\bar{H}^{k-1+N-2(N'-1)})}=\|\bK_{\sigma,k-1}^{N,N'}\|_{\CB(H^{k-1},H^{k-1+N-2(N'-1)})}$ over $\Gamma$ and $\Gamma'$, we are done.

	%
	%
	%
	%

	\section{Appendix: Explicit computation of $\bP_\sigma$ in the $(t_*,r,x_{*,\bullet},y_{*,\bullet})$ coordinates}
	\label{App computation P_sigma}
	We give a detailed derivation of the spectral family near the poles of $\mathbb{S}^2$ (\emph{cf.} Sect. \ref{Spectral family}). We drop the index $*,\bullet$ to lighten notations.
	
	Using $\p_{\bvarphi}=\frac{1+\lambda}{\kappa(\bullet)}\p_\varphi$, $x\p_y-y\p_x=\p_{\bvarphi}$ as well as $g^{tx}=-yg^{t\varphi}\frac{\kappa(\bullet)}{1+\lambda}$, $g^{ty}=xg^{t\varphi}\frac{\kappa(\bullet)}{1+\lambda}$, $g^{rx}=-yg^{r\varphi}\frac{\kappa(\bullet)}{1+\lambda}$ and $g^{ry}=xg^{r\varphi}\frac{\kappa(\bullet)}{1+\lambda}$, we compute:
	\begin{align*}
	\nabla_\gamma\nabla^\gamma&=g^{tt}\p_t^2+g^{tr}\p_t\p_r+g^{tx}\p_t\p_x+g^{ty}\p_t\p_y\\
	&\quad+\frac{\Omega^4}{\rho^2}\p_r\!\left(\frac{\rho^2}{\Omega^4}\right)(g^{tr}\p_t+g^{rx}\p_x+g^{ry}\p_y)-\frac{\Omega^4}{\rho^2}\p_r\frac{\mu}{\Omega^2}\p_r+\p_r(g^{tr}\p_t+g^{rx}\p_x+g^{ry}\p_y)\\
	&\quad+\frac{\Omega^4\cos\theta}{\rho^2}\p_x\!\left(\frac{\rho^2}{\Omega^4\cos\theta}\right)(g^{tx}\p_t+g^{rx}\p_r+g^{xx}\p_x+g^{xy}\p_y)+\p_x(g^{tx}\p_t+g^{rx}\p_r+g^{xx}\p_x+g^{xy}\p_y)\\
	&\quad+\frac{\Omega^4\cos\theta}{\rho^2}\p_y\!\left(\frac{\rho^2}{\Omega^4\cos\theta}\right)(g^{ty}\p_t+g^{ry}\p_r+g^{xy}\p_x+g^{yy}\p_y)+\p_y(g^{ty}\p_t+g^{ry}\p_r+g^{xy}\p_x+g^{yy}\p_y)\\
	&=g^{tt}\p_t^2+2g^{t\varphi}\p_t\p_\varphi-\frac{\Omega^4}{\rho^2}\p_r\frac{\mu}{\Omega^2}\p_r+\p_x(g^{xx}\p_x+g^{xy}\p_y)+\p_y(g^{xy}\p_x+g^{yy}\p_y)\\
	&\quad+[g^{tr}\p_t+g^{t\varphi}\p_\varphi,\p_r]_{+}+\frac{\Omega^4}{\rho^2}\p_r\!\left(\frac{\rho^2}{\Omega^4}\right)(g^{tr}\p_t+g^{r\varphi}\p_\varphi)\\
	&\quad+\frac{\Omega^4\cos\theta}{\rho^2}\p_x\!\left(\frac{\rho^2}{\Omega^4\cos\theta}\right)(g^{xx}\p_x+g^{xy}\p_y)+\frac{\Omega^4\cos\theta}{\rho^2}\p_y\!\left(\frac{\rho^2}{\Omega^4\cos\theta}\right)(g^{xy}\p_x+g^{yy}\p_y).
	\end{align*}
	Using next $xA_y-yA_x=\frac{1+\lambda}{\kappa(\bullet)}A_\varphi$, $g^{tx}A_x+g^{ty}A_y=g^{t\varphi}A_\varphi$, $g^{rx}A_x+g^{ry}A_y=g^{r\varphi}A_\varphi$, we compute:
	\begin{align*}
	\nabla_\gamma A^\gamma+A_\gamma\nabla^\gamma&=[\p_\gamma,g^{\gamma\delta}A_\delta]_{+}+\frac{1}{\sqrt{\det g}}\p_\gamma\!\left(\sqrt{\det g}\right)g^{\gamma\delta}A_\delta\\
	&=[\p_t,g^{tt}A_t+g^{tx}A_x+g^{ty}A_y]_{+}+[\p_r,g^{tr}A_t+g^{rx}A_x+g^{ry}A_y]_{+}\\
	&\quad+[\p_x,g^{tx}A_t+g^{xx}A_x+g^{xy}A_y]_{+}+[\p_y,g^{ty}A_t+g^{xy}A_x+g^{yy}A_y]_{+}\\
	&\quad+\frac{\Omega^4}{\rho^2}\p_r\!\left(\frac{\rho^2}{\Omega^4}\right)(g^{tr}A_t+g^{rx}A_x+g^{ry}A_y)+\frac{\Omega^4\cos\theta}{\rho^2}\p_x\!\left(\frac{\rho^2}{\Omega^4\cos\theta}\right)(g^{tx}A_t+g^{xx}A_x+g^{xy}A_y)\\
	&\quad+\frac{\Omega^4\cos\theta}{\rho^2}\p_y\!\left(\frac{\rho^2}{\Omega^4\cos\theta}\right)(g^{ty}A_t+g^{xy}A_x+g^{yy}A_y)\\
	&=2g^{tt}A_t\p_t+2g^{t\varphi}(A_t\p_\varphi+A_\varphi\p_t)+[g^{tr}A_t+g^{r\varphi}A_\varphi,\p_r]_{+}\\
	&\quad+[g^{xx}A_x+g^{xy}A_y,\p_x]_{+}+[g^{xy}A_x+g^{yy}A_y,\p_y]_{+}\\
	&\quad+\frac{\Omega^4}{\rho^2}\p_r\!\left(\frac{\rho^2}{\Omega^4}\right)(g^{tr}A_t+g^{r\varphi}A_\varphi)+\frac{\Omega^4\cos\theta}{\rho^2}\p_x\!\left(\frac{\rho^2}{\Omega^4\cos\theta}\right)(g^{xx}A_x+g^{xy}A_y)\\
	&\quad+\frac{\Omega^4\cos\theta}{\rho^2}\p_y\!\left(\frac{\rho^2}{\Omega^4\cos\theta}\right)(g^{xy}A_x+g^{yy}A_y).
	\end{align*}
	Finally:
	\begin{align*}
	A_\gamma A^\gamma&=(g^{tt}A_t+g^{tx}A_x+g^{ty}A_y)A_t+(g^{tx}A_t+g^{xx}A_x+g^{xy}A_y)A_x+(g^{ty}A_t+g^{xy}A_x+g^{yy}A_y)A_y\\
	&=g^{tt}A_t^2+2g^{t\varphi}A_tA_\varphi+(g^{xx}A_x+g^{xy}A_y)A_x+(g^{xy}A_x+g^{yy}A_y)A_y.
	\end{align*}
	Gathering everything together, we obtain:
	\begin{align*}
	(\nabla_\gamma-\mathrm{i}qA_\gamma)(\nabla^\gamma-\mathrm{i}qA^\gamma)
	%
	%
	%
	%
	%
	%
	%
	%
	%
	%
	%
	&=(\p_x-\mathrm{i}qA_x)(g^{xx}(\p_x-\mathrm{i}qA_x)+g^{xy}(\p_y-\mathrm{i}qA_y))\\
	&\quad+(\p_y-\mathrm{i}qA_y)(g^{xy}(\p_x-\mathrm{i}qA_x)+g^{yy}(\p_y-\mathrm{i}qA_y))\\
	&\quad-\frac{\Omega^4}{\rho^2}\p_r\frac{\mu}{\Omega^2}\p_r+[g^{tr}(\p_t-\mathrm{i}qA_t)+g^{t\varphi}(\p_\varphi-\mathrm{i}A_\varphi),\p_r]_{+}\\
	&\quad+\frac{\Omega^4\cos\theta}{\rho^2}\p_x\!\left(\frac{\rho^2}{\Omega^4\cos\theta}\right)(g^{xx}(\p_x-\mathrm{i}qA_x)+g^{xy}(\p_y-\mathrm{i}qA_y))\\
	&\quad+\frac{\Omega^4\cos\theta}{\rho^2}\p_y\!\left(\frac{\rho^2}{\Omega^4\cos\theta}\right)(g^{xy}(\p_x-\mathrm{i}qA_x)+g^{yy}(\p_y-\mathrm{i}qA_y))\\
	&\quad+\frac{\Omega^4}{\rho^2}\p_r\!\left(\frac{\rho^2}{\Omega^4}\right)(g^{tr}(\p_t-\mathrm{i}qA_t)+g^{r\varphi}(\p_\varphi-\mathrm{i}qA_\varphi))\\
	&\quad+g^{tt}(\p_t-\mathrm{i}qA_t)^2+2g^{t\varphi}(\p_t-\mathrm{i}qA_t)(\p_\varphi-\mathrm{i}qA_\varphi).
	\end{align*}
	%

	%
	%
	%
	%
	
	%

\begin{thebibliography}{50}
		%
		\bibitem[AF]{AF}
		R. A. Adams, J. J. Fournier, \emph{Sobolev Ppaces}, Elsevier, 2nd edition, 320 pp, 2003.
		%
		\bibitem[B-th]{Bthesis}
		N. Besset, \emph{L’\'{e}quation charg\'{e}e de Klein-Gordon en m\'{e}trique de De Sitter-Reissner-Nordstr\"{o}m}, PhD thesis, https://tel.archives-ouvertes.fr/tel-02944782, 2019.
		%
		\bibitem[B1]{B1}
		N. Besset, \emph{Decay of the Local Energy for the Charged KG Equation in the Exterior De Sitter-Reissner-Nordström Spacetime}, Annales Henri Poincar\'{e}, \textbf{21}(8), 2433-2484, 2020.
		%
		\bibitem[B2]{B2}
		N. Besset, \emph{Scattering Theory for the Charged KG Equation in the Exterior De Sitter-Reissner-Nordström Spacetime}, The Journal of Geometric Analysis, \textbf{31}(11), 10521-10585, 2020.
		%
		\bibitem[B3]{B3}
		N. Besset, \emph{Parametrix constructions and low frequency resonances for the charged Klein-Gordon equation on the De Sitter-Reissner-Nordstr\"{o}m metric}, arXiv:2203.06930v1, 2022.
		%
		\bibitem[BeHa]{BeHa20}
		N. Besset, D. H\"{a}fner, \emph{Existence of exponentially growing finite energy solutions for the Charged KG equation on the De Sitter-Kerr-Newman metric}, Journal of Hyperbolic Differential Equations, \textbf{18}(2), pp. 293-310, 2021.
		%
		\bibitem[BoHa]{BoHa}
		J.-F. Bony, D. H\"afner, \emph{Decay and non-decay of the local energy for the wave equation in the De Sitter-Schwarzschild metric}, Comm. Math. Phys. {\bf 282}(3), 697-719, 2008.
		%
		\bibitem[CCDHJ]{CCDHJ}
		V. Cardoso, J. Costa, K. Destounis, P. Hintz, A. Jansen, \emph{Strong cosmic censorship in charged black-hole spacetimes: still subtle}, Physical Review D, \textbf{98}(10), 104007, 2018.
		%
		%
		%
		\bibitem[D]{D}
		S. Dyatlov, \emph{Asymptotic distribution of quasi-normal modes for Kerr–de Sitter black holes}, Annales Henri Poincar\'{e} \textbf{13}, 1101–1166, 2012.
		%
		\bibitem[DZ]{DZ}
		S. Dyatlov, M. Zworski, \emph{Mathematical Theory of Scattering Resonances}, AMS Graduate Studies in Mathematics \textbf{200}, 2019.
		%
		\bibitem[GGH]{GGH}
		V. Georgescu, C. G\'{e}rard, D. H\"{a}fner, \emph{Asymptotic completeness for superradiant KG equations and applications to the De Sitter-Kerr metric}, J. Eur. Math. Soc. \textbf{19}, 2171-2244, 2017.
		%
		\bibitem[HHV]{HHV}
		D. H\"{a}fner, P. Hintz, A. Vasy, \emph{Linear stability of slowly rotating Kerr black holes}, Inventiones Math.
		\textbf{223}, 1227-1406, 2021.
		%
		\bibitem[HN]{HN}
		D. H\"{a}fner, J.-P. Nicolas, \emph{Scattering of massless Dirac fields by a Kerr black hole}, Rev. in Math. Phys. \textbf{16}(1), 29-123, 2004.
		%
		\bibitem[H1]{H1}
		P. Hintz, \emph{Nonlinear stability of the Kerr-Newman-de Sitter family of charged black holes}, Annals of PDE {\bf 4}, 11, 2018.
		%
		\bibitem[H2]{H2}
		P. Hintz, \emph{A sharp version of Price's law for wave decay on asymptotically flat spacetimes}, Comm. Math. Phys., 389:491–542, 2022.
		%
		\bibitem[H3]{H3}
		P. Hintz, \emph{Mode stability and shallow quasinormal modes of Kerr-de Sitter black holes away from extremality}, preprint: arXiv:2112.14431v1, 2021.
		%
		\bibitem[HV]{HV}
		P. Hintz, A. Vasy, \emph{The global non-linear stability of the Kerr–de Sitter family of black holes}, Acta mathematica \textbf{220}, pp 1–206, 2018.
		%
		\bibitem[HX1]{HX1}
		P. Hintz, Y. Xie, \emph{Quasinormal modes and dual resonant states on de Sitter space}, Physical Review D, \textbf{104}(6), 064037, 2021.
		%
		\bibitem[HX2]{HX2}
		P. Hintz, Y. Xie, \emph{Quasinormal modes of small Schwarzschild--de Sitter black holes}, Journal of Mathematical Physics, \textbf{63}, 011509, 2022.
		%
		\bibitem[H\"{o}]{Ho}
		L. H\"{o}rmander, \emph{Analysis of linear partial differential operators III. Pseudo-differential operators}, Classics in Mathematics, Springer, Berlin, 1994.
		%
		\bibitem[MM]{MM}
		R. R. Mazzeo, R. B. Melrose, \emph{Meromorphic extension of the resolvent on complete space with asymptotically constant negative curvature}, Journal of Functional Analysis \textbf{75}, 260-310, 1986.
		%
		\bibitem[Mo]{Mo17}
		M. Mokdad, \emph{Reissner-Nordström-de Sitter Manifold : Photon Sphere and Maximal Analytic Extension}, Classical and Quantum Gravity \textbf{34}(17), 2017.
		%
		\bibitem[PC]{PC}
		Planck Collaboration, \emph{Planck 2018 results. VI. Cosmological parameters}, arXiv:1807.06209.
		%
		\bibitem[PG]{PG}
		J. Podolsk\`{y}, J. B. Griffiths, \emph{Accelerating Kerr-Newman black holes in (anti-)de Sitter space-time}, Phys. Rev. D, \textbf{73}(4):044018, 2006.
		%
		\bibitem[SZ]{SZ}
		A. S\'a Barreto, M. Zworski, \emph{Distribution of resonances for spherical black holes}, Math. Res. Lett. \textbf{4}(1), 103–121, 1997.
		%
		%
		\bibitem[S]{S}
		B. Simon, \emph{Trace Ideals and Their Applications}, Mathematical Surveys and Monographs, American Mathematical Society, vol. \textbf{120}, 2005.
		%
		\bibitem[V]{V}
		A. Vasy, \emph{Microlocal analysis of asymptotically hyperbolic and Kerr-de Sitter spaces (with an appendix by Semyon Dyatlov)}, Inventiones Math {\bf 194}, 381-513, 2013.
		%
		\bibitem[Z]{Z}
		M. Zworski, \emph{Semiclassical analysis}, Graduate Studies in Mathematics Vol. 138, 431 pp, 2012.
	\end{thebibliography}
\end{document}